\documentclass[11pt,a4paper]{amsart}
\usepackage[margin=2.2cm]{geometry}
\usepackage{amsmath, amssymb, amsthm}
\usepackage{enumerate}
\usepackage{mathrsfs}
\usepackage{dsfont} 
\usepackage{bbm}
\usepackage{cases}
\usepackage{xcolor}
\usepackage[colorlinks=true, linkcolor = blue, citecolor = blue]{hyperref}
\usepackage[numbers]{natbib}
\numberwithin{equation}{section}

\newcommand{\bE}{\mathbb E}
\newcommand{\bF}{\mathbb F}

\newcommand{\bH}{\mathbb H}

\newcommand{\bL}{\mathbb L}

\newcommand{\bP}{\mathbb P}

\newcommand{\bR}{\mathbb R}
\newcommand{\bS}{\mathbb S}

\newcommand{\cF}{\mathcal F}

\newcommand{\cU}{\mathcal U}

\renewcommand{\epsilon}{\varepsilon}

\begin{document}
\theoremstyle{plain}

\newtheorem{theorem}{Theorem}[section]
\newtheorem{lemma}[theorem]{Lemma}
\newtheorem{example}[theorem]{Example}
\newtheorem{proposition}[theorem]{Proposition}
\newtheorem{corollary}[theorem]{Corollary}
\newtheorem{definition}[theorem]{Definition}
\newtheorem{assumption}[theorem]{Assumption}
\newtheorem{condition}[theorem]{Condition}
\theoremstyle{definition}
\newtheorem{remark}[theorem]{Remark}
\newtheorem{SA}[theorem]{Standing Assumption}

\newcommand{\red}[1]{\textcolor{red}{#1}}
\newcommand{\blue}[1]{\textcolor{blue}{#1}}
\newcommand{\gray}[1]{\textcolor{gray}{#1}}
\newcommand{\violet}[1]{\textcolor{black}{#1}}
\newcommand{\white}[1]{\textcolor{white}{#1}}
\newcommand{\violetNew}[1]{\textcolor{violet}{#1}}

\makeatletter
\newcommand{\mylabel}[2]{#2\def\@currentlabel{#2}\label{#1}}
\makeatother

\makeatletter
\@namedef{subjclassname@2020}{%
	\textup{2020} Mathematics Subject Classification}
\makeatother

\title[]{Convergence of a Deep BSDE solver with jumps}
\author[A. Gnoatto]{Alessandro Gnoatto}
\address{Via Cantarene 24, 37129 Verona, Italy.}
\email{alessandro.gnoatto@univr.it}
\author[K. Oberpriller]{Katharina Oberpriller}
\address{Theresienstr. 39, 80333 M\"unchen, Germany.}
\email{oberpriller@lmu.de}
\author[A. Picarelli]{Athena Picarelli}
\address{Via Cantarene 24, 37129 Verona, Italy.}
\email{athena.picarelli@univr.it}


\subjclass[2020]{}

\date{\today}

\maketitle

\begin{abstract}
We study the error arising in the numerical approximation of FBSDEs and related PIDEs by means of a deep learning-based method. Our results focus on decoupled FBSDEs with jumps and extend the seminal work of \cite{Han_Long_2020} analyzing the numerical error of the deep BSDE solver proposed in \cite{ehanjen17}. We provide a priori and a posteriori error estimates for the finite and infinite activity case.
\end{abstract}

\section{Introduction}


In this paper we provide a priori and a posteriori error estimates for the numerical solution of a non-linear partial integro-differential equation (PIDE) by means of a deep learning-based algorithm. The PIDE we would like to numerically solve by finding the function $u:[0,T] \times \mathbb{R}^d \to \mathbb{R}$ is of the form
\begin{align}
\begin{aligned}
 -u_t(t, x)-\mathscr{L} u(t, x) -f\left(t, x, u(t, x), D_x u(t, x) \sigma(x), \mathscr{J} u(t, x)\right)&=0, \quad \quad(t, x) \in[0, T) \times \mathbb{R}^d, \label{eq:PIDEintro1}\\
 u(T, x)&=g(x), \quad x \in \mathbb{R}^d,
 \end{aligned}
\end{align}
where $u_t$ denotes the partial derivative with respect to time, $D_x,D_x^2$ denote the gradient and the Hessian with respect to the space variable and
\begin{align}
\mathscr{L} u(t, x)= & \left\langle {b}(x), D_x u(t, x)\right\rangle+\frac{1}{2}\left\langle\sigma(x) D_x^2 u(t, x), \sigma(x)\right\rangle \nonumber  \\
& +\int_{\mathbb{R}^d}\left(u(t, x+\Gamma(x, z))-u(t, x)-\left\langle\Gamma(x, z), D_x u(t, x)\right\rangle\right) \nu(d z),  \label{eq:PIDEintro3}\\
\mathscr{J} u(t, x)= & \int_{\mathbb{R}^d}(u(t, x+\Gamma(x, z))-u(t, x)) \nu(d z) . \nonumber 
\end{align}
Our analysis covers PIDEs associated to jump processes of finite and infinite activity and also of infinite variation. Such type of PIDEs arise naturally in finance in the context of option pricing problems when the underlying security/state variable follows a process with discontinuous sample paths such as CGMY \cite{carr2002fine}, Merton \cite{Merton1976} or affine jump diffusions \cite{duffie2000transform} among many.

When the dimension of the state space is low (typically up to $d=3$) the standard approach for the numerical solution of the PIDE above involves an application of the finite difference or finite element method, see \cite{Cont2003}, \cite{Hilber2013}. However, when the dimension is large, such approaches suffer from an exponential growth of their complexity known as the curse of dimension. This has motivated the recent increasing popularity of approaches employing deep learning techniques, mostly based on the Feynman-Ka\'c representation formula, which links PIDEs to system of forward backward stochastic differential equations (FBSDEs).

More precisely, for $t \in [0,T]$ and $x \in \bR^d,$ given the FBSDE
\begin{align} \label{eq:SDEIntro}
	X_s^{t,x}=x + \int_t^s b(X_{r-}^{t,x}) \textnormal{d}r + \int_t^s \sigma(X_{r-}^{t,x} )^{\top} \textnormal{d}W_r + \int_t^s \int_{\mathbb{R}^d} \Gamma (X_{r-}^{t,x},z) \widetilde{N}(\textnormal{d}r,\textnormal{d}z), \quad {s \in [t,T],}\end{align}
and 
\begin{align} \label{eq:BSDEIntro}
Y_{  {s}}^{t,x}&=g(X_T^{t,x})+\int_{  {s}}^T f\left(r,X_{r-}^{t,x}, Y_{r-}^{t,x}, Z_r^{t,x}, \int_{\mathbb{R}^d} U_r^{t,x}(z) \nu (\textnormal{d}z) \right)\textnormal{d}r - \int_{  {s}}^T (Z_r^{t,x})^{\top}\textnormal{d}W_r  \nonumber \\
	& \quad - \int_{  {s}}^T \int_{\mathbb{R}^d} U_r^{t,x}(z) \widetilde{N} (\textnormal{d}r,\textnormal{d}z), \quad {s \in [t,T].}
\end{align}
the Feynman-Ka\'c representation formula states that $u(s,x)=Y^{t,x}_s$, meaning that the PIDE can be solved by considering the associated FBSDE. The literature on FBSDEs is large. Apart from the pioneering work of \cite{pardoux1990adapted} in the diffusive case, we mention \cite{barles_buckdahn_pardoux} that covers the jump diffusion case which is relevant for the present work. Further information can be found in \cite{delong} or \cite{Zhang2017}. Artificial Neural Networks (ANNs) are then typically employed to parametrize quantities related to the solution of the FBSDE.

The seminal work in this stream of literature is \cite{ehanjen17} coupled with the convergence analysis of \cite{Han_Long_2020}. In these works, formulated in a setting without jumps, the BSDE is discretized forward in time using an Euler discretization and then a family of ANNs, one ANN for each time step, is used to parametrize the control process of the BSDE. The parameters of the ANN are chosen as to minimize the squared expected distance between the terminal condition of the FBSDE and the Euler discretization of the BSDE at terminal time. Alternatively, \cite{hpw2021} base their approach on a backward procedure based on the dynamic programming principle. We also mention \cite{bbcjn2021} based on the idea of operator splitting and \cite{SIRIGNANO2018}. Extensions of the approaches above to the case with jumps have been proposed. The work \cite{vk2021} generalized \cite{bbcjn2021}, but considers only jumps of finite activity. A first generalization of \cite{hpw2021} is studied in \cite{castro2021}, whose result have been recently refined in \cite{jakobsen_mazid_2024}.
Of particular relevance for the present paper is \cite{gnoatto_patacc_picarelli} which extends \cite{ehanjen17} to the case of jumps. Further algorithms are studied and compared in \cite{abdw24}.

Concerning the convergence analysis of such algorithms, our main reference is the aforementioned contribution \cite{Han_Long_2020} that considers weakly coupled FBSDEs without jumps and provides the convergence analysis for the approach of \cite{ehanjen17}. A first extension involving jumps is proposed in \cite{wang_wang_li_gao_fu}: weakly coupled FBSDEs are treated, only jumps of finite activity are considered and the analysis is limited to a posteriori estimates. Our analysis of a posteriori estimates (which excludes weak coupling) is in line with that of \cite{wang_wang_li_gao_fu} but we also consider infinite activity jumps and, more importantly, we obtain a priori estimates which are missing in \cite{wang_wang_li_gao_fu}. Further contribution on the error analysis of deep learning based solvers for FBSDEs/PIDEs include \cite{castro2021} and \cite{jakobsen_mazid_2024} for the deep backward dynamic programming principle of \cite{hpw2021} and \cite{vk2022} which generalized \cite{bbcjn2021}.

This paper is organized as follows: in Section \ref{sec:settingPrelim} we present the problem and illustrate the generalization of the algorithm of \cite{ehanjen17} that constitutes our object of analysis. In Section \ref{sec:aPosterioriEstimates} we provide a self contained derivation of a posteriori estimates that is in line with the results of \cite{wang_wang_li_gao_fu} but is extended to consider the infinite activity case in Section \ref{sec:InfiniteActivityCase}. Section \ref{sec:aPrioriEstimates} provides a priori estimates both for the finite and infinite activity case.

\section{Setting and Preliminaries}\label{sec:settingPrelim}

Given a fixed time horizon $T <\infty$, let $(\Omega, \cF, \bF, \bP)$ be a filtered probability space with the filtration $\bF=(\cF_t)_{t \in [0,T]}$ satisfying the usual conditions. We assume that the filtered probability space supports an $\bR^d$-valued standard Brownian motion $W=(W_t)_{t \in [0,T]}$ and a Poisson random measure $N$ with associated L\'evy measure $\nu$, such that 
\begin{align} \label{eq:AssumpLevyMeasure}
\nu(\lbrace 0 \rbrace)=0, \quad \int_{\bR^d} (1 \wedge \vert z \vert^2) \nu( \textnormal{d}z)<\infty, \quad \text{and} \quad \int_{\vert z \vert \geq 1 } \vert z\vert^2 \nu(\textnormal{d} z)<\infty.
\end{align}
We denote the compensated random measure by
\begin{align} \label{eq:CompensatedRandomMeasure}
	\widetilde{N}(\textnormal{d}t,\textnormal{d}z)=N(\textnormal{d}t,\textnormal{d}z)-\nu(\textnormal{d}z)\textnormal{d}t.
\end{align}
In order to define the solution of FBSDEs we introduce the following spaces.
\begin{itemize}
	\item $\bL_{\cF}^2(\bR^d)$ is the space of all $\cF$-measurable random variables $X: \Omega \to \bR^d$ satisfying
	\begin{equation*}
	\| X \|^2:=\bE \left[ \big \vert X \big \vert^2\right]< \infty;
	\end{equation*}
	\item $\bH_{[0,T]}^2(\bR^d)$ is the space of all predictable processes $\phi:\Omega\times [0,T] \to \bR^d$ satisfying
	\begin{equation*}
		\| \phi \|^2_{\bH^2}:= \bE \left[ \int_0^T \vert \phi_t \vert^2\textnormal{d}t\right]< \infty;	
	\end{equation*}
	\item $\bH_{[0,T],N}^2(\bR)$ is the space of all predictable processes $\phi:\Omega\times [0,T]\times \bR^d \to \bR$ satisfying
	\begin{equation*}
		\| \phi \|^2_{\bH^2_N}:= \bE \left[ \int_0^T \int_{\bR^d}\vert \phi_t(z) \vert^2 \nu(\textnormal{d}z)\textnormal{d}t\right]< \infty;
	\end{equation*}
	\item $\bS_{[0,T]}^2(\bR^d)$ is the space of all $\bF$-adapted c\`{a}dl\`{a}g processes $\phi:\Omega \times [0,T] \to \bR^d$ satisfying
	\begin{equation*}
		\|\phi \|_{\bS^2}^2:=\bE \left[ \sup_{t \in [0,T]} \big \vert  \phi_t \big \vert ^2\right]< \infty.
	\end{equation*}
\end{itemize}
For $t \in [0,T]$ and $x \in \bR^d,$ we consider the following FBSDE defined by
\begin{align} \label{eq:SDE}
	X_s^{t,x}=x + \int_t^s b(X_{r-}^{t,x}) \textnormal{d}r + \int_t^s \sigma(X_{r-}^{t,x} )^{\top} \textnormal{d}W_r + \int_t^s \int_{\mathbb{R}^d} \Gamma (X_{r-}^{t,x},z) \widetilde{N}(\textnormal{d}r,\textnormal{d}z), \quad s \in [t,T],\end{align}
and 
\begin{align} \label{eq:BSDE}
Y_{  {s}}^{t,x}&=g(X_T^{t,x})+\int_{  {s}}^T f\left(r,X_{r-}^{t,x}, Y_{r-}^{t,x}, Z_r^{t,x}, \int_{\mathbb{R}^d} U_r^{t,x}(z) \nu (\textnormal{d}z) \right)\textnormal{d}r - \int_{  {s}}^T (Z_r^{t,x})^{\top}\textnormal{d}W_r  \nonumber \\
	& \quad - \int_{  {s}}^T \int_{\mathbb{R}^d} U_r^{t,x}(z) \widetilde{N} (\textnormal{d}r,\textnormal{d}z), \quad {s \in [t,T].}
\end{align}
Throughout the paper we make us of the following assumptions.
\begin{assumption} \label{assump:ExistenceUniquenessFBSDE} The following holds:
\begin{enumerate}
\item[\textbf{(A1)}] 
\begin{enumerate}
	\item The vector fields $b: \bR^d \to \bR^d, \sigma:\bR^d \to \bR^{d \times d}$ are measurable functions which are Lipschitz continuous with Lipschitz constants $K_b$ and $K_{\sigma}$, respectively.
	\item The vector field $ \Gamma: \bR^d \times \bR^d \to \bR^d$ is measurable and satisfies for some constant $K_{\Gamma} >0$
	 \begin{align*}
	 	\int_{\vert z \vert <1} \big \vert \Gamma(x,z)-\Gamma(x',z) \big \vert^2\nu(\textnormal{d}z)& \leq K_{\Gamma}  \vert x-x'  \vert^2, \\
	 	\int_{\vert z \vert <1} \big \vert \Gamma(x,z) \big \vert^2\nu(\textnormal{d}z) &\leq K_{\Gamma} ( 1+ \vert x  \vert^2 ).
	 	\end{align*}
	 \item The function $f:[0,T] \times \bR^d \times \bR \times \bR^d \times \bR^d \to \bR$ is Lipschitz continuous with respect to the state variables, uniformly in $t$, i.e. there exists $K_f >0$ such that
	  \begin{align*}
 	 \big \vert f(t,x,y,z,u)-f(t,x',y',z',u') \big \vert \leq K_f \big(\big \vert  x-x' \big \vert + \big \vert  y-y' \big \vert +\big \vert  z-z' \big \vert + \big \vert  u-u' \big \vert\big) 
 	\end{align*}
 	for all $(t,x,y,z,u), (t,x',y',z',u') \in [0,T] \times \bR^d \times \bR \times \bR^d \times \bR^d$.
	\item The function $g: \bR^d \to \bR$ is Lipschitz continuous with Lipschitz constant $K_g$. 
\end{enumerate}
\vspace{5pt}
\item [\textbf{(A2)}]
\begin{enumerate}
 \item The function $f:[0,T] \times \bR^d \times \bR \times \bR^d \times \bR \to \bR$ is $\frac{1}{2}$-H\"{o}lder continuous with respect to $t$ and constant $K_{f,t}.$
 \item For each $z\in \mathbb{R}^d$ the mapping $x \mapsto \Gamma(x,z)$ admits a Jacobian matrix $\nabla_x \Gamma (\cdot,z)$ such that the function $a(\cdot, \cdot;z): \mathbb{R}^d \times \mathbb{R}^d \to \bR$ defined by
 \begin{equation} \label{eq:AssumpGammaErrorEstiamte1}
a(x,\xi;z):=\xi^{\top}\left( \nabla_x \Gamma(x,z)+ \textnormal{I}_{d\times d}\right)\xi
 \end{equation}
satisfies one of the following conditions uniformly in $(x,\xi) \in \bR^d \times \bR^d$
\begin{equation} \label{eq:AssumpGammaErrorEstiamte2}	
a(x,\xi;z) \geq \vert \xi \vert^2 C^{-1} \quad \textnormal{or} \quad a(x,\xi;z) \leq -\vert \xi\vert^2 C^{-1}.
\end{equation}
\end{enumerate}
\item [\textbf{(A3)}] 
	Let the function $\Gamma: \bR^d \times \bR^d \to \bR^d$ be of the form
	\begin{equation} \label{eq:Gamma}
		\Gamma(x,z):=\gamma(x)z,
	\end{equation}
	for some Lipschitz continuous function $\gamma: \bR^d \to \bR^{d \times d}$ with Lipschitz constant $K_{\gamma}$.
\end{enumerate}
\end{assumption}

In particular, Assumption \ref{assump:ExistenceUniquenessFBSDE} \textbf{(A1)} guarantees the existence of a unique solution \linebreak $(X^{t,x}, Y^{t,x}, Z^{t,x}, U^{t,x}) \in \bS_{[0,T]}^2(\bR^d) \times \bS_{[0,T]}^2(\bR) \times \bH_{[0,T]}^2(\bR^d) \times \bH_{[0,T],N}^2(\bR) $ to the FBSDE in \eqref{eq:SDE}-\eqref{eq:BSDE}, see Theorem 4.1.3 in \cite{delong} and Theorem 6.2.9 in \cite{applebaum}.
\begin{theorem}
	Under Assumption \ref{assump:ExistenceUniquenessFBSDE} \textbf{(A1)} there exists a unique solution \linebreak $(X^{t,x}, Y^{t,x}, Z^{t,x}, U^{t,x}) \in \bS_{[0,T]}^2(\bR^d) \times \bS_{[0,T]}^2(\bR) \times \bH_{[0,T]}^2(\bR^d) \times \bH_{[0,T],N}^2(\bR) $ to the FBSDE \eqref{eq:SDE}-\eqref{eq:BSDE}.
\end{theorem}
Moreover, it is well-known that there exists a link of the solution to the FBSDE \eqref{eq:SDE}-\eqref{eq:BSDE} to the solution $u:[0,T] \times \mathbb{R}^d \to \mathbb{R}$ of the following PIDE
\begin{align}
 -u_t(t, x)-\mathscr{L} u(t, x) -f\left(t, x, u(t, x), D_x u(t, x) \sigma(x), \mathscr{J} u(t, x)\right)&=0, \quad \quad(t, x) \in[0, T) \times \mathbb{R}^d, \label{eq:PIDE1}\\
 u(T, x)&=g(x), \quad x \in \mathbb{R}^d, \label{eq:PIDE2}
\end{align}
where 
\begin{align}
\mathscr{L} u(t, x)= & \left\langle {b}(x), D_x u(t, x)\right\rangle+\frac{1}{2}\left\langle\sigma(x) D_x^2 u(t, x), \sigma(x)\right\rangle \nonumber  \\
& +\int_{\mathbb{R}^d}\left(u(t, x+\Gamma(x, z))-u(t, x)-\left\langle\Gamma(x, z), D_x u(t, x)\right\rangle\right) \nu(d z),  \label{eq:PIDE3 }\\
\mathscr{J} u(t, x)= & \int_{\mathbb{R}^d}(u(t, x+\Gamma(x, z))-u(t, x)) \nu(d z) . \nonumber 
\end{align}
From now on, by saying that a generic constant $C\in \mathbb R$ only depends on the data of the FBSDE in \eqref{eq:SDE}-\eqref{eq:BSDE}, we mean that $C$ only depends on $b, \sigma, \Gamma, f, g, {t}, T, x$ and, wether finite,  on $\nu(\mathbb{R}^d)$.

By combining Proposition 4.1.1, Lemma 4.1.1 and Theorem 4.2.1 and Lemma 4.1.1 in \cite{delong} we get the following result.
\begin{theorem} \label{theorem:LinkPIDE}
Let Assumption \ref{assump:ExistenceUniquenessFBSDE} \textbf{(A1)} hold and denote by $X^{t,x} \in \bS_{[0,T]}^2(\bR^d)$ the unique solution to the SDE \eqref{eq:SDE}-\eqref{eq:BSDE}. 
If $u \in C^{1,2}([0,T] \times \mathbb{R}^d, \mathbb{R})$ satisfies the PIDE \eqref{eq:PIDE1}-\eqref{eq:PIDE2} and there exists a constant $K>0$ such that  
\begin{align} \label{eq:EstimatesRegularityPIDE}
\vert u(t,x) \vert \leq K(1+ \vert x \vert), \quad \vert D_xu(t,x) \vert \leq K (1+ \vert x \vert), \quad \text{for } (t,x) \in [0,T] \times \mathbb{R}^d,
\end{align}
then $(Y^{t,x}, Z^{t,x}, U^{t,x})$ defined by
\begin{align}
    Y_s&=u\left(s,X_s^{t,x}\right), \quad  s \in [t,T], \label{eq:RelationY}  \\
Z_s^{t, x} & =\sigma(X_{s-}^{t, x})^{\top} D_x u(s, X_{s-}^{t, x}), \quad s \in [t,T], \nonumber  \\
U_s^{t, x}(z) & =u(s, X_{s-}^{t, x}+\Gamma(X_{s-}^{t, x}, z))-u(s, X_{s-}^{t, x}), \quad s \in [t,T], \, z \in \mathbb{R}^d,\label{eq:RelationU} 
\end{align}
is the unique solution to the BSDE \eqref{eq:BSDE}. 
Moreover, for some constant $C>0$ only depending on the data of the FBSDE in \eqref{eq:SDE}-\eqref{eq:BSDE}, $u$ satisfies
\begin{equation} \label{eq:LocallyHoelder}
\vert u(t,x)-u(t',x')\vert^2 \leq C \left( \vert x-x' \vert^2 + \left( 1+ \vert x {\vee} x' \vert^2 \right) \vert t-t' \vert \right), \quad \text{for }(t,x), (t',x') \in [0,T]\times \mathbb{R}^d.
\end{equation}
\end{theorem}
From now on, we assume that the PIDE in \eqref{eq:PIDE1}-\eqref{eq:PIDE2} has a solution $u\in C^{1,2}([0,T] \times \mathbb{R}^d, \mathbb{R})$ which satisfies \eqref{eq:EstimatesRegularityPIDE} and which is $\frac{1}{2}$-H\"older continuous with respect to time.
\begin{remark}
For examples of PIDEs which are relevant in financial applications and allow for a $C^{1,2}$-solution satisfying \eqref{eq:EstimatesRegularityPIDE} we refer to Proposition 2 in \cite{cont_voltchkova} and references therein.
\end{remark}
Theorem \ref{theorem:LinkPIDE} motivates the following stochastic optimal control problem 
\begin{align} \label{eq:StochastiOptimalControl}
&\begin{aligned}
& \underset{y \in \mathbb{R}, \,  Z \in \mathbb{H}_T^2\left(\mathbb{R}^d\right), \, U \in \mathbb{H}_{[0,T], N}^2(\mathbb{R})}{\operatorname{minimise}} \quad \mathbb{E}\left[\left(g\left(X_T^{t, x}\right)-Y_T^{y, Z, U}\right)^2\right] \nonumber
\end{aligned}\\
&\text { subject to: }\\
&\begin{cases}
X_s^{t, x}&=x+\int_t^s b(X_{r-}^{t, x}) \textnormal{d} r+\int_t^s \sigma(X_{r-}^{t, x})^{\top} \textnormal{d} W_r+\int_t^s \int_{\mathbb{R}^d} \Gamma(X_{r-}^{t, x}, z) \widetilde{N}(\textnormal{d} r, \textnormal{d} z) \nonumber \\
Y_s^{y, Z, U}&= y-\int_t^s f\left(r, X_{r-}^{t,x}, Y_{r-}^{y,Z,U}, Z_r, \int_{\mathbb{R}^d} U_r(z) \nu(\textnormal{d} z)\right) \textnormal{d} r+\int_t^s\left(Z_r\right)^{\top} \textnormal{d} W_r \\
& \quad +\int_t^s \int_{\mathbb{R}^d} U_r(z) \widetilde{N}(\textnormal{d}r,\textnormal{d}z),  \quad s \in[t, T] .
\end{cases}
\end{align}
In particular, under Assumption \ref{assump:ExistenceUniquenessFBSDE} \textbf{(A1)} the minimum in problem \eqref{eq:StochastiOptimalControl} is achieved and the corresponding minimizer $(y^*,Z^*, U^*)$ is the unique solution to the FBSDE in \eqref{eq:SDE}-\eqref{eq:BSDE}, see Lemma 2.6 in \cite{gnoatto_patacc_picarelli}. Then, Theorem \ref{theorem:LinkPIDE} provides the theoretical foundation for the algorithm  we present in the next subsection.
\subsection{The extension of the deep BSDE solver in \cite{ehanjen17} to the case with jumps}
{Using the stochastic optimal control problem in \eqref{eq:StochastiOptimalControl} the idea for the algorithm is to discretize the FBSDE and approximate at each time step the control processes by families of artificial neural networks. The goal is to minimize the loss function in \eqref{eq:StochastiOptimalControl} by fitting the parameters of the neural networks. While this approach is frequently used in the case without jumps, see e.g. \cite{ehanjen17}, the difficulty with jumps is to develop an algorithm which is also able to deal with a forward process $X$ with infinitely many jumps of finite or infinite variation. To overcome this problem the following procedure is used. \\
In a first step an algorithm for the finite activity case is developed. In a second step the forward process $X$ is approximated by a jump diffusion $X^{\epsilon}$ with finitely many jumps, which allows to make use of the already existing algorithm for the finite activity case. The approximation of the forward process relies here on results in \cite{asmussen_rosinski_2001}, \cite{cohen_rosinski_2007} and \cite{jum_2015}.
\\  Before presenting the algorithms for the finite and infinite activity case in details, we introduce the architecture of the underlying artificial neural networks we employ.}\\
{
We use feedforward ANNs with $\mathcal{L}+1 \in \mathbb{N} \backslash\{1,2\}$ layers, with each layer consisting of $v_{\ell}$ nodes or neurons, for $\ell=0, \ldots, \mathcal{L}$. The $0$-th layer is called the input layer, while the $\mathcal{L}$-th layer represents the output layer. The remaining $\mathcal{L}-1$ layers are termed hidden layers. For simplicity, we set $v_{\ell}=v, \ell=1, \ldots, \mathcal{L}-1$. The dimension of the input layer is equal to $d$, which equals the dimension of the forward process $X$.
A feedforward neural network is a function defined via the composition
$$
x \in \mathbb{R}^d \longmapsto \mathcal{A}_{\mathcal{L}} \circ \varrho \circ \mathcal{A}_{\mathcal{L}-1} \circ \ldots \circ \varrho \circ \mathcal{A}_1(x),
$$
where all $\mathcal{A}_{\ell}, \ell=1, \ldots, \mathcal{L}$, are linear transformations of the form $\mathcal{A}_{\ell}(x):=\mathcal{W}_{\ell} x+\beta_{\ell}, \ell=1, \ldots, \mathcal{L}$, where $\mathcal{W}_{\ell}$ and $\beta_{\ell}$ are matrices and vectors of suitable size called, respectively, weights and biases. The activation function $\varrho$ is a univariate function $\varrho: \mathbb{R} \mapsto \mathbb{R}$ that is applied componentwise to vectors. With an abuse of notation, we denote $\varrho\left(x_1, \ldots, x_v\right)=\left(\varrho\left(x_1\right), \ldots, \varrho\left(x_v\right)\right)$. The elements of $\mathcal{W}_{\ell}$ and $\beta_{\ell}$ are the parameters of the neural network. One can regroup all parameters in a vector of size $R:=\sum_{\ell=0}^{\mathcal{L}} v_{\ell}\left(1+v_{\ell}\right)$.} 

\section{A posteriori estimates}\label{sec:aPosterioriEstimates}

\subsection{A posteriori error estimate for the finite activity case} \label{sec:LiteratureSchemesFBSDEsJumps}
{The goal of this section is to present the algorithm  to approximate the solution of the  FBSDE in \eqref{eq:SDE}-\eqref{eq:BSDE} for the finite activity case and provide an a posteriori error estimate. Thus,} we work under the following assumption in this section.
 \begin{assumption} \label{assump:FiniteActivity}
	The L\'evy measure satisfies $K_{\nu}:=\nu(\bR^d)< \infty$. 
\end{assumption}
Furthermore, we assume that Assumption \ref{assump:ExistenceUniquenessFBSDE} \textbf{(A3)} holds from now on. For fixed $M \in \mathbb{N}$ we consider the uniform time discretization $0 \leq t = t_0<t_1<...<t_M=T$ with time step $\Delta t= t_{n+1}-t_n$ for $n=0,...,M-1$.
Next, we introduce a neural network $\cU_n^{\rho_n}: \mathbb{R}^d \to \mathbb{R}^d, n=0,...,M-1,$ parametrized by  $\rho=(\rho_0,...,\rho_{M-1}) \in (\mathbb{R}^R)^M$ and indexed by the time discretization $n \in \lbrace 0,...,M-1 \rbrace$. 
Then motivated by the stochastic optimal control problem in \eqref{eq:StochastiOptimalControl} we aim to find the family of neural networks $(\cU_n^{\rho_n})_{n=0,...,M-1}$ which minimizes the following loss function
\begin{align} \label{eq:LossFunctionAlgorithm}
\mathbb{E} \left[ \left( g(\widetilde{X}_M)-\widetilde{Y}_M^{y,\rho} \right)^2 \right],
\end{align}
where $\widetilde{X}=(\widetilde{X}_n)_{n=0,...,M-1}, \widetilde{Y}^{y, \rho}=( \widetilde{Y}^{y, \rho}_n)_{n=0,...,M-1}$ are given by 
\begin{small}
\begin{align} \label{eq:SchemeDeepSolverOneNeuralNetwork}
\begin{cases}
	\widetilde{X}_{n+1}&=\widetilde{X}_n +b(\widetilde{X}_n) \Delta t+ \sigma(\widetilde{X}_n)^{\top}\Delta W_n + \gamma(\widetilde{X}_n)\sum_{i=N([0,t_n], \mathbb{R}^d)+1}^{N([0,t_{n+1}], \mathbb{R}^d)}  z_i - \Delta t \int_{\mathbb{R}^d} \gamma(\widetilde{X}_n) z \nu (\textnormal{d}z), \\
	\widetilde{X}_0&=x,\\
	\widetilde{Y}_{n+1}^{y,\rho}&= \widetilde{Y}_n^{y, \rho}- f\left( t_n, \widetilde{X}_n, \widetilde{Y}_n^{y, \rho}, \sigma(\widetilde{X}_n)^{\top} D_x \mathcal{U}_{n}^{\rho_n}(\widetilde{X}_n), \int_{\mathbb{R}^d}\left(\cU_n^{\rho_n}(\widetilde{X}_n +\gamma(\widetilde{X}_n)z)- \cU_n^{\rho_n}(\widetilde{X}_n) \right)\nu(\textnormal{d} z) \right) \Delta t \\
	& \quad +(D_x \mathcal{U}_{n}^{\rho_n}(\widetilde{X}_n))^{\top}\sigma(\widetilde{X}_n) \Delta W_n + \left( \cU_n^{\rho_n}\left(\widetilde{X}_n + \gamma(\widetilde{X}_n)\sum_{i=N([0,t_n], \mathbb{R}^d)+1}^{N([0,t_{n+1}], \mathbb{R}^d)} z_i \right)- \cU_n^{\rho_n}(\widetilde{X}_n) \right)  \\
	& \quad -\Delta t \int_{\mathbb{R}^d} \left(\cU_n^{\rho_n}(\widetilde{X}_n + \gamma(\widetilde{X}_n)z)- \cU_n^{\rho_n}(\widetilde{X}_n) \right)\nu( \textnormal{d}z),  \\
	\widetilde{Y}_0^{y,\rho}&=\cU_0^{\rho_0}=:y \in \mathbb{{R}},
\end{cases}	
\end{align}
\end{small}
for $n=1,...,M-1$. Here, $N([0,t_n],\mathbb{R}^d)$ denotes the number of $\mathbb{R}^d$-valued jumps, which occur over the time interval $[0,t_n]$ and  $D_x \mathcal{U}_{n}^{\rho_n}$ is the gradient of $\cU_n^{\rho_n}$ that can be computed by means of the \texttt{Tensorflow} automatic differentiation capabilities.

\begin{remark}

\begin{enumerate}
\item 
A computationally efficient variant of the algorithm above is proposed in \cite{gnoatto_patacc_picarelli}. The presented minimization problem in \eqref{eq:LossFunctionAlgorithm} and \eqref{eq:SchemeDeepSolverOneNeuralNetwork} differs from the problem studied in \cite{gnoatto_patacc_picarelli} in the following way. While we consider one family of neural networks which approximates the stochastic integral with respect to the Brownian motion and with respect to the compensated random measure, the approach in \cite{gnoatto_patacc_picarelli} introduces a first family of ANNs  to approximate the gradient $D_x u$, a second one that approximates $\Delta u(t_n,\widetilde{X}_n, \Delta \widetilde{X}_n)$ with 
\begin{align*}
\Delta \widetilde{X}_n&= \gamma(\widetilde{X}_n)\sum_{i=N([0,t_n], \mathbb{R}^d)+1}^{N([0,t_{n+1}], \mathbb{R}^d)} z_i, \nonumber \\
\Delta u(t_n,\widetilde{X}_n, \Delta \widetilde{X}_n)&=u(t_n, \widetilde{X}_n + \Delta  \widetilde{X}_n)-u(t_n, \widetilde{X}_n),
\end{align*}
and a third one  that approximates
\begin{align}  \label{eq:CompensatorTerm}
\int_{\mathbb{R}^d} u(t_n, \widetilde{X}_n+\gamma(\widetilde{X}_n) z)-u(t_n, \widetilde{X}_n) \nu(\textnormal{d} z).
\end{align}

For the purpose of deriving error estimates we consider only one family of neural networks to keep the representation of the estimates simpler, as training two further families of neural networks requires the introduction of additional penalization terms in the loss function.

\item In \cite{ehanjen17} and \cite{gnoatto_patacc_picarelli}, instead of computing the gradient of $\mathcal{U}_n^{\rho_n}$ by using  {\fontfamily{qcr}\selectfont Tensorflow}, another family of neural networks $(\mathcal{V}_n^{\theta_n})_{n=0,...,M-1}$ is introduced to approximate this gradient $D_x u$. Under the growth condition 
$$
\big \vert  \mathcal{V}_n^{\theta_n} (x) \big \vert^2\leq C_n,
$$
for a constant $C_n\geq 0$, for $n=0,\ldots, M-1$ (see Assumption \ref{assump:NeuralNetwork} below), the error estimates derived in the following still hold true.
\end{enumerate}
\end{remark}

We assume from now on that the SDE in \eqref{eq:SDE} starts (in the fixed point $x \in \mathbb{R}^d$) at time $t=0$ and to simplify the notation we denote the solution of the FBSDE in \eqref{eq:SDE}-\eqref{eq:BSDE} by $(X,Y,Z,U)$. Moreover, given this solution we define the process $\Psi=(\Psi_t)_{t \in [0,T]}$ by 
\begin{equation} \label{eq:Integral}
\Psi_t:=\int_{\bR^d} U_t(z)\nu(\textnormal{d}z), \quad t \in [0,T]. 
\end{equation}
With this notation we rewrite the scheme in \eqref{eq:SchemeDeepSolverOneNeuralNetwork} as follows 
\begin{align} \label{eq:SchemeDeepSolverOneNeuralNetworkRewritten1}
\begin{cases}
	\widetilde{X}_{n+1}&=\widetilde{X}_n +b(\widetilde{X}_n) \Delta t+ \sigma(\widetilde{X}_n)^{\top}\Delta W_n + \int_{t_n}^{t_{n+1}}\int_{\mathbb{R}^d} \gamma(\widetilde{X}_n) z \widetilde{N}(\textnormal{d}r,\textnormal{d}z), \\
	\widetilde{X}_0&=x,\\
	\widetilde{Z}_n^{  {\rho}}&= \sigma(\widetilde{X}_n)^{\top} D_x \mathcal{U}_{n}^{\rho_n}(\widetilde{X}_n),\\
	\widetilde{\Psi}_n^{  {\rho}}&= \int_{\bR^d}\left(\cU_n^{\rho_n}(\widetilde{X}_n +\gamma(\widetilde{X}_n) z )- \cU_n^{\rho_n}(\widetilde{X}_n)\right) \nu(\textnormal{d}z) \\
	\widetilde{Y}_{n+1}^{y,\rho}&= \widetilde{Y}_n^{y, \rho}- f\left( t_n, \widetilde{X}_n, \widetilde{Y}_n^{y, \rho}, \widetilde{Z}_n, \widetilde{\Psi}_n \right) \Delta t  +\widetilde{Z}_n^{\top} \Delta W_n   \\
	&\quad + \left( \cU_n^{\rho_n}\left(\widetilde{X}_n + \gamma(\widetilde{X}_n) \sum_{i=N([0,t_n], \mathbb{R}^d)+1}^{N([0,t_{n+1}], \mathbb{R}^d)}z_i \right)- \cU_n^{\rho_n}(\widetilde{X}_n) \right)	 - \int_{t_n}^{t_{n+1}} \widetilde{\Psi}_n \textnormal{d}t,  \\
	\widetilde{Y}_0^{y,\rho}&=y
\end{cases}	
\end{align}
for $n=0,...,M-1$ and {by suppressing the dependence on the neural network {parameter} $\rho$ and the starting value $y$ {to ease the notation,} }we denote its solution by $(\widetilde{X}_n, \widetilde{Y}_n, \widetilde{Z}_n, \widetilde{\Psi}_n)_{n =0,...,M}$. In particular, we set here $\widetilde{Z}_M=\widetilde{\Psi}_M:=0$ to simplify the notation.

As we are interested in $L^2$-error estimates, the solution $(\widetilde{X}_n, \widetilde{Y}_n, \widetilde{Z}_n, \widetilde{\Psi}_n)_{n =0,...,M}$ of \eqref{eq:SchemeDeepSolverOneNeuralNetworkRewritten1} needs to be square-integrable. To guarantee this, we add assumptions for the classes of families of neural networks we consider.
\begin{assumption} \label{assump:NeuralNetwork}  
We assume that for each $n=0,...,M-1$ the function $\mathcal{U}_n^{\rho_n}: \bR^d \to \bR^d$ is measurable and satisfies the following growth condition
\begin{equation}\label{eq:AssumptionNeuralNetworkGrowth}
	\big \vert \mathcal{U}_n^{\rho_n} (x) \big \vert^2\leq A_n + B_n \vert x \vert^2
\end{equation} 
for constants $A_n,B_n {\geq 0}$. 
Moreover, for each $n=0,...,M-1$ the function $\mathcal{U}_n^{\rho_n}: \bR^d \to \bR^d$ is differentiable and the gradient $D_x \mathcal{U}_{n}^{\rho_n}(x)$ satisfies the following growth condition
\begin{equation}\label{eq:AssumptionNeuralNetworkGrowthGradient}
	\big \vert D_x \mathcal{U}_n^{\rho_n} (x) \big \vert^2\leq C_n
\end{equation} 
for a constant $C_n  {\geq 0}$.
\end{assumption}
\begin{remark}
Note that Assumption \ref{assump:NeuralNetwork} is satisfied for {families of} neural networks defined by means of the sigmoid activation function.
\end{remark}
{In the next lemma we make use of the following estimates, which we summarize briefly.}
{
\begin{remark} \label{remark:Gamma}
Note that if the function $\Gamma: \bR^d \times \bR^d \to \bR^d$ satisfies Assumption \ref{assump:ExistenceUniquenessFBSDE} \textbf{(A3)} the following estimates hold for $x \in \mathbb{R}^d$
\begin{align*}
    \vert \gamma(x) \vert &\leq \vert \gamma (0) \vert + K_{\gamma} \vert x \vert \leq \widetilde{C}_{\gamma}(1+ \vert x \vert),   \nonumber \\
    \vert \gamma(x) \vert^2 &\leq 2 \vert \gamma (0) \vert^2 +2 K_{\gamma}^2 \vert x \vert^2\leq  C_{\gamma}(1+ \vert x \vert^2)
\end{align*}
with $\widetilde{C}_{\gamma}:=2 ( \vert \gamma (0) \vert \vee K_{\gamma}) $ and $C_{\gamma}:=2 ( \vert \gamma (0) \vert^2 \vee K_{\gamma}^2).$ {Note that similar estimates also hold for the Lipschitz continuous functions $\sigma$ and $b$. }
\end{remark}}

\begin{lemma} \label{lemma:SchemeAdaptedIntegrable}
Let Assumptions \ref{assump:ExistenceUniquenessFBSDE} \textbf{(A1)}, \textbf{(A3)}, \ref{assump:FiniteActivity} and \ref{assump:NeuralNetwork} hold. Then the solution \linebreak $(\widetilde{X}_n, \widetilde{Y}_n, \widetilde{Z}_n, \widetilde{\Psi}_n)_{n =0,...,\violet{M}}$ of \eqref{eq:SchemeDeepSolverOneNeuralNetworkRewritten1} is $(\mathcal{F}_n)_{  {n=0,...,M}}$-adapted and square-integrable.
\end{lemma}
\begin{proof}
The adaptedness follows straightforward by induction. 
We prove the square-integrability of $(\widetilde{X}_n)_{n=0,...,M}$ by induction. It is clear that $\widetilde{X}_0:=x$ is square-integrable. Thus, for $n >0$ we have 
\begin{align} 
	\bE\left[ \big  \vert \widetilde{X}_{n+1} \big \vert ^2\right]&\leq 4 \left(  \bE\left[ \big \vert \widetilde{X}_n \big \vert ^2\right] + \bE \left[ \big \vert b(\widetilde{X}_n) \Delta t \big \vert^2 \right] + \bE \left[ \Big \vert \int_{t_n}^{t_{n+1}}\int_{\mathbb{R}^d} \gamma(\widetilde{X}_n) z \widetilde{N}(\textnormal{d}r,\textnormal{d}z) \Big \vert^2  \right] \right).\label{eq:SquareIntegrabilityX}
\end{align}
By Assumption \ref{assump:ExistenceUniquenessFBSDE} \textbf{(A3)} and Assumption \ref{assump:FiniteActivity} itself and its implication $\int_{\bR^d}\big \vert  z\big \vert^2\nu(dz)<\infty$, {by Remark \ref{remark:Gamma}, the fact that  $\gamma(\widetilde{X}_n) z$ as a constant process on $[t_n, t_{n+1}]$ is an element in $\bH_{[t_n,t_{n+1}],N}^2(\bR^d)$ and {the} induction hypothesis it follows}
\allowdisplaybreaks
\begin{align}\bE \left[ \Big \vert \int_{t_n}^{t_{n+1}}\int_{\mathbb{R}^d} \gamma(\widetilde{X}_n) z \widetilde{N}(\textnormal{d}r,\textnormal{d}z) \Big \vert^2  \right] 
&= \bE \left[  \int_{t_n}^{t_{n+1}}\int_{\mathbb{R}^d} \vert \gamma(\widetilde{X}_n) z\vert ^2 \nu(\textnormal{d}z)\textnormal{d}r  \right]  \nonumber \\
&= \Delta t \bE\left[\big \vert   \gamma(\widetilde{X}_n) \big \vert^2 \right] \int_{\bR^d}\big \vert    z\big \vert^2\nu(dz)\nonumber \\
&{\leq  \Delta t C_{\gamma} \left( 1+ \mathbb{E}\left[ \big \vert \widetilde{X}_n \big \vert^2\right]\right) \int_{\bR^d}\big \vert    z\big \vert^2\nu(dz)} \nonumber \\
&<\infty. \label{eq:SquareIntegrabilityX1}
\end{align}
Again by the induction assumption, Remark \ref{remark:Gamma} and \eqref{eq:SquareIntegrabilityX}, \eqref{eq:SquareIntegrabilityX1} we conclude
$$
\mathbb{E}\left[ \big \vert \widetilde{X}_{n+1} \big\vert^2 \right]< \infty.
$$
Obviously, $\mathbb{E}\left[\vert\widetilde{Z}_M \vert^2 \right]=\mathbb{E}\left[\vert\widetilde{\Psi}_M \vert^2 \right]=0$. For $n=0,...,M-1$, we get
\begin{align}
\bE\left[ \vert \widetilde{Z}_n \vert^2 \right]	&= \bE \left[ \vert \sigma(\widetilde{X}_n)^{\top} D_x \cU_n^{\rho_n}(\widetilde{X}_n)\vert^2\right]\nonumber \\
&\leq \bE \left[ \big \vert \sigma(\widetilde{X}_n)\big \vert^2 C_n\right]\nonumber\\
& \leq 2 C_n \left( \vert \sigma(0) \vert^2 \vee K_{\sigma}^2\right) \left( 1 + \mathbb{E}\left[ \big \vert \widetilde{X}_n \big \vert^2 \right]\right)   \nonumber \\
&< \infty,
\end{align}
by using Assumption \ref{assump:NeuralNetwork}, the Lipschitz continuity of $\sigma$ and the square-integrability of $\widetilde{X}_n$. 
Furthermore, it holds by Assumption \ref{assump:ExistenceUniquenessFBSDE} \textbf{(A3)},  \ref{assump:FiniteActivity}, \ref{assump:NeuralNetwork}, { Jensen's inequality for $\sigma$-finite measures, see Corollary 23.10 in \cite{schilling_book_2005}, Remark \ref{remark:Gamma}} and the square-integrability of $\widetilde{X}_n$
\allowdisplaybreaks
\begin{align}
\bE \left[ \vert \widetilde{\Psi}_n \vert^2 \right]	&= \bE \left[ \left\vert \int_{\bR^d}\left(\cU_n^{\rho_n}(\widetilde{X}_n +\gamma(\widetilde{X}_n) z )- \cU_n^{\rho_n}(\widetilde{X}_n)\right) \nu(\textnormal{d}z)\right \vert^2 \right] \nonumber \\
& {\leq 2 \bE \left[  \int_{\bR^d}|\cU_n^{\rho_n}(\widetilde{X}_n+\gamma(\widetilde{X}_n) z )|^2  \nu(\textnormal{d}z) \right] +2\bE \left[  \int_{\bR^d}|\cU_n^{\rho_n}(\widetilde{X}_n)|^2 \nu(\textnormal{d}z) \right]} \nonumber \\
& {\leq 2 \bE \left[  \int_{\bR^d} \left(A_n + B_n\left(\widetilde{X}_n+\gamma(\widetilde{X}_n) z \right)^2\right)  \nu(\textnormal{d}z) \right] +2\bE \left[  \int_{\bR^d}\left(A_n + B_n  \vert \widetilde{X}_n \vert^2 \right) \nu(\textnormal{d}z) \right]} \nonumber \\
& {= 2 K_{\nu}\left( 2 A_n + 3 B_n  \mathbb{E}\left[ \big \vert \widetilde{X}_n \big \vert^2 \right] \right) + 4B_n C_{\gamma} \left(1+ \bE\left[ \big \vert \widetilde{X}_n \big \vert^2 \right]\right) \int_{\mathbb{R}^d} \vert z \vert^2 \nu(\textnormal{d}z) } \nonumber \\
&< \infty,\nonumber 
\end{align}
for $n=0,...,M-1$. 
Moreover, note that $\mathbb{E}\left[\vert \widetilde{Y}_0\vert^2\right]=\mathbb{E}\left[\vert y\vert^2\right]<\infty$ and for $n=1,...,M-1$ we have
\begin{align}
\bE \left[\big \vert \widetilde{Y}_{n+1} \big\vert^2\right]&\leq 4\bE \left[\big \vert \widetilde{Y}_{n} \big\vert^2\right] +4\bE \left[\big \vert \- f\left( t_n, \widetilde{X}_n, \widetilde{Y}_n, \widetilde{Z}_n, \widetilde{\Psi}_n \right) \Delta t \big \vert^2 \right]  +4\bE\left[\big \vert \widetilde{Z}_n^{\top} \Delta W_n \big \vert^2\right]\nonumber \\
& \quad +4 \bE \left[\bigg \vert \int_{t_n}^{t_{n+1}} \int_{\bR^d}  \left( \cU_n^{\rho_n}(\widetilde{X}_n + \gamma(\widetilde{X}_n) z)- \cU_n^{\rho_n}(\widetilde{X}_n) \right) \widetilde{N}(\textnormal{d}r, \textnormal{d}z) \bigg \vert ^2 \right], \label{eq:SquareIntegrabilityY}
\end{align}
where the first three terms in \eqref{eq:SquareIntegrabilityY} are {finite} by using the Lipschitz continuity of $f$ and the square-integrability of $\widetilde{X}_n,\widetilde{Y}_n, \widetilde{Z}_{  {n}}, \widetilde{\Psi}_{  {n}}$. Thus, we only consider the last term. In particular, we show that $\cU_n^{\rho_n}(\widetilde{X}_n + \gamma(\widetilde{X}_n) z)- \cU_n^{\rho_n}(\widetilde{X}_n)$, as a constant process on $[t_n, t_{n+1}]$ is an element in $\bH_{[t_n,t_n+1],N}^2(\bR^d)$. We have 
\allowdisplaybreaks
\begin{align*} 
&\bE \left[ \int_{t_n}^{t_{n+1}} \int_{\bR^d}  \big \vert \cU_n^{\rho_n}(\widetilde{X}_n + \gamma(\widetilde{X}_n) z)- \cU_n^{\rho_n}(\widetilde{X}_n) \big\vert^2 \nu(\textnormal{d}z) \textnormal{d}r  \right] \nonumber \\
&\leq \Delta t \ \bE \left[  \int_{\bR^d} \left( \big \vert \cU_n^{\rho_n}(\widetilde{X}_n + \gamma(\widetilde{X}_n) z) \big \vert^2 + \big \vert \cU_n^{\rho_n}(\widetilde{X}_n) \big\vert^2 +2 \big \vert \cU_n^{\rho_n}(\widetilde{X}_n + \gamma(\widetilde{X}_n) z) {\big\vert \big \vert} \cU_n^{\rho_n}(\widetilde{X}_n)  \big \vert \right) \nu(\textnormal{d}z)   \right] \nonumber \\
& \leq {\Delta t K_{\nu} \left( 2A_n + 2 B_n \mathbb{E}\left[ \big \vert \widetilde{X}_n \big \vert^2  \right]+ B_n C_{\gamma}\left( 1+  \mathbb{E}\left[ \big \vert \widetilde{X}_n \big \vert^2 \right]\right) \int_{\mathbb{R}^d} \vert z \vert^2 \nu(\textnormal{d}z) \right) } \nonumber \\
& \quad {+ \Delta t \ \bE \left[  \int_{\bR^d} 2 \sqrt{A_n + B_n  \vert \widetilde{X}_n \vert^2 } \sqrt{A_n + B_n \left( \vert \widetilde{X}_n \vert^2 + 2  \vert \gamma(\widetilde{X}_n)\vert \vert z \vert \vert \widetilde{X}_n \vert + \vert \gamma(\widetilde{X}_n)\vert^2 \vert z \vert^2 \right) }  \nu(\textnormal{d}z)   \right]} \nonumber \\
&  \leq {\Delta t K_{\nu} \Bigg( 4A_n + 4 B_n \mathbb{E}\left[ \big \vert \widetilde{X}_n \big \vert^2 \right]+ 3B_n C_{\gamma}\left( 1+  \mathbb{E}\left[ \big \vert \widetilde{X}_n \big \vert^2 \right]\right) \int_{\mathbb{R}^d} \vert z \vert^2 \nu(\textnormal{d}z) } \nonumber \\
& \quad {+ B_n\widetilde{C}_{\gamma} \mathbb{E}\left[ \vert \widetilde{X}_n\vert (1+ \vert \widetilde{X}_n\vert) \right] \int_{\mathbb{R}^d} \vert z \vert \nu(\textnormal{d}z) \Bigg) }\nonumber \\
&< \infty 
\end{align*}
by Assumption \ref{assump:ExistenceUniquenessFBSDE} \textbf{(A3)}, \ref{assump:FiniteActivity} and \ref{assump:NeuralNetwork}.
 Thus, it follows 
\begin{align}
	&\bE \left[\bigg \vert \int_{t_n}^{t_{n+1}} \int_{\bR^d}  \left( \cU_n^{\rho_n}(\widetilde{X}_n + \gamma(\widetilde{X}_n) z)- \cU_n^{\rho_n}(\widetilde{X}_n) \right) \widetilde{N}(\textnormal{d}r, \textnormal{d}z) \bigg \vert ^2 \right]\nonumber \\
	&=\bE \left[ \int_{t_n}^{t_{n+1}} \int_{\bR^d}  \left| \cU_n^{\rho_n}(\widetilde{X}_n + \gamma(\widetilde{X}_n) z)- \cU_n^{\rho_n}(\widetilde{X}_n) \right|^2 \nu(\textnormal{d}z)\textnormal{d}r  \right]	<\infty.
	\end{align}
\end{proof}

{To derive an a posteriori error estimate for the solution of the FBSDE we follow a similar approach as in \cite{Han_Long_2020}, where the case of coupled FBSDEs without jumps is studied.
Our goal is to find an upper bound for the following expression
\begin{align} \label{eq:ErrorEstimateFBSDEGoal1}
 &\max_{n=0,1,...,M-1} \mathbb{E} \left[ \sup_{t \in [t_n,t_{n+1}]} \big \vert X_t - \widetilde{X}_n \big \vert^2 \right]+ \max_{n=0,1,...,M-1} \mathbb{E} \left[ \sup_{t \in [t_n,t_{n+1}]} \big \vert Y_t - \widetilde{Y}_n \big \vert^2 \right] \nonumber \\
 & \quad  + \mathbb{E} \left[ \sum_{n=0}^{M-1}\int_{t_n}^{t_{n+1}} \big \vert Z_t - \widetilde{Z}_n \big \vert^2 \textnormal{d}t \right]+  \mathbb{E} \left[ \sum_{n=0}^{M-1}\int_{t_n}^{t_{n+1}}\big \vert \Psi_t - \widetilde{\Psi}_n \big \vert^2 \textnormal{d}t \right], \end{align}
 where $(\widetilde{X}_n, \widetilde{Y}_n,\widetilde{Z}_n,\widetilde{\Psi}_n)_{n=0,...,  {M}}$ is defined in \eqref{eq:SchemeDeepSolverOneNeuralNetworkRewritten1} and $(X,Y,Z,U)$ is the solution of the FBSDE in \eqref{eq:SDE}-\eqref{eq:BSDE} and $\Psi$ is given by \eqref{eq:Integral}. \\
To do so, we make use of an error estimate for a discretization scheme for the FBSDE in \eqref{eq:SDE}-\eqref{eq:BSDE}, which is introduced in \cite{bouchard_elie_2008}.}
Given a regular grid $0=t_0<t_1<...<t_M=T$ with step size $\Delta t$ we approximate $(X,Y,Z,\Psi)$ by $(\overline{X}_n, \overline{Y}_n,\overline{Z}_n,\overline{\Psi}_n)_{n=0,...,M}$ defined by
 \begin{align} \label{eq:SchemeDelong}
 \begin{cases}
 	\overline{X}_{n+1}&=\overline{X}_n +b(\overline{X}_n) \Delta t+ \sigma(\overline{X}_n)^{\top}\Delta W_n + \int_{t_n}^{t_{n+1}}\int_{\mathbb{R}^d} \gamma(\overline{X}_n)z\widetilde{N}( \textnormal{d}r, \textnormal{d}z), \\
	\overline{X}_0&=x,\\
\overline{Y}_{M}&=g(\overline{X}_{M}), \\
	\overline{Z}_n&= \frac{1}{\Delta t} \mathbb{E} \left[ \overline{Y}_{n+1} \Delta W_n \vert \mathcal{F}_n \right], \\
	\overline{\Psi}_n&= \frac{1}{\Delta t} \mathbb{E} \left[ \overline{Y}_{n+1} \int_{t_n}^{t_{n+1}}\int_{\mathbb{R}^d} \widetilde{N}(\textnormal{d}r,\textnormal{d}z) \vert \mathcal{F}_n \right], \\
	\overline{Y}_n&= \mathbb{E}\left[\overline{Y}_{n+1} \vert \mathcal{F}_n \right]+f\left( t_n, \overline{X}_n, \overline{Y}_n, \overline{Z}_n, \overline{\Psi}_n\right) \Delta t,
	\end{cases}
 \end{align}
for $n=0,...,M-1$
with $\mathcal{F}_n:=\mathcal{F}_{t_n}$ for $n=0,...,M$ and $\overline{Z}_M=\overline{\Psi}_M:=0.$ 
For this scheme the following error estimate holds, see Theorem 5.1.1 in \cite{delong} and Corollary 2.1 and   {Remark 2.7} in \cite{bouchard_elie_2008}.

\begin{theorem} \label{theorem:ErrorEstimateSchemeDelong}
Let Assumptions \ref{assump:ExistenceUniquenessFBSDE} \textbf{(A1)}, \textbf{(A2)} and \ref{assump:FiniteActivity} hold. Let $(X,Y,Z,U)$ be the solution of the FBSDE in \eqref{eq:SDE}-\eqref{eq:BSDE} and $(\overline{X}_n, \overline{Y}_n,\overline{Z}_n,\overline{\Psi}_n)_{n=0,...,M}$ be given in \eqref{eq:SchemeDelong}. Then 
\begin{align} \label{eq:ErrorEstimateFBSDEWithJumps2}
 &\max_{n=0,1,...,M-1} \mathbb{E} \left[ \sup_{t \in [t_n,t_{n+1}]} \big \vert X_t - \overline{X}_n \big \vert^2 \right]+ \max_{n=0,1,...,M-1} \mathbb{E} \left[ \sup_{t \in [t_n,t_{n+1}]} \big \vert Y_t - \overline{Y}_n \big \vert^2 \right] \nonumber \\
 & \quad  + \mathbb{E} \left[ \sum_{n=0}^{M-1}\int_{t_n}^{t_{n+1}} \big \vert Z_t - \overline{Z}_n \big \vert^2 \textnormal{d}t \right]+  \mathbb{E} \left[ \sum_{n=0}^{M-1}\int_{t_n}^{t_{n+1}}\big \vert \Psi_t - \overline{\Psi}_n \big \vert^2 \textnormal{d}t \right] \leq C \Delta t,
 \end{align}
 where $C$ is a constant {only} depending on the data of the FBSDE in \eqref{eq:SDE}-\eqref{eq:BSDE}.
\end{theorem}
{Under Assumptions \ref{assump:ExistenceUniquenessFBSDE} \textbf{(A1)}, \textbf{(A2)} and \ref{assump:FiniteActivity}, we can rewrite \eqref{eq:ErrorEstimateFBSDEGoal1} with Theorem \ref{theorem:ErrorEstimateSchemeDelong} as follows}
 \allowdisplaybreaks
 \begin{align} \label{eq:ErrorEstimateFBSDEGoal2}
 &\max_{n=0,1,...,M-1} \mathbb{E} \left[ \sup_{t \in [t_n,t_{n+1}]} \big \vert X_t - \widetilde{X}_n \big \vert^2 \right]+ \max_{n=0,1,...,M-1} \mathbb{E} \left[ \sup_{t \in [t_n,t_{n+1}]} \big \vert Y_t - \widetilde{Y}_n \big \vert^2 \right] \nonumber \\
 & \quad  + \mathbb{E} \left[ \sum_{n=0}^{M-1}\int_{t_n}^{t_{n+1}} \big \vert Z_t - \widetilde{Z}_n \big \vert^2 \textnormal{d}t \right]+  \mathbb{E} \left[ \sum_{n=0}^{M-1}\int_{t_n}^{t_{n+1}}\big \vert \Psi_t - \widetilde{\Psi}_n \big \vert^2 \textnormal{d}t \right] \nonumber \\
 & \leq  {C} \Delta t+
  2  \max_{n=0,1,...,M-1} \mathbb{E} \left[ \big \vert \widetilde{X}_n - \overline{X}_n \big \vert^2 \right]+ 2\max_{n=0,1,...,M-1} \mathbb{E} \left[  \big \vert \widetilde{Y}_n - \overline{Y}_n \big \vert^2 \right] \nonumber \\
 & \quad  + 2 \Delta t \mathbb{E} \left[ \sum_{n=0}^{M-1} \big \vert \widetilde{Z}_n - \overline{Z}_n \big \vert^2  \right]+  2 \Delta t \mathbb{E} \left[ \sum_{n=0}^{M-1}\big \vert \widetilde{\Psi}_n - \overline{\Psi}_n \big \vert^2 \right],   \end{align}
 where $(\overline{X}_n,\overline{Y}_n,\overline{Z}_n,\overline{\Psi}_n)_{n=0,...,  {M}}$ is given in \eqref{eq:SchemeDelong}. Thus, we now aim to bound the expected squared differences between   $(\overline{X}_n,\overline{Y}_n,\overline{Z}_n,\overline{\Psi}_n)_{n=0,...,  {M}}$ and $(\widetilde{X}_n, \widetilde{Y}_n,\widetilde{Z}_n,\widetilde{\Psi}_n)_{n=0,...,M}$ in \eqref{eq:ErrorEstimateFBSDEGoal2} by $\bE\left[ \vert g (\widetilde{X}_T)-\widetilde{Y}_T \vert^2 \right].$
 To do so, we introduce the following system of equations for $(\overline{\overline{X}}_n,\overline{\overline{Y}}_n, \overline{\overline{Z}}_n, \overline{\overline{\Psi}}_n )_{n=0,...,M}$
 \begin{align} \label{eq:SchemeDelongWithoutTerminalCondition}
 \begin{cases}
 	\overline{\overline{X}}_{n+1}&=\overline{\overline{X}}_n +b(\overline{\overline{X}}_n) \Delta t+ \sigma(\overline{\overline{X}}_n)^{\top}\Delta W_n + \int_{t_n}^{t_{n+1}}\int_{\mathbb{R}^d} \gamma(\overline{\overline{X}}_n)z\widetilde{N}( \textnormal{d}r, \textnormal{d}z), \\
	\overline{\overline{X}}_0&=x,\\
	\overline{\overline{Z}}_n&= \frac{1}{\Delta t} \mathbb{E} \left[ \overline{\overline{Y}}_{n+1} \Delta W_n \vert \mathcal{F}_n \right], \\
	\overline{\overline{\Psi}}_n&= \frac{1}{\Delta t} \mathbb{E} \left[ \overline{\overline{Y}}_{n+1} \int_{t_n}^{t_{n+1}}\int_{\mathbb{R}^d} \widetilde{N}(\textnormal{d}r,\textnormal{d}z) \vert \mathcal{F}_n \right], \\
	\overline{\overline{Y}}_n&= \mathbb{E}\left[\overline{\overline{Y}}_{n+1} \vert \mathcal{F}_n \right]+f\left( t_n, \overline{\overline{X}}_n, \overline{\overline{Y}}_n, \overline{\overline{Z}}_n, \overline{\overline{\Psi}}_n\right) \Delta t
	\end{cases}
 \end{align}
with $\overline{\overline{Z}}_M=\overline{\overline{\Psi}}_M:=0$ and $\overline{\overline{Y}}_M$ is an arbitrary $\mathcal{F}_T$-measurable random variable.
Obviously, the discretization scheme $(\overline{X}, \overline{Y},\overline{Z},\overline{\Psi})$ defined in \eqref{eq:SchemeDelong} satisfies also the system of equations \eqref{eq:SchemeDelongWithoutTerminalCondition}, as we do not fix a terminal condition in \eqref{eq:SchemeDelongWithoutTerminalCondition}. We now proceed as follows to find an error estimate in \eqref{eq:ErrorEstimateFBSDEGoal2}.\\
 In a first step, we show that the solution of the scheme in \eqref{eq:SchemeDeepSolverOneNeuralNetwork}, $(\widetilde{X},\widetilde{Y}, \widetilde{Z}, \widetilde{\Psi}),$ satisfies the equations in the system \eqref{eq:SchemeDelongWithoutTerminalCondition}. In a second step, we derive an estimate for the difference of two solutions of the scheme in \eqref{eq:SchemeDelongWithoutTerminalCondition}. Applying this result to $(\widetilde{X},\widetilde{Y}, \widetilde{Z}, \widetilde{\Psi})$ and $(\overline{X}, \overline{Y},\overline{Z},\overline{\Psi})$, which both satisfy the system of equations in \eqref{eq:SchemeDelongWithoutTerminalCondition}, allows to conclude the error estimate in \eqref{eq:ErrorEstimateFBSDEGoal2}.

\allowdisplaybreaks

\begin{lemma} \label{lemma:SatisfiesScheme}
Let Assumptions \ref{assump:ExistenceUniquenessFBSDE} \textbf{(A1)}, \textbf{(A3)}, \ref{assump:FiniteActivity} and \ref{assump:NeuralNetwork} hold. Then the solution \linebreak $(\widetilde{X}_n, \widetilde{Y}_n, \widetilde{Z}_n, \widetilde{\Psi}_n)_{n =0,...,  {M}}$ of \eqref{eq:SchemeDeepSolverOneNeuralNetworkRewritten1} satisfies the system of equations in \eqref{eq:SchemeDelongWithoutTerminalCondition}.\end{lemma}
\begin{proof}
By taking the conditional expectation of $\widetilde{Y}_{n+1}$ in \eqref{eq:SchemeDeepSolverOneNeuralNetworkRewritten1}, we have   {with Lemma \ref{lemma:SchemeAdaptedIntegrable}}
\begin{align}
\bE \left[\widetilde{Y}_{n+1} \vert \cF_n \right]&= 	\widetilde{Y}_n- f\left( t_n, \widetilde{X}_n, \widetilde{Y}_n, \widetilde{Z}_n, \widetilde{\Psi}_n \right) \Delta t + \widetilde{Z}_n^{\top}\bE \left[ \Delta W_n \vert \cF_n\right] \nonumber \\
& \quad  + \bE \left[ \left( \cU_n^{\rho_n}\left(\widetilde{X}_n + \gamma(\widetilde{X}_n)\sum_{i=N([0,t_n], \mathbb{R}^d)+1}^{N([0,t_{n+1}], \mathbb{R}^d)} z_i\right)- \cU_n^{\rho_n}(\widetilde{X}_n) \right)	 - \int_{t_n}^{t_{n+1}} \widetilde{\Psi}_n \textnormal{dt}\vert \cF_n \right] \nonumber \\
&= \widetilde{Y}_n- f\left( t_n, \widetilde{X}_n, \widetilde{Y}_n, \widetilde{Z}_n, \widetilde{\Psi}_n \right) \Delta t  \nonumber \\
& \quad +\bE \left[ \int_{t_n}^{t_{n+1}} \int_{\bR^d}  \left( \cU_n^{\rho_n}(\widetilde{X}_n + \gamma(\widetilde{X}_n) z)- \cU_n^{\rho_n}(\widetilde{X}_n) \right) \widetilde{N}(\textnormal{d}r, \textnormal{d}z) \vert \cF_n \right] \nonumber \\
& = \widetilde{Y}_n- f\left( t_n, \widetilde{X}_n, \widetilde{Y}_n, \widetilde{Z}_n, \widetilde{\Psi}_n \right) \Delta t , \label{eq:YSatisfiesScheme1}
\end{align}
where we use {the fact that} $\cU_n^{\rho_n}(\widetilde{X}_n + \gamma(\widetilde{X}_n) z)- \cU_n^{\rho_n}(\widetilde{X}_n)$, as a constant process on $[t_n, t_{n+1}]$, is an element in $\bH_{[t_n,t_n+1],N}^2(\bR^d)$. 
By multiplying the equation for $\widetilde{Y}_{n+1}$ by $\Delta W_n$ and conditioning on $\cF_n$ we have
\begin{align}
&\bE \left[\widetilde{Y}_{n+1} \Delta W_n \vert \cF_n \right] \nonumber \\
&= \Delta t \ \widetilde{Z}_n + \bE \left[ \left(\int_{t_n}^{t_{n+1}} \int_{\bR^d}  \left( \cU_n^{\rho_n}(\widetilde{X}_n + \gamma(\widetilde{X}_n) z)- \cU_n^{\rho_n}(\widetilde{X}_n) \right) \widetilde{N}(\textnormal{d}r, \textnormal{d}z) \right) \left( \int_{t_n}^{t_{n+1}} dW_r\right) \vert \cF_n \right] \nonumber \\
&=\Delta t \ \widetilde{Z}_n, \nonumber
\end{align}
where we use again the regularity of $ \cU_n^{\rho_n}(\widetilde{X}_n + \gamma(\widetilde{X}_n) z)- \cU_n^{\rho_n}(\widetilde{X}_n)$. This, shows that $\widetilde{Z}_n$ satisfies the corresponding equation in \eqref{eq:SchemeDelongWithoutTerminalCondition}. 
  {M}ultiplying the equation for $\widetilde{Y}_{n+1}$ by $\int_{t_n}^{t_{n+1}} \int_{\mathbb{R}^d} \widetilde{N}(\textnormal{d}r,\textnormal{d}z)$ and considering the   {conditional} expected value, yields
\allowdisplaybreaks
\begin{align}
&\mathbb{E}\left[ \widetilde{Y}_{n+1}\int_{t_n}^{t_{n+1}} \int_{\mathbb{R}^d} \widetilde{N}(\textnormal{d}r,\textnormal{d}z)\vert \mathcal{F}_n\right] \nonumber
\\
&=\mathbb{E}\left[ \int_{t_n}^{t_{n+1}}\int_{\mathbb{R}^d}\left( \cU_n^{\rho_n}(\widetilde{X}_n + \gamma(\widetilde{X}_n) z)- \cU_n^{\rho_n}(\widetilde{X}_n) \right)\widetilde{N}(\textnormal{d}r,\textnormal{d}z) \int_{t_n}^{t_{n+1}} \int_{\mathbb{R}^d} \widetilde{N}(\textnormal{d}r,\textnormal{d}z)\vert \mathcal{F}_n\right] \nonumber \\
&=\mathbb{E}\left[ \int_{t_n}^{t_{n+1}}\int_{\mathbb{R}^d}\left( \cU_n^{\rho_n}(\widetilde{X}_n + \gamma(\widetilde{X}_n) z)- \cU_n^{\rho_n}(\widetilde{X}_n) \right){N}(\textnormal{d}r,\textnormal{d}z) \vert \mathcal{F}_n\right] \nonumber \\
&=\mathbb{E}\left[ \int_{t_n}^{t_{n+1}}\int_{\mathbb{R}^d}\left( \cU_n^{\rho_n}(\widetilde{X}_n + \gamma(\widetilde{X}_n) z)- \cU_n^{\rho_n}(\widetilde{X}_n) \right){\nu}(\textnormal{d}z)\textnormal{d}r \vert \mathcal{F}_n\right]\nonumber \\
&= \Delta t \ \int_{\mathbb{R}^d}\left( \cU_n^{\rho_n}(\widetilde{X}_n + \gamma(\widetilde{X}_n) z)- \cU_n^{\rho_n}(\widetilde{X}_n) \right){\nu}(\textnormal{d}z)\nonumber \nonumber\\
&=\Delta t \ \widetilde{\Psi}_n, 
 \end{align}
 which concludes the proof.
\end{proof}

{We prove the following auxiliary result, that we use to derive in a second step an estimate for the difference of two solutions satisfying the scheme in \eqref{eq:SchemeDelongWithoutTerminalCondition}.}
\begin{lemma} \label{lemma:FirstEstimateAuxiliary}
	Let $0 \leq s_1 < s_2 \leq T$ and $Q \in L^2(\Omega, \cF_{s_2},\bP)$. Let $(Z_t)_{t \in [s_1,s_2]}$ and $(U_t)_{t \in [s_1,s_2]}$ be $\bF$-predictable processes which are integrable with respect to $W$ and $\widetilde{N},$ respectively, and such that by the predictable representation property it holds
\begin{align}
Q = \bE \left[Q\vert \cF_{s_1} \right]+ \int_{s_1}^{s_2}( Z_r)^{\top} \textnormal{d}W_r + \int_{s_1}^{s_2}\int_{\bR^d} U_r(z) \widetilde{N}(\textnormal{d}r,\textnormal{d}z).
\end{align}
Then we have
\begin{align}
\bE \left[Q(W_{s_2}-W_{s_1})\vert \cF_{s_1} \right]&= \bE \left[\int_{s_1}^{s_2} Z_r\textnormal{d}r \Big \vert \cF_{s_1} \right],\label{eq:ProductDeltaBrownianMotion} \\
\bE \left[Q\left( \int_{s_1}^{s_2} \int_{\bR^d} \widetilde{N}(\textnormal{d}r,\textnormal{dz})\right)\Big \vert \cF_{s_1} \right]&= \bE \left[\int_{s_1}^{s_2} \int_{\bR^d} U_r(z) {N}(\textnormal{d}r,\textnormal{dz}) \Big \vert \cF_{s_1}\right].\label{eq:ProductDeltaJump}
\end{align}
\end{lemma}
\begin{proof}
Define the auxiliary stochastic process
\begin{equation*}
	Q_s = \left(\bE \left[Q \vert \cF_{s_1} \right]+\int_{s_1}^{s}( Z_r)^{\top} \textnormal{d}W_r + \int_{s_1}^{s}\int_{\bR^d} U_r(z) \widetilde{N}(\textnormal{d}r,\textnormal{d}z)\right) \int_{s_1}^s \textnormal{d}W_r, \quad s \in [s_1,s_2].
\end{equation*}	
By applying It\^{o}'s formula we get
\begin{align*}
\textnormal{d}Q_{s}& =(W_s - W_{s_1})(Z_s)^{\top}\textnormal{d}W_s+ \int_{\mathbb{R}^d}(W_s - W_{s_1})U_s(z)\widetilde{N}(\textnormal{d}s,\textnormal{d}z) + Z_s \textnormal{d}s  \nonumber \\
	& \quad + \left(\bE \left[Q \vert \cF_{s_1} \right]+\int_{s_1}^{s}(Z_r)^{\top} \textnormal{d}W_r + \int_{s_1}^{s}\int_{\bR^d} U_r(z) \widetilde{N}(\textnormal{d}r,\textnormal{d}z)\right) \textnormal{d}W_s .
\end{align*}
By taking the conditional expectation and using $Q_{s_1}=0$ and $Q_{s_2}=Q (W_{s_2}-W_{s_1})$, we have
\begin{align*}
	\bE \left [Q_{s_2} \vert \cF_{s_1} \right]= \bE \left[ \int_{s_1}^{s_2} Z_r \textnormal{d}r \Big \vert \cF_{s_1} \right],
\end{align*}
by using the martingale property of $\int \cdot \textnormal{d}W_r, \int \int_{\bR^d} \cdot \widetilde{N}(\textnormal{d}r,\textnormal{d}z)$ and $\left[\int \cdot \textnormal{d}W_r ,\int \int_{\bR^d} \cdot \widetilde{N}(\textnormal{d}r,\textnormal{d}z) \right]=0.$\\
We now define another auxiliary stochastic process
\begin{equation*}
	\overline{Q}_s = \left(\bE \left[\overline{Q}\vert \cF_{s_1} \right]+\int_{s_1}^{s}( Z_r)^{\top} \textnormal{d}W_r + \int_{s_1}^{s}\int_{\bR^d} U_r(z) \widetilde{N}(\textnormal{d}r,\textnormal{d}z)\right) {\int_{s_1}^s \int_{\bR^d} \widetilde{N}(\textnormal{d}r}, \textnormal{d}z), \quad s \in [s_1,s_2].
\end{equation*}	
By It\^{o}'s formula we have
\begin{align*}
	\textnormal{d}\overline{Q}_s&= \left(\bE \left[\overline{Q}\vert \cF_{s_1} \right]+\int_{s_1}^{s}( Z_r)^{\top} \textnormal{d}W_r + \int_{s_1}^{s}\int_{\bR^d} U_r(z) \widetilde{N}(\textnormal{d}r,\textnormal{d}z)\right) \int_{\bR^d}\widetilde{N}(\textnormal{d}s, \textnormal{d}z)  \nonumber\\
	& \quad + \left(\int_{s_1}^s \int_{\bR^d} \widetilde{N}(\textnormal{d}r,dz) \right)(Z_s)^{\top}\textnormal{d}W_s + \left( \int_{s_1}^s \int_{\bR^d} \widetilde{N}(\textnormal{d}r,\textnormal{d}z) \right)\int_{\bR^d}U_s(z) \widetilde{N}(\textnormal{d}s,\textnormal{d}z) \nonumber \\
    & \quad+ \int_{\bR^d}U_s(z) N(\textnormal{d}s,\textnormal{d}z),
\end{align*}
and by considering the conditional expectation we get
\begin{align*}
	\bE \left[Q\left( \int_{s_1}^{s_2} \int_{\bR^d} \widetilde{N}(\textnormal{d}r,\textnormal{dz})\right) \Big \vert \cF_{s_1} \right]= \bE \left[\int_{s_1}^{s_2} \int_{\bR^d} U_r(z) {N}(\textnormal{d}r,\textnormal{dz}) \Big \vert \cF_{s_1}\right].
\end{align*}

\end{proof}
{We are now in a position to state an estimate for the difference of two solutions of the scheme in \eqref{eq:SchemeDelongWithoutTerminalCondition}.}
 \begin{proposition} \label{prop:SchemeWithoutTerminalDifference}
Let Assumptions \ref{assump:ExistenceUniquenessFBSDE} \textbf{(A1)}, \textbf{(A3)} and \ref{assump:FiniteActivity} hold. For $j=1,2,$ let \linebreak $(X^j_n,Y^j_n,Z^j_n,\Psi^j_n)_{n=0,...,{M}}$ be two solutions of \eqref{eq:SchemeDelongWithoutTerminalCondition} such that $X^j_n,Y_n^j \in L^2(\Omega, \cF_n, \bP)$ for $n=0,...,{M}$. 
 Then for $n=0,...,{M-1}${,}
for any $\lambda>0$ and sufficiently small $\Delta t$ such that
\begin{equation} \label{eq:LambdaConstraint1}
\lambda^2 \geq (1 \vee K_{\nu})K_f \vee \frac{1-\sqrt{1-{12} (\Delta t)^2 K_f^2}}{{6} \Delta t K_f} \quad \text{and} \quad \lambda^2 \leq \frac{1+\sqrt{1-{12}  (\Delta t)^2 K_f^2}}{{6} \Delta t K_f},
\end{equation}
 we have
\begin{align} \label{eq:DifferenceTwoSolutionsSchemeWithoutTerminal2}
	&\bE\left[ \big \vert Y_n^1-Y_n^2 \big \vert^2 \right]\leq e^{F_1(M-n) \Delta t}\bE \left[ \big \vert  Y_{M}^1-Y_M^2 \big \vert^2 \right]
\end{align}
with 
\begin{equation} \label{eq:DefinitionConstantF}
    F_1:=-\frac{\ln(1-\Delta t K_f(4 \lambda ^2 + \frac{1}{\lambda^2}))}{\Delta t } . 
\end{equation} 
 \end{proposition}
 \begin{proof}
 For $n=0,...,{M}$ we use the following notation
 \begin{align*}
 \delta X_n := X^1_n -X^2_n, \quad  \delta Y_n := Y^1_n -Y^2_n, \quad  \delta Z_n := Z^1_n -Z^2_n, \quad  \delta \Psi_n := \Psi^1_n -\Psi^2_n.
\end{align*}
{Clearly, we have 
$ \delta X_n=0,$
as $X_0^1=X_0^2=x.$}
Moreover, we set
\begin{align*}
	\delta f_n:=& f(t_n,X^1_n,Y_n^1,Z_n^1,\Psi_n^1)-f(t_n,X^2_n,Y_n^2,Z_n^2,\Psi_n^2).
\end{align*}
With this notation we have {for $n=0,...,M-1$}
\begin{align}
\delta Z_{n}&= \frac{1}{\Delta t} \bE \left[ \delta Y_{n+1} \Delta W_n \vert \cF_n\right], \label{eq:DifferenceZ} \\
\delta \Psi_{n}&= \frac{1}{\Delta t} \bE \left[ \delta Y_{n+1} \int_{t_n}^{t_{n+1}}\int_{\bR^d} \widetilde{N} (\textnormal{d}r,\textnormal{d}z) \Big \vert \cF_n \right], \label{eq:DifferencePsi} \\
\delta Y_{n}&= \bE \left[ \delta Y_{n+1} \vert \cF_n\right]+\delta f_n \Delta t. \label{eq:DifferenceY} 
\end{align}
  {Fix $n=0,...,M-1.$}
Applying the martingale representation theorem{,} there exists $\bF$-predictable processes $(\delta Z_t)_{t \in [t_{  {n}},t_{  {n}+1}]}$ and $(\delta U_t)_{t \in [t_{  {n}},t_{  {n}+1}]}$,   {which are} integrable with respect to $W$ and $\widetilde{N}$, {respectively,} such that
\begin{align} \label{eq:ProofPredictableRepresentation}
\delta Y_{n+1} = \bE \left[\delta Y_{n+1} \vert \cF_n \right]+ \int_{t_n}^{t_{n+1}}(\delta Z_r)^{\top} \textnormal{d}W_r + \int_{t_n}^{t_{n+1}}\int_{\bR^d}\delta U_r(z) \widetilde{N}(\textnormal{d}r,\textnormal{d}z).
\end{align} 
{M}oreover, {by combining \eqref{eq:DifferenceY} and \eqref{eq:ProofPredictableRepresentation} it follows}
\begin{align}
&\bE\left[ \vert \delta Y_{n+1} \vert^2 \right]\nonumber \\
&=\bE\left[\bigg \vert  \delta Y_n {-} \delta f_n \Delta t+ \int_{t_n}^{t_{n+1}}(\delta Z_r)^{\top} \textnormal{d}W_r + \int_{t_n}^{t_{n+1}}\int_{\bR^d }\delta U_r(z) \widetilde{N}(\textnormal{d}r,\textnormal{d}z) \bigg \vert^2  \right]	\nonumber \\
&=\bE \left[ \vert \delta Y_n {-} \delta f_n \Delta t\vert^2 \right]+ \bE \left[\bigg \vert \int_{t_n}^{t_{n+1}}(\delta Z_r)^{\top} \textnormal{d}W_r + \int_{t_n}^{t_{n+1}}\int_{\bR^d }\delta U_r(z) \widetilde{N}(\textnormal{d}r,\textnormal{d}z)\bigg \vert^2  \right] \nonumber \\
& \quad +2 \bE \left[  \left( \delta Y_n {-} \delta f_n \Delta t\right) \left( \int_{t_n}^{t_{n+1}}(\delta Z_r)^{\top} \textnormal{d}W_r + \int_{t_n}^{t_{n+1}}\int_{\bR^d }\delta U_r(z) \widetilde{N}(\textnormal{d}r,\textnormal{d}z)\right)\right]\nonumber \\
&=\bE \left[ \vert \delta Y_n {-} \delta f_n \Delta t\vert^2 \right]+ \bE \left[\int_{t_n}^{t_{n+1}}\vert\delta Z_r\vert^2 \textnormal{d}r \right]+ \bE \left[\int_{t_n}^{t_{n+1}}\int_{\bR^d } \big \vert \delta U_r(z) \big \vert^2 \nu(\textnormal{d}z)\textnormal{d}r  \right] \nonumber \\
& \geq \bE \left[ \vert \delta Y_n \vert^2 \right]-2 \Delta t \bE \left[ \vert \delta f_n \vert \vert \delta Y_n\vert \right]  + \bE \left[\int_{t_n}^{t_{n+1}}\vert\delta Z_r\vert^2 \textnormal{d}r \right]+ \bE \left[\int_{t_n}^{t_{n+1}}\int_{\bR^d } \big \vert \delta U_r(z) \big \vert^2 \nu(\textnormal{d}z)\textnormal{d}r  \right]  \nonumber \\
& \geq \bE \left[ \vert \delta Y_n \vert^2 \right]-2 \Delta t K_f\bE \left[\vert \delta Y_n\vert (\vert \delta Y_n \vert+ \vert \delta Z_n \vert + \vert \delta \Psi_n \vert) \right]   \nonumber\\
& \quad + \bE \left[\int_{t_n}^{t_{n+1}}\vert\delta Z_r\vert^2 \textnormal{d}r \right] + \bE \left[\int_{t_n}^{t_{n+1}}\int_{\bR^d } \big \vert \delta U_r(z) \big \vert^2 \nu(\textnormal{d}z)\textnormal{d}r  \right].  \label{eq:ProofDifferencesY1}
\end{align}
By using that $-2ab \geq - a^2 - b^2$ it follows that 
\begin{align}
    &-2 \Delta t K_f\bE \left[\vert \delta Y_n\vert ( \vert \delta Y_n \vert+ \vert \delta Z_n \vert + \vert \delta \Psi_n \vert) \right] \nonumber \\
    &=-2 \Delta t K_f\bE \left[{\lambda \vert \delta Y_n\vert} \left(\frac{1}{\lambda} \vert \delta Y_n \vert+ \frac{1}{\lambda}\vert \delta Z_n \vert +\frac{1}{\lambda} \vert \delta \Psi_n \vert\right) \right] \nonumber \\
    & \geq - \Delta t K_f\left( {3}  \lambda^2 \bE \left[\vert \delta Y_n \vert^2 \right] +  \frac{1}{\lambda^2} \bE \left[\vert \delta Y_n \vert^2 \right] +   \frac{1}{\lambda^2} \bE \left[\vert \delta Z_n \vert^2 \right] + \frac{1}{\lambda^2} \bE \left[\vert \delta \Psi_n \vert^2 \right] \right). \label{eq:ImprovedApproach1}
\end{align}
By applying \eqref{eq:ProductDeltaJump} in Lemma \ref{lemma:FirstEstimateAuxiliary} to $\delta Y_{n+1}$ we get{,} with \eqref{eq:DifferencePsi} and \eqref{eq:ProofPredictableRepresentation}{,}
\begin{align*}
	\Delta t \delta {\Psi}_n &=   { \bE \left[ \delta Y_{n+1} \int_{t_n}^{t_{n+1}}\int_{\bR^d} \widetilde{N} (\textnormal{d}r,\textnormal{d}z) \Big \vert \cF_n \right]} =   \bE \left[ \int_{t_n}^{t_{n+1}} \int_{\bR^d} \delta U_r(z) N (\textnormal{d}r,\textnormal{d}z) \Big \vert \cF_n\right]\\
 &=\bE \left[ \int_{t_n}^{t_{n+1}} \int_{\bR^d} \delta U_r(z) \nu(\textnormal{d}z)\textnormal{d}r \Big \vert \cF_n\right],
\end{align*}
and thus
\begin{align}
\Delta t \bE \left[ \vert \delta \Psi_n \vert^2\right]
&= \frac{1}{\Delta t}  \bE \left[ \left(\bE \left[ \int_{t_n}^{t_{n+1}} \int_{\bR^d} \delta U_r(z)  \nu(\textnormal{d}z)\textnormal{d}r \vert \cF_n\right]\right)^2\right]\nonumber\\
&\leq  \frac{1}{\Delta t} \bE \left[ \left(\int_{t_n}^{t_{n+1}} \int_{\bR^d} \delta U_r(z)  \nu(\textnormal{d}z)\textnormal{d}r \right)^2\right]\nonumber\\
& \leq K_{\nu} \bE \left[ \left( \int_{t_n}^{t_{n+1}} \int_{\bR^d} \vert \delta U_r(z))\vert^2 \nu(\textnormal{d}z)\textnormal{d}r \right)\right]{.}\label{eq:WhatWeWantToProve1}
\end{align}
Similarly, by using \eqref{eq:ProductDeltaBrownianMotion} in Lemma \ref{lemma:FirstEstimateAuxiliary} together with \eqref{eq:DifferenceZ} and \eqref{eq:ProofPredictableRepresentation} we get
\begin{equation}\label{eq:SimilarHanLong}
\Delta t \bE \left[ \big \vert \delta  Z_n \big \vert^2 \right]   {\leq} \bE \left[ \int_{t_n}^{t_{n+1}} \vert Z_r\vert^2 \textnormal{d}r \right],
\end{equation}
  {for more details we refer to the proof of Lemma 1 in \cite{Han_Long_2020}.} Combining \eqref{eq:ProofDifferencesY1}, \eqref{eq:ImprovedApproach1}, \eqref{eq:WhatWeWantToProve1} and \eqref{eq:SimilarHanLong} it follows
\begin{align}
 \bE \left[ \big \vert \delta Y_{n+1} \big \vert^2\right] & \geq \left(1- \Delta t K_f \left({3} \lambda^2 + \frac{1}{\lambda^{2}}\right)\right)\bE \left[ \big \vert \delta Y_n \big \vert^2 \right] +\Delta t\left(1- K_f \frac{1}{\lambda^2}\right) \bE \left[ \big \vert \delta  Z_n \big \vert^2 \right]\nonumber \\
 & \quad + \Delta t\left(K_{\nu}^{-1}- K_f \frac{1}{\lambda^2}\right)  \bE \left[ \big \vert \delta \Psi_n \big \vert^2 \right]. \label{eq:BetterEstimate2}
\end{align}
We now choose $\lambda>0$ and $\Delta t$ sufficiently small such that 
\begin{align} 
    &1- \Delta t K_f \left({3} \lambda^2 + \frac{1}{\lambda^{2}}\right) \in (0,1) \quad \text{and} \quad 1- K_f \frac{1}{\lambda^2}\geq 0 \quad \text{and} \quad K_{\nu}^{-1}- K_f \frac{1}{\lambda^2} \geq 0. \label{eq:ConstraintNew1}
\end{align}
Note that \eqref{eq:ConstraintNew1} can be rewritten as
\begin{align*}
\lambda^2 \geq (1 \vee K_{\nu})K_f \vee \frac{1-\sqrt{1-{12} (\Delta t)^2 K_f^2}}{{6}  \Delta t K_f} \quad \text{and} \quad \lambda^2 \leq \frac{1+\sqrt{1-{12} (\Delta t)^2 K_f^2}}{{6} \Delta t K_f}.
\end{align*}
  {Then} \eqref{eq:BetterEstimate2} and \eqref{eq:ConstraintNew1}   {yields}
\begin{align*}
    \bE \left[ \big \vert \delta Y_n \big \vert^2\right] \leq \frac{1}{1-\Delta t K_f({3} \lambda ^2 + \frac{1}{\lambda^2})} \bE \left[\big  \vert \delta Y_{n+1} \big \vert^2\right]
\end{align*}
and   {by} induction for $0 \leq n \leq M$   {it follows}
\begin{align*}
     \bE \left[ \big \vert \delta Y_{n} \big \vert^2\right] \leq e^{{F}_1(M-n)\Delta t} \bE \left[\big  \vert \delta Y_M \big \vert^2\right]
\end{align*}
with 
$$
{F}_1:=-\frac{\ln(1-\Delta t K_f({3} \lambda ^2 + \frac{1}{\lambda^2}))}{\Delta t }.
$$
\end{proof}
{Combining the results of the first and second step, we now state an a posteriori error estimate for the finite activity case.}
 \begin{theorem} \label{theorem:ErrorEstiamteFiniteActivity}
Let Assumptions \ref{assump:ExistenceUniquenessFBSDE} \textbf{(A1)}-\textbf{(A3)}, \ref{assump:FiniteActivity} and \ref{assump:NeuralNetwork} hold. Let $(X,Y,Z,U)$ be the solution of the FBSDE in \eqref{eq:SDE}-\eqref{eq:BSDE}, $\Psi$ given in \eqref{eq:Integral} and $(\widetilde{X}_n,\widetilde{Y}_n,\widetilde{Z}_n,\widetilde{\Psi}_n)_{n=0,...,M}$ defined as in \eqref{eq:SchemeDeepSolverOneNeuralNetworkRewritten1}. Let $\overline{\lambda}>0$ and choose $\lambda {>0}$ and a sufficiently small $\Delta t$ such that \eqref{eq:LambdaConstraint1} holds 
{with 
$$
\lambda^2> K_f \left( 1 \vee K_{\nu} \right),
$$}then 
\begin{align} \label{eq:ErrorEstimateFBSDETheorem}
 &\max_{n=0,1,...,M-1} \mathbb{E} \left[ \sup_{t \in [t_n,t_{n+1}]} \big \vert X_t - \widetilde{X}_n \big \vert^2 \right]+ \max_{n=0,1,...,M-1} \mathbb{E} \left[ \sup_{t \in [t_n,t_{n+1}]} \big \vert Y_t - \widetilde{Y}_n \big \vert^2 \right] \nonumber \\
 & \quad  + \mathbb{E} \left[ \sum_{n=0}^{M-1}\int_{t_n}^{t_{n+1}} \big \vert Z_t - \widetilde{Z}_n \big \vert^2 \textnormal{d}t \right]+  \mathbb{E} \left[ \sum_{n=0}^{M-1}\int_{t_n}^{t_{n+1}}\big \vert \Psi_t - \widetilde{\Psi}_n \big \vert^2 \textnormal{d}t \right] \nonumber \\
 &\leq C \Delta t+ C(\lambda,\overline{\lambda}) \bE \left[\big \vert  \widetilde{Y}_M-g(\widetilde{X}_M)\big\vert^2 \right] \end{align}
with
$$
{C(\lambda,\overline{\lambda})}:= {2}\left(1+\overline{\lambda}^{-1}\right)  {e^{\overline{F}_1T} } \left[ {1}+ \frac{1+T K_f ({3} \lambda^2 + \frac{1}{\lambda^{2}})}{( 1 \wedge K_{\nu}^{-1})- K_f \frac{1}{\lambda^2}} \right]
$$
  {and
$$\overline{F}_1:= K_f \left({3} \lambda^2 + \frac{1}{\lambda^{2}}\right)+ \overline{\lambda}.$$}
 \end{theorem}

 \begin{proof}
 Let $(\widetilde{X}_n,\widetilde{Y}_n,\widetilde{Z}_n,\widetilde{\Psi}_n)_{n=0,...,M}$ and $(\overline{X}_n,\overline{Y}_n,\overline{Z}_n,\overline{\Psi}_n)_{n=0,...,M}$ be defined as in \eqref{eq:SchemeDeepSolverOneNeuralNetworkRewritten1} and \eqref{eq:SchemeDelong}, respectively. By Lemma \ref{lemma:SatisfiesScheme} it follows that $(\widetilde{X}_n,\widetilde{Y}_n,\widetilde{Z}_n,\widetilde{\Psi}_n)_{n=0,...,M}$ satisfies the scheme in \eqref{eq:SchemeDelongWithoutTerminalCondition} and obviously $(\overline{X}_n,\overline{Y}_n,\overline{Z}_n,\overline{\Psi}_n)_{n=0,...,M}$ does as well. We now apply the results of Proposition \ref{prop:SchemeWithoutTerminalDifference} in order to estimate the expectation of the following differences
  \begin{align*}
 \delta X_n := \overline{X}_n -\widetilde{X}_n, \quad  \delta Y_n := \overline{Y}_n -\widetilde{Y}_n,\quad  \delta Z_n := \overline{Z}_n -\widetilde{Z}_n,\quad  \delta \Psi_n :=\overline{\Psi}_n -\widetilde{\Psi}_n.
\end{align*}
For {any} $\overline{\lambda}>0$   {it holds}
 	\begin{align}
 		\bE \left[\big \vert \delta Y_M\big \vert^2 \right]
&= \bE \left[\big \vert  \widetilde{Y}_M-g(\overline{X}_M)\big\vert^2 \right]\nonumber \\
 		 &\leq \left(1+\overline{\lambda}^{-1}\right) \bE \left[\big \vert  \widetilde{Y}_M-g(\widetilde{X}_M)\big\vert^2 \right]+ (K_g)^2 \left(1+\overline{\lambda}\right)\bE \left[\big \vert \delta   {{X}}_M\big\vert^2 \right] \nonumber\\
 		 &= \left(1+\overline{\lambda}^{-1}\right) \bE \left[\big \vert  \widetilde{Y}_M-g(\widetilde{X}_M)\big\vert^2 \right], \label{eq:TheoremSummary1.0}
 	\end{align}
 by using {that $ \delta X_{  {M}}=0$} in the last equality.  
    {F}or suitable $\lambda   {>0}$ and a sufficiently small $\Delta t,$   {satisfying \eqref{eq:LambdaConstraint1}, it follows by} \eqref{eq:DifferenceTwoSolutionsSchemeWithoutTerminal2}   {and \eqref{eq:TheoremSummary1.0}}
 \begin{align}
 \max_{n=0,...,M-1}	\bE\left[ \big \vert \delta Y_n \big \vert^2 \right]&\leq \bE \left[\big  \vert \delta Y_M \big \vert^2 \right] \max_{n=0,...,M-1} e^{F_1(M-n)\Delta t }\nonumber \\
& {\leq \left(1+\overline{\lambda}^{-1}\right)   e^{F_1T} \bE \left[\big \vert  \widetilde{Y}_M-g(\widetilde{X}_M)\big\vert^2 \right]},
\label{eq:ThoeremSummary1} \end{align}
  {as} $F_1$   {which is defined in \eqref{eq:DefinitionConstantF} is greater than zero.}
Moreover, by \eqref{eq:BetterEstimate2}   {it holds}
\begin{align}
&\Delta t\left( (1 \wedge K_{\nu}^{-1} )- K_f \frac{1}{\lambda^2}\right) \bE \left[ \big \vert \delta  Z_n \big  \vert^2 \right]+ \Delta t\left((1 \wedge K_{\nu}^{-1} )- K_f \frac{1}{\lambda^2}\right)  \bE \left[ \big \vert \delta \Psi_n \big\vert^2 \right] \nonumber \\
&\leq \Delta t\left(1- K_f \frac{1}{\lambda^2}\right) \bE \left[ \big \vert \delta  Z_n \big \vert^2 \right]+ \Delta t\left(K_{\nu}^{-1}- K_f \frac{1}{\lambda^2}\right)  \bE \left[\big \vert \delta \Psi_n \big \vert^2 \right] \nonumber \\
&\leq {\bE} \left[ \big \vert \delta Y_{n+1} \big \vert^2 \right] -  \left(1- \Delta t K_f \left({3} \lambda^2 + \frac{1}{\lambda^{2}}\right)\right)\bE \left[ \big \vert \delta Y_n \big \vert^2 \right]. \label{eq:NeededForSecondEstimate}
\end{align}
  {For $\lambda>0$ such that additionally} $ ( 1 \wedge K_{\nu}^{-1} )- K_f \frac{1}{\lambda^2}{>} 0,$ {i.e. $\lambda^2 > (1 \vee K_{\nu})K_f$,} it follows
\begin{align}
&\Delta t \bE \left[ \big \vert \delta Z_n \big \vert^2 \right] + \Delta t  \bE \left[ \big \vert \delta \Psi_n \big \vert^2 \right] \nonumber \\
&\leq \frac{1}{( 1 \wedge K_{\nu}^{-1} )- K_f \frac{1}{\lambda^2}} \left({\bE} \left[ \big \vert \delta Y_{n+1} \big \vert^2 \right] - \left(1- \Delta t K_f \left({3} \lambda^2 + \frac{1}{\lambda^{2}}\right)\right)\bE \left[ \big \vert \delta Y_n \big \vert^2 \right]\right). \nonumber 
\end{align}
By summing up we get
\begin{align}
& \sum_{n=0}^{M-1} \bE \left[ \big \vert \delta Z_n \big \vert^2\right] \Delta t + \sum_{n=0}^{M-1} \bE \left[ \big \vert \delta \Psi_n \big \vert^2\right] \Delta t \nonumber \\
& \leq  \frac{1}{( 1 \wedge K_{\nu}^{-1} )- K_f \frac{1}{\lambda^2}} \left(\sum_{n=0}^{M-1} \bE \left[ \big \vert \delta Y_{n+1} \big \vert^2\right]  -  \left(1- \Delta t K_f \left({3} \lambda^2 + \frac{1}{\lambda^{2}}\right)\right)\sum_{n=0}^{M-1}\bE \left[ \big \vert \delta Y_n \big \vert^2 \right]  \right) \nonumber \\
&\leq  \frac{1}{( 1 \wedge K_{\nu}^{-1} )- K_f \frac{1}{\lambda^2}}  \left( 1+ T K_f \left({3} \lambda^2 + \frac{1}{\lambda^{2}}\right)  \right) \max_{n=0,...,M}	\bE \left[ \big \vert \delta Y_{n} \big \vert^2 \right] \nonumber \\
& \leq   \frac{1}{( 1 \wedge K_{\nu}^{-1} )- K_f \frac{1}{\lambda^2}} \left( 1+T K_f \left({3} \lambda^2 + \frac{1}{\lambda^{2}}\right) \right) \left(1+\overline{\lambda}^{-1}\right)  {e^{F_1T} } \bE \left[\big \vert  \widetilde{Y}_M-g(\widetilde{X}_M)\big\vert^2 \right],\label{eq:ThoeremSummary2}
\end{align}
  {where we use} \eqref{eq:ThoeremSummary1} in the last inequality.
 Note by applying   {l'H\^{o}pital's rule it holds} 
 \begin{align}
\lim_{\Delta t \to 0} F_1 = \lim_{\Delta t \to 0} \frac{K_f \left({3} \lambda^2 + \frac{1}{\lambda^{2}}\right)}{1-\Delta t K_f \left({3} \lambda^2 + \frac{1}{\lambda^{2}}\right)} = K_f \left({3}\lambda^2 + \frac{1}{\lambda^{2}}\right),
 \end{align}
and as $F_1$ is increasing in $\Delta t$, it follows that for $\Delta t$ sufficiently small $F_1 \leq K_f \left({3} \lambda^2 + \frac{1}{\lambda^{2}}\right)+ a$ for any constant $a>0.$
Combining \eqref{eq:ThoeremSummary1}, \eqref{eq:ThoeremSummary2} with \eqref{eq:ErrorEstimateFBSDEGoal1}, \eqref{eq:ErrorEstimateFBSDEGoal2} yields
\begin{align}
 &\max_{n=0,1,...,M-1} \mathbb{E} \left[ \sup_{t \in [t_n,t_{n+1}]} \big \vert X_t - \widetilde{X}_n \big \vert^2 \right]+ \max_{n=0,1,...,M-1} \mathbb{E} \left[ \sup_{t \in [t_n,t_{n+1}]} \big \vert Y_t - \widetilde{Y}_n \big \vert^2 \right] \nonumber \\
 & \quad  + \mathbb{E} \left[ \sum_{n=0}^{M-1}\int_{t_n}^{t_{n+1}} \big \vert Z_t - \widetilde{Z}_n \big \vert^2 \textnormal{d}t \right]+  \mathbb{E} \left[ \sum_{n=0}^{M-1}\int_{t_n}^{t_{n+1}}\big \vert \Psi_t - \widetilde{\Psi}_n \big \vert^2 \textnormal{d}t \right] \nonumber \\ 
 &\leq  C \Delta t+
  2  \max_{n=0,1,...,M-1} \mathbb{E} \left[ \big \vert { \delta X_n} \big \vert^2 \right]+ 2\max_{n=0,1,...,M-1} \mathbb{E} \left[  \big \vert { \delta Y_n}\big \vert^2 \right] \nonumber \\
 & \quad  + 2 \Delta t  \sum_{n=0}^{M-1} \mathbb{E} \left[\big \vert { \delta Z_n}\big \vert^2\right] +  2 \Delta t  \sum_{n=0}^{M-1}\mathbb{E}\left[\big \vert { \delta \Psi_n} \big \vert^2\right]  \nonumber \\
 & \leq  C \Delta t+ 2\left(1+\overline{\lambda}^{-1}\right) e^{\overline{F}_1 T}  \bE \left[\big \vert  \widetilde{Y}_M-g(\widetilde{X}_M)\big\vert^2 \right]\nonumber \\
 & \quad + {2}\frac{1}{( 1 \wedge K_{\nu}^{-1})- K_f \frac{1}{\lambda^2}} \left( 1+T K_f \left({3} \lambda^2 + \frac{1}{\lambda^{2}}\right) \right) \left(1+\overline{\lambda}^{-1}\right)  {e^{\overline{F}_1T} }  \bE \left[\big \vert  \widetilde{Y}_M-g(\widetilde{X}_M)\big\vert^2 \right] \nonumber \\
 &=C \Delta t+ {C(\lambda,\overline{\lambda})}\bE \left[\big \vert  \widetilde{Y}_M-g(\widetilde{X}_M)\big\vert^2 \right].
 \end{align} 
 \end{proof}

\subsection{A posteriori error estimate for the infinite activity case} \label{sec:InfiniteActivityCase}
In this section we consider a forward process $X$ in \eqref{eq:SDE} with {possibly} infinitely many jumps, i.e. we drop Assumption \ref{assump:FiniteActivity}. To do so, we proceed as in Section 3.2 in \cite{gnoatto_patacc_picarelli}. 
{More specifically, we approximate $X$ with a jump diffusion process $X^{\epsilon}$ with finitely many jumps. By doing so the big jumps correspond to those of a compound Poisson process, while the small jumps will be approximated via the diffusion process by increasing the volatility. Such a procedure has also been applied among others in \cite{asmussen_rosinski_2001}, \cite{cohen_rosinski_2007} and \cite{jum_2015}.}
This means for $\epsilon \in (0,1]$ we define
\begin{align}
\nu^{\epsilon}(\textnormal{d}z)&:= \textbf{1}_{\lbrace \vert z \vert> \epsilon\rbrace}\nu (\textnormal{d}z), \nonumber \\
\nu_{\epsilon}(\textnormal{d}z)&:= \textbf{1}_{\lbrace \vert z \vert \leq \epsilon\rbrace}\nu (\textnormal{d}z), \nonumber
\end{align}
such that $\nu=\nu_{\epsilon}+\nu^{\epsilon}.$ We introduce the notation $K_{\nu^{\epsilon}}:= \nu^{\epsilon}(\mathbb{R}^d).$ \\
We denote the compensators of $N_{\epsilon}$ and $N^{\epsilon}$ by $\nu_{\epsilon}$ and $\nu^{\epsilon}$, respectively, which means that $N=N_{\epsilon}+ N^{\epsilon}$.
We introduce the quantity
$$
\Sigma_{\epsilon}:=\int_{\bR^d} zz^{\top}\nu_{\epsilon}(\textnormal{d}z)= \int_{\vert z \vert < \epsilon} zz^{\top}\nu(\textnormal{d}z),
$$
which takes values in the cone of positive semidefinite $d\times d$ matrices $S_+^d.$
For every $(t,x) \in [0,T] \times \bR^d,$ we approximate the solution $X^{t,x}$ of \eqref{eq:SDE} by the process $X^{\epsilon,t,x}$ which is given by
\begin{align} \label{eq:SDEApproximation}
	X_s^{\epsilon, t,x}=x + \int_t^s b(X_{r-}^{\epsilon,t,x}) \textnormal{d}r + \int_t^s \sigma_\epsilon(X_{r-}^{\epsilon,t,x} )^{\top}  \textnormal{d}W_r + \int_t^s \int_{\mathbb{R}^d} \gamma (X_{r-}^{\epsilon,t,x})z \widetilde{N}^{\epsilon}(\textnormal{d}r,\textnormal{d}z), \quad {s \in [t,T],}
    \end{align}
where $\widetilde{N}^{\epsilon}:=N^{\epsilon}(\textnormal{d}r, \textnormal{d}z)- \nu^{\epsilon}(\textnormal{d}z) \textnormal{d}r$ and the function ${\sigma}_\epsilon: \bR^d \to \bR^{d \times d}$ is defined by
\begin{equation*}
{\sigma_\epsilon}(x):= \sigma(x)+ \gamma(x)\sqrt{ \Sigma_{\epsilon}}  {,}
\end{equation*} 
where $\sqrt{\Sigma_{\epsilon}} $ is the square root of ${\Sigma_{\epsilon}} \in {{S}}_{+}^d$.
As in the previous section, we omit the dependence on $x$ in the notation, {set $t=0$} and {do not keep track of constants mentioning, time by time, only the quantities they depend on.}
We consider the following backward dynamics given by
\begin{align} \label{eq:BSDEApproximation}
	Y_t^{\epsilon}&=g(X_T^{\epsilon})+\int_t^T f\left(  {r},X_{r-}^{\epsilon}, Y_{r-}^{\epsilon}, Z_r^{\epsilon}, \int_{\mathbb{R}^d} U_r^{\epsilon}(z) \nu^{\epsilon} (\textnormal{d}z) \right)\textnormal{d}r - \int_t^T (Z_r^{\epsilon})^{\top}\textnormal{d}W_r  \nonumber \\
	& \quad - \int_t^T \int_{\mathbb{R}^d} U_r^{\epsilon}(z) \widetilde{N}^{\epsilon} (\textnormal{d}r,\textnormal{d}z), \quad {t \in [0,T]}
\end{align}
where the solution $(Y^{\epsilon}, Z^{\epsilon},U^{\epsilon})$ is an approximation of $(Y,Z,U)$ in \eqref{eq:BSDE}. As before, we introduce the process $\Psi^{\epsilon}=(\Psi_t^{\epsilon})_{t \in [0,T]}$ by
\begin{equation} \label{eq:IntegralApproximation}
\Psi_t^{\epsilon}:=\int_{\bR^d} U^{\epsilon}_t(z)\nu^{\epsilon}(\textnormal{d}z), \quad t \in [0,T]. 
\end{equation}
{The following a priori error estimates for the forward and backward approximation are derived  in \cite[Proposition 4.2 and 4.3]{gnoatto_patacc_picarelli}.\\

\begin{proposition}\label{prop:ErrorEstimatesApproximationAlessandroAthena}
Let Assumptions \ref{assump:ExistenceUniquenessFBSDE} \textbf{(A1)}, \textbf{(A3)} hold.
Let $(X,Y,Z,U)$ and $(X^{\epsilon},Y^{\epsilon},Z^{\epsilon},U^{\epsilon} )$ be the solution of the FBSDE \eqref{eq:SDE}-\eqref{eq:BSDE} and \eqref{eq:SDEApproximation}-\eqref{eq:BSDEApproximation}, respectively. Let $\Psi, \Psi^{\epsilon}$ be defined in \eqref{eq:Integral}, \eqref{eq:IntegralApproximation}, respectively. Then there exists {a} constant ${C}>0$ {that only depends on the data {of the FBSDE in \eqref{eq:SDE}-\eqref{eq:BSDE}}} such that
\begin{align}
    \mathbb{E}\left[\sup _{t \in[0, T]}\big \vert X_t-X_t^\epsilon\big \vert^2\right] + \bE \left[ \sup_{t \in [0,T]} \big \vert Y_t - Y_t^{ \epsilon} \big  \vert^2 \right]+ \bE \left[ \int_0^T \vert Z_t - Z_t^{\epsilon}\vert^2 \textnormal{d}t \right] \\
    \quad +\,  \bE \left[ \int_0^T \int_{\bR^d} \big \vert U_t(z)-U_t^{\epsilon}(z) \big \vert^2 \nu^{\epsilon}(\textnormal{d}z) \textnormal{d}t\right] \nonumber\leq {C}\left(1+ \vert x \vert^2 \right)\int_{\bR^d} \vert z \vert^2 \nu_{\epsilon} (\textnormal{d}z).
\end{align}
\end{proposition}} 

The goal is to approximate $(X^{\epsilon},Y^{\epsilon},Z^{\epsilon},\Psi^{\epsilon} )$ by $(\widetilde{X}^{ \epsilon}_n,\widetilde{Y}^{\epsilon}_n,\widetilde{Z}^{\epsilon}_n,\widetilde{\Psi}^{\epsilon}_n)_{n=0,...,M}$, {which is defined by the following scheme, analogous to  \eqref{eq:SchemeDeepSolverOneNeuralNetworkRewritten1} for the FBSDE \eqref{eq:SDEApproximation}-\eqref{eq:BSDEApproximation} {for $n=0,...,M-1$}
\begin{align} \label{eq:SchemeDeepSolverOneNeuralNetwork_eps}
\begin{cases}
	\widetilde{X}^\epsilon_{n+1}&=\widetilde{X}^\epsilon_n +b(\widetilde{X}^\epsilon_n) \Delta t+ \sigma_\epsilon(\widetilde{X}^\epsilon_n)^{\top}\Delta W_n + \int_{t_n}^{t_{n+1}}\int_{\mathbb{R}^d} \gamma(\widetilde{X}^\epsilon_n) z \widetilde{N}^\epsilon(\textnormal{d}r,\textnormal{d}z), \\
	\widetilde{X}^\epsilon_0&=x,\\
	\widetilde{Z}^\epsilon_n&= \sigma_\epsilon(\widetilde{X}^\epsilon_n)^{\top} D_x \mathcal{U}_{n}^{\epsilon,\rho_n}(\widetilde{X}^\epsilon_n),\\
	\widetilde{\Psi}^\epsilon_n&= \int_{\bR^d}\left(\cU_n^{\epsilon,\rho_n}(\widetilde{X}^\epsilon_n +\gamma(\widetilde{X}^\epsilon_n) z )- \cU_n^{\epsilon,\rho_n}(\widetilde{X}^\epsilon_n)\right) \nu^\epsilon(\textnormal{d}z) \\
	\widetilde{Y}_{n+1}^{\epsilon}&= \widetilde{Y}_n^{\epsilon}- f\left( t_n, \widetilde{X}^\epsilon_n, \widetilde{Y}_n^{\epsilon}, \widetilde{Z}^{\epsilon}_n, \widetilde{\Psi}^{\epsilon}_n \right) \Delta t  +\widetilde{Z}_n^{\top} \Delta W_n   \\
	&\quad +\left( \cU_n^{\epsilon,\rho_n}\left(\widetilde{X}_n + \gamma(\widetilde{X}^\epsilon_n)\sum_{i=N^\epsilon([0,t_n], \mathbb{R}^d)+1}^{N^\epsilon([0,t_{n+1}], \mathbb{R}^d)}  z_i \right)- \cU_n^{\epsilon,\rho_n}(\widetilde{X}^\epsilon_n) \right)	 - \int_{t_n}^{t_{n+1}} \widetilde{\Psi}^\epsilon_n \textnormal{d}t,  \\
	\widetilde{Y}_0^{\epsilon}&=y.
\end{cases}	
\end{align}
As before, we set $\widetilde{Z}_M^{\epsilon}=\widetilde{\Psi}_M^{\epsilon}=0$ to simplify the notation.
}
In order to use the estimates in Section \ref{sec:LiteratureSchemesFBSDEsJumps} we need to verify that the FBSDE in \eqref{eq:SDEApproximation}-\eqref{eq:BSDEApproximation} satisfies the assumptions introduced in the previous section. 
As $\sigma$ and $\gamma$ are Lipschitz continuous functions it follows that $\sigma_\epsilon$ is also Lipschitz continuous with {the Lipschitz constant $K_{\sigma_\epsilon}:=K_\sigma+ K_\gamma|\sqrt{\Sigma_\epsilon}|.$} Obviously, all the other conditions in Assumption \ref{assump:ExistenceUniquenessFBSDE} \textbf{(A1)} are satisfied. Furthermore, by construction 
we consider the finite activity case and thus Assumption \ref{assump:FiniteActivity} holds. \\
In the following, we say that a generic constant $C_{\epsilon}\in \mathbb R$ only depends on the data of the FBSDE  \eqref{eq:SDEApproximation}-\eqref{eq:BSDEApproximation} if $C_\epsilon$ only depends on $b, \sigma^{\epsilon}, \Gamma, f, g, T, x$ and ${K_{\nu^{\epsilon}}}$.
\begin{proposition}
\label{prop:InfiniteActivityEstimate}
Let Assumptions \ref{assump:ExistenceUniquenessFBSDE} \textbf{(A1)}-\textbf{(A3)} and \ref{assump:NeuralNetwork} hold true. Let  $(X^{\epsilon},Y^{\epsilon},Z^{\epsilon}, {U}^{\epsilon} )$ be the solution {of the FBSDE} in \eqref{eq:SDEApproximation}-\eqref{eq:BSDEApproximation} and let $\Psi^{\epsilon}$ be defined by \eqref{eq:IntegralApproximation}. Let $(\widetilde{X}^{\epsilon}_n,\widetilde{Y}^{\epsilon}_n,\widetilde{Z}^{\epsilon}_n,\widetilde{\Psi}^{\epsilon}_n)_{n=0,...,M}$ be the approximation of $(X^{\epsilon},Y^{\epsilon},Z^{\epsilon},\Psi^{\epsilon} )$,  defined by \eqref{eq:SchemeDeepSolverOneNeuralNetwork_eps}. Let $\overline{\lambda}>0$ and choose $\lambda>0 $ and a sufficiently small $\Delta t$ such that
\begin{equation} \label{eq:LambdaConstraint1InfiniteActivity}
\lambda^2 >  (1 \vee K_{\nu^{\epsilon}})K_f \vee  \frac{1-\sqrt{1-{12} (\Delta t)^2 K_f^2}}{{6} \Delta t K_f}, \quad \text{and}\quad 
\lambda^2 \leq \frac{1+\sqrt{1-{12} (\Delta t)^2 K_f^2}}{{6} \Delta t K_f}.
\end{equation} 
Then it holds
\begin{align*}
 \max_{n=0,1,...,M-1} \mathbb{E} \left[ \sup_{t \in [t_n,t_{n+1}]} \big \vert X^\epsilon_t - \widetilde{X}_n^{\epsilon} \big \vert^2 \right] \leq C_\epsilon \Delta t 
\end{align*}
and
\begin{align}
&\max_{n=0,1,...,M-1} \mathbb{E} \left[ \sup_{t \in [t_n,t_{n+1}]} \big \vert Y^\epsilon_t - \widetilde{Y}_n^{\epsilon} \big \vert^2 \right]\nonumber   + \mathbb{E} \left[ \sum_{n=0}^{M-1}\int_{t_n}^{t_{n+1}} \big \vert Z^\epsilon_t - \widetilde{Z}_n^{\epsilon} \big \vert^2 \textnormal{d}t \right] +  \mathbb{E} \left[ \sum_{n=0}^{M-1}\int_{t_n}^{t_{n+1}}\big \vert \Psi^\epsilon_t - \widetilde{\Psi}_n^{\epsilon} \big \vert^2 \textnormal{d}t \right] \nonumber 
\\
&\leq C_\epsilon  \Delta t+ \overline{C}_\epsilon(\lambda,\overline{\lambda}) \bE \left[\big \vert  \widetilde{Y}_M^{\epsilon}-g(\widetilde{X}_M^{ \epsilon})\big\vert^2 \right],  \nonumber 
\end{align}
where $C_\epsilon>0$ denotes a constant only depending on the data of the FBSDE in \eqref{eq:SDEApproximation}-\eqref{eq:BSDEApproximation} and  
\begin{align}
\overline{C}_\epsilon(\lambda,\overline{\lambda}):= {2}\left(1+\overline{\lambda}^{-1}\right)  {e^{\overline{F}_1T} } \left[ {1}+ \frac{1+T K_f ({3} \lambda^2 + \frac{1}{\lambda^{2}})}{ \left ( 1 \wedge K_{\nu^\epsilon}^{-1} \right )- K_f \frac{1}{\lambda^2}} \right],\label{def:barCeps}
\end{align}
with $\overline{F}_1= K_f \left({3} \lambda^2 + \frac{1}{\lambda^{2}}\right)+ \overline{\lambda}$.
\end{proposition} 

\begin{proof}
The result follows by applying  Theorem \ref{theorem:ErrorEstiamteFiniteActivity}  to the truncated FBSDE \eqref{eq:SDEApproximation}-\eqref{eq:BSDEApproximation}. 
\end{proof}

We can now combine the results above in order to estimate the difference between the exact solution of the FBSDE \eqref{eq:SDE}-\eqref{eq:BSDE} and its approximation given by  \eqref{eq:SchemeDeepSolverOneNeuralNetwork_eps}.

 \begin{theorem} \label{teo:InfiniteActivityEstimate}
Let Assumptions \ref{assump:ExistenceUniquenessFBSDE} \textbf{(A1)}-\textbf{(A3)} and \ref{assump:NeuralNetwork} hold true. Let $(X,Y,Z,U)$ be the solution of the FBSDE \eqref{eq:SDE}-\eqref{eq:BSDE} and $\Psi$ be defined in \eqref{eq:Integral}. Let $(\widetilde{X}^{ \epsilon}_n,\widetilde{Y}^{\epsilon}_n,\widetilde{Z}^{\epsilon}_n,\widetilde{\Psi}^{\epsilon}_n)_{n=0,...,M}$  be the solution of \eqref{eq:SchemeDeepSolverOneNeuralNetwork_eps}. Let $\overline{\lambda}>0$ and choose $\lambda>0 $ and a sufficiently small $\Delta t$ such that \eqref{eq:LambdaConstraint1InfiniteActivity} holds. Then, 
 \begin{align}
     \max_{n=0,1,...,M-1} \mathbb{E} \left[ \sup_{t \in [t_n,t_{n+1}]} \big \vert X_t - \widetilde{X}_n^{\epsilon} \big \vert^2 \right] 
     \leq C\left( 1+\vert x \vert^2 \right) \int_{\bR^d} \vert z \vert^2 \nu_{\epsilon}(\textnormal{d}z)+C_\epsilon\Delta t, \nonumber
 \end{align}
 and
 \begin{align}
 &\max_{n=0,1,...,M-1} \mathbb{E} \left[ \sup_{t \in [t_n,t_{n+1}]} \big \vert Y_t - \widetilde{Y}_n^{\epsilon} \big \vert^2 \right] + \mathbb{E} \left[ \sum_{n=0}^{M-1}\int_{t_n}^{t_{n+1}} \big \vert Z_t - \widetilde{Z}_n^{\epsilon} \big \vert^2 \textnormal{d}t \right] +  \mathbb{E} \left[ \sum_{n=0}^{M-1}\int_{t_n}^{t_{n+1}}\big \vert \Psi_t - \widetilde{\Psi}_n^{\epsilon} \big \vert^2 \textnormal{d}t \right] \nonumber \\
 & \leq C_\epsilon\left(\Delta t + \int_{\bR^d} \vert z \vert^2 \nu_{\epsilon}(\textnormal{d}z)\right) + 2 \overline{C}_\epsilon(\lambda,\overline{\lambda}) \bE \left[\big \vert  \widetilde{Y}_M^{{\epsilon}}-g(\widetilde{X}_M^{{ \epsilon}})\big\vert^2 \right], \nonumber
 \end{align}
 where $\overline{C}_\epsilon$ is defined by \eqref{def:barCeps} and $C_\epsilon>0$ only depends on the data of the FBSDE in \eqref{eq:SDEApproximation}-\eqref{eq:BSDEApproximation}.
 \end{theorem}

\begin{proof}
  By combining Propositions \ref{prop:ErrorEstimatesApproximationAlessandroAthena} and \ref{prop:InfiniteActivityEstimate}  it easily follows
    \begin{align}
    \max_{n=0,1,...,M-1} \mathbb{E} \left[ \sup_{t \in [t_n,t_{n+1}]} \big \vert X_t - \widetilde{X}_n^{\epsilon} \big \vert^2 \right] 
    & \leq  { C}\left( 1+\vert x \vert^2 \right) \int_{\bR^d} \vert z \vert^2 \nu_{\epsilon}(\textnormal{d}z)+ C_\epsilon \Delta t.\label{eq:DistanceInfiniteActivityX}
\end{align}
We  also have
 \begin{align}
    & \mathbb{E} \left[ \sum_{n=0}^{M-1}\int_{t_n}^{t_{n+1}}\big \vert \Psi_t - \widetilde{\Psi}_n^{\epsilon} \big \vert^2 \textnormal{d}t \right]\nonumber \\
    & \leq 2 \mathbb{E} \left[\int_{0}^{T}\Big \vert \int_{\bR^d} U_t(z) \nu(dz) - \int_{\bR^d} U_t^{\epsilon} (z) \nu^{\epsilon} (\textnormal{d}z)\Big \vert^2 \textnormal{d}t \right]+2 \mathbb{E} \left[ \sum_{n=0}^{M-1}\int_{t_n}^{t_{n+1}}\big \vert  \Psi_{t}^{\epsilon}-  \widetilde{\Psi}_n^{\epsilon} \big \vert^2 \textnormal{d}t \right] \nonumber \\
    & \leq 4 \mathbb{E} \left[\int_{0}^{T}\Big \vert \int_{\bR^d} \left(U_t(z)-U_t^{\epsilon}(z) \right) \nu^{\epsilon}(\textnormal{d}z)\Big \vert^2 \textnormal{d}t \right] + 4 \mathbb{E} \left[\int_{0}^{T}\Big \vert \int_{\bR^d} U_t(z)  \nu_{\epsilon}(\textnormal{d}z)\Big \vert^2 \textnormal{d}t \right]\nonumber \\
    & \quad +2 \mathbb{E} \left[ \sum_{n=0}^{M-1}\int_{t_n}^{t_{n+1}}\big \vert  \Psi_{t}^{\epsilon}-  \widetilde{\Psi}_n^{\epsilon} \big \vert^2 \textnormal{d}t \right]\nonumber \\
 & \leq 4 \mathbb{E} \left[\int_{0}^{T}\Big \vert \int_{\bR^d} \left(U_t(z)-U_t^{\epsilon}(z) \right) \nu^{\epsilon}(\textnormal{d}z)\Big \vert^2 \textnormal{d}t \right] + 4 \mathbb{E} \left[\int_{0}^{T} \int_{\bR^d} \vert U_t(z)\vert^2  \nu_{\epsilon}(\textnormal{d}z) \textnormal{d}t \right]\nonumber \\
    & \quad +2 \mathbb{E} \left[ \sum_{n=0}^{M-1}\int_{t_n}^{t_{n+1}}\big \vert  \Psi_{t}^{\epsilon}-  \widetilde{\Psi}_n^{\epsilon} \big \vert^2 \textnormal{d}t \right],\label{eq:DistanceInfiniteActivityProof2}
 \end{align}
 by using Jensen's inequality for $\sigma$-finite measures, see Corollary 23.10 in \cite{schilling_book_2005}, in the last inequality. 
 {By Theorem \ref{theorem:LinkPIDE} and \eqref{eq:EstimatesRegularityPIDE} for $t \in [0,T]$ and $z \in \bR^d,$ $U_t(z)$ has the following representation}
 \begin{align*}
    U_t(z)=u(t,X_{t-}+ \gamma(X_{t-})z)-u(t,X_{t-}),
 \end{align*}    
 where $u$ is Lipschitz continuous with $\vert u(t,x) \vert \leq C(1+\vert x \vert )$, for some constant $C>0$. \\
 Thus, we get
\begin{align} 
\mathbb{E} \left[\int_{0}^{T} \int_{\bR^d} \vert U_t(z)\vert^2  \nu_{\epsilon}(\textnormal{d}z) \textnormal{d}t \right]
& \leq C  \mathbb{E} \left[\int_{0}^{T} \int_{\bR^d} (1+ \vert X_{t-}\vert^2 ) \vert z \vert^2  \nu_{\epsilon}(\textnormal{d}z) \textnormal{d}t \right] \nonumber \\
& \leq C \int_{\bR^d} \vert z \vert^2 \nu_{\epsilon}(\textnormal{d}z) \bE \left[ \int_0^T \left( 1+ \sup_{s \in [0,T]} \vert X_{s-} \vert^2  \right)\textnormal{d}t \right] \nonumber \\
& \leq C (1+ \vert x \vert^2)\int_{\bR^d} \vert z \vert^2 \nu_{\epsilon}(\textnormal{d}z)  \label{eq:DistanceInfiniteActivityProof2a},
\end{align}
where we use Theorem 4.1.1 in \cite{delong}   {in the last inequality}.
By the Cauchy-Schwarz inequality we also have
\begin{align}
    \mathbb{E} \left[\int_{0}^{T}\Big \vert \int_{\bR^d} \left(U_t(z)-U_t^{\epsilon}(z) \right) \nu^{\epsilon}(\textnormal{d}z)\Big \vert^2 \textnormal{d}t \right]\leq T K_{\nu^\epsilon} \mathbb{E} \left[\int_{0}^{T} \int_{\bR^d}\big \vert  U_t(z)-U_t^{\epsilon}(z) \big \vert^2 \nu^{\epsilon}(\textnormal{d}z)\textnormal{d}t \right]. \label{eq:DistanceInfiniteActivityProof4}
\end{align}
Thus, by   {putting together} \eqref{eq:DistanceInfiniteActivityProof2}, \eqref{eq:DistanceInfiniteActivityProof2a}, \eqref{eq:DistanceInfiniteActivityProof4},  it follows 
\begin{align}
& \mathbb{E} \left[ \sum_{n=0}^{M-1}\int_{t_n}^{t_{n+1}}\big \vert \Psi_t - \widetilde{\Psi}_n^{\epsilon} \big \vert^2 \textnormal{d}t \right] \nonumber \\
& \leq 4 T K_{\nu^\epsilon} \mathbb{E} \left[\int_{0}^{T} \int_{\bR^d}\big \vert  U_t(z)-U_t^{\epsilon}(z) \big \vert^2 \nu^{\epsilon}(\textnormal{d}z)\textnormal{d}t \right] + 4 C (1+ \vert x \vert^2)\int_{\bR^d} \vert z \vert^2 \nu_{\epsilon}(\textnormal{d}z) \nonumber \\
& \quad + 2 \mathbb{E} \left[ \sum_{n=0}^{M-1}\int_{t_n}^{t_{n+1}}\big \vert  \Psi_{t}^{\epsilon}-  \widetilde{\Psi}_n^{\epsilon} \big \vert^2 \textnormal{d}t \right]\nonumber 
\end{align}
and then  
\begin{align}
&\max_{n=0,1,...,M-1} \mathbb{E} \left[ \sup_{t \in [t_n,t_{n+1}]} \big \vert Y_t - \widetilde{Y}_n^{\epsilon} \big \vert^2 \right] + \mathbb{E} \left[ \sum_{n=0}^{M-1}\int_{t_n}^{t_{n+1}} \big \vert Z_t - \widetilde{Z}_n^{\epsilon} \big \vert^2 \textnormal{d}t \right] + \mathbb{E} \left[ \sum_{n=0}^{M-1}\int_{t_n}^{t_{n+1}}\big \vert \Psi_t - \widetilde{\Psi}_n^{\epsilon} \big \vert^2 \textnormal{d}t \right]\nonumber \\
& \leq 2 \mathbb{E} \left[ \sup_{  {[}t \in 0,T]} \big \vert Y_t - {Y}_t^{ \epsilon} \big \vert^2 \right]  + 2 \bE \left[\int_0^T \vert Z_t- Z_t^{\epsilon} \vert^2 \textnormal{d}t \right] + 4 T K_{\nu^\epsilon} \mathbb{E} \left[\int_{0}^{T} \int_{\bR^d}\big \vert  U_t(z)-U_t^{\epsilon}(z) \big \vert^2 \nu^{\epsilon}(\textnormal{d}z)\textnormal{d}t \right]\nonumber \\
& \quad + \max_{n=0,1,...,M-1} 2 \mathbb{E} \left[ \sup_{t \in [t_n,t_{n+1}]} \big \vert Y_t^{\epsilon} - \widetilde{Y}_n^{\epsilon} \big \vert^2 \right] + 2 \mathbb{E} \left[ \sum_{n=0}^{M-1}\int_{t_n}^{t_{n+1}} \big \vert Z_t^{\epsilon} - \widetilde{Z}_n^{\epsilon} \big \vert^2 \textnormal{d}t \right]   \nonumber\\
& \quad + 2 \mathbb{E} \left[ \sum_{n=0}^{M-1}\int_{t_n}^{t_{n+1}}\big \vert  \Psi_{t}^{\epsilon}-  \widetilde{\Psi}_n^{\epsilon} \big \vert^2 \textnormal{d}t \right] + 4 C (1+ \vert x \vert^2)\int_{\bR^d} \vert z \vert^2 \nu_{\epsilon}(\textnormal{d}z)\nonumber \\
& \leq C_\epsilon \Delta t+ 2 \overline{C}_\epsilon(\lambda,\overline{\lambda}) \bE \left[\big \vert  \widetilde{Y}_M^{\epsilon}-g(\widetilde{X}_M^{ \epsilon})\big\vert^2 \right] + C_\epsilon\left( 1+\vert x \vert^2 \right) \int_{\bR^d} \vert z \vert^2 \nu_{\epsilon}(\textnormal{d}z),\nonumber 
\end{align}
where the last inequality is obtained by applying Propositions \ref{prop:ErrorEstimatesApproximationAlessandroAthena} and \ref{prop:InfiniteActivityEstimate}, and $C_\epsilon>0$.
\end{proof}

\section{A priori estimates}\label{sec:aPrioriEstimates}
\subsection{Finite activity case}
{As for the a posteriori estimate,} we first consider the finite activity case. We now aim to find an estimate for the term $$\mathbb{E} \left[ \big \vert \widetilde{Y}_M - g(\widetilde{X}_M) \big \vert^2 \right],$$
where $\widetilde{X}_M,\widetilde{Y}_M$ are defined in  \eqref{eq:SchemeDeepSolverOneNeuralNetworkRewritten1}.
In particular, under Assumption \ref{assump:FiniteActivity} the following result regarding the path regularity of the FBSDE in \eqref{eq:SDE}-\eqref{eq:BSDE} is derived in Theorem 2.1 (i) and (ii) in \cite{bouchard_elie_2008}.
\begin{theorem} \label{theorem:Bouchard_Elie_Path_regularity}
 Let Assumptions \ref{assump:ExistenceUniquenessFBSDE} \textbf{(A1)}, \textbf{(A2)} and \ref{assump:FiniteActivity} hold. Let $(X,Y,Z,U) \in \bS_{[0,T]}^2(\bR^d) \times \bS_{[0,T]}^2(\bR) \times \bH_{[0,T]}^2(\bR^d) \times \bH_{[0,T],N}^2(\bR) $ be the unique solution to the FBSDE in \eqref{eq:SDE}-\eqref{eq:BSDE},   {and $\Psi$ given in \eqref{eq:Integral}.} Then the following path regularity holds
 \begin{align}
    & \max_{n=0,...,M-1} \bE \left[ \sup_{t \in [t_n, t_{n+1}]} \big \vert X_t - \widetilde{\widetilde{X}}_t \big\vert^2\right]+ \max_{n=0,...,M-1} \bE \left[ \sup_{t \in [t_n, t_{n+1}]} \big\vert Y_t - \widetilde{\widetilde{Y}}_t \big\vert^2\right] \nonumber \\
     & \quad +   {\mathbb{E} \left[ \int_0^T  \big \vert {Z_t} - \widetilde{\widetilde{Z}}_t\big \vert^2 \textnormal{d}t\right]} +   {\mathbb{E} \left[ \int_0^T  \big \vert \Psi_t - \widetilde{\widetilde{\Psi}}_t\big \vert^2 \textnormal{d}t\right]} \leq C \Delta t, \label{eq:PathregularityEstimate}
 \end{align}
 where $C$ is a constant only depending on the data of the FBSDE in \eqref{eq:SDE}-\eqref{eq:BSDE} and the processes $\widetilde{\widetilde{X}}, \widetilde{\widetilde{Y}}, \widetilde{\widetilde{Z}},\widetilde{\widetilde{\Psi}}$ are defined on $[t_n,t_{n+1})$ by
 \begin{align} \label{eq:ProcessesBouchardElie}
 \widetilde{\widetilde{X}}_t:= X_n, \quad  \widetilde{\widetilde{Y}}_t:= Y_n, \quad \widetilde{\widetilde{Z}}_t:=\frac{1}{\Delta t}   \bE \left[ \int_{t_n}^{t_{n+1}} Z_s \textnormal{d}s \Big \vert \cF_{n}\right], \quad \widetilde{\widetilde{\Psi}}_t:=\frac{1}{\Delta t} \bE \left[ \int_{t_n}^{t_{n+1}} \Psi_s \textnormal{d}s \Big \vert \cF_{n}\right]
 \end{align}
 with $X_n:=X_{t_n}$ and $Y_n:=Y_{t_n}$ for $n=0,...,  {M-1}.$
\end{theorem}

As in the previous sections, we omit the dependence on $(t,x)$ in the notation and do not keep track of constants mentioning, time by time, only the quantities they depend on. In particular, as before $C$ denotes a constant depending only on the data of the FBSDE in \eqref{eq:SDE}-\eqref{eq:BSDE} and $C(\cdot)$ is a constant depending on additional arguments, which are specified in the bracket.

\begin{proposition} \label{LemmaSecondEstimateAuxiliary}
Let Assumptions \ref{assump:ExistenceUniquenessFBSDE} \textbf{(A1)}-\textbf{(A3)}, \ref{assump:FiniteActivity} and \ref{assump:NeuralNetwork} hold. Let $(\widetilde{X}_n, \widetilde{Y}_n, \widetilde{Z}_n, \widetilde{\Psi}_n)_{ n=0,..., M}$ be defined in \eqref{eq:SchemeDeepSolverOneNeuralNetworkRewritten1} and   {$ \widetilde{\widetilde{Z}}_{n}:= \widetilde{\widetilde{Z}}_{t_n}$ and $ \widetilde{\widetilde{\Psi}}_{n}:= \widetilde{\widetilde{\Psi}}_{t_n}$ defined in \eqref{eq:ProcessesBouchardElie}}. {Given $\lambda_3>0$ there exists a constant $C(\lambda_3)$ such that for sufficiently small $\Delta t$ it holds
  \begin{align*}
   \mathbb{E}\left[ \big  \vert g(\widetilde{X}_M)-\widetilde{Y}_M \big \vert^2 \right] &\leq   {\bar H} (1+\lambda_3) \Delta t \sum_{n=0}^{M-1}\left(\bE \left[ \big \vert  \widetilde{\widetilde{Z}}_{n}- \widetilde{Z}_{n} \big \vert^2 \right] + \bE \left[ \big \vert  \widetilde{\widetilde{\Psi}}_{n}- \widetilde{\Psi}_{n} \big \vert^2 \right]\right) \nonumber \\
& \quad + (1+K_g)^2C(\lambda_3) \left(\big \vert Y_0-y \big \vert^2 +  \Delta t+ \textnormal{Error}^{\rho}\right),
  \end{align*}}
where 
\begin{align}
\textnormal{Error}^{\rho}&:=\max_{n=0,...,M-1}\bigg(\bE \left[   \int_{\bR^d}\left(  u(t_n, \widetilde{X}_n + \gamma(\widetilde{X}_n)z) -\cU_n^{\rho_n}(\widetilde{X}_n+\gamma(\widetilde{X}_n)z)\right)^2\nu(\textnormal{d}z)\right] \nonumber \\
& \quad \quad + K_{\nu} \bE \left[\left(  u(t_n, \widetilde{X}_n) -\cU_n^{\rho_n}(\widetilde{X}_n)\right)^2\right]\bigg) \label{eq:DefiError}
\end{align}
and $  {\bar H} :=\min_{x \in \mathbb{R}_+} H(x)$ for 
\begin{equation} \label{eq:DefinitionFunctionH}
H(x):=(1+K_g)^2\left(1 + 5(K_f)^2 {x}^{-1} \right) e^{T\left( x+5   {\bar K} ^2 x^{-1}+  {\bar K} ^2 \max \lbrace{ x^{-1} + 1+ \int_{\bR^d} \vert z \vert^2 \nu(\textnormal{d}z ), 5 x^{-1} \rbrace} \right)}
\end{equation} 
 with   {$  {\bar K} :=\max\lbrace { K_b,K_{\sigma},K_{\gamma}, K_f, K_{\nu} \rbrace}$.} 
\end{proposition}
\begin{proof}
Let $(X,Y,Z,U) \in \bS_{[0,T]}^2(\bR^d) \times \bS_{[0,T]}^2(\bR) \times \bH_{[0,T]}^2(\bR^d) \times \bH_{[0,T],N}^2(\bR)$ be the unique solution to the FBSDE in \eqref{eq:SDE}-\eqref{eq:BSDE} and $(\widetilde{X}_n, \widetilde{Y}_n, \widetilde{Z}_n, \widetilde{\Psi}_n)_{ n=0,..., M}$ be defined in \eqref{eq:SchemeDeepSolverOneNeuralNetworkRewritten1}.\\
We introduce the processes $(\widetilde{X}_t)_{t \in [0,T]}, (\widetilde{Y}_t)_{t \in [0,T]}$ by 
\begin{align}
    \widetilde{X}_{t}&=\widetilde{X}_n +b(\widetilde{X}_n) (t-t_n)+ \sigma(\widetilde{X}_n)^{\top}(W_t-W_n) + \int_{t_n}^{t}\int_{\mathbb{R}^d} \gamma(\widetilde{X}_n) z \widetilde{N}(\textnormal{d}r,\textnormal{d}z), \nonumber  \\
\widetilde{Y}_t&= \widetilde{Y}_n+ f\left( t_n, \widetilde{X}_n, \widetilde{Y}_n, \widetilde{Z}_n, \widetilde{\Psi}_n \right) (t-t_n) -\widetilde{Z}_n^{\top} (W_t-W_n)  \nonumber \\
	&\quad - \left( \cU_n^{\rho_n}\left(\widetilde{X}_n + \gamma(\widetilde{X}_n)\sum_{i=N([0,t_n], \mathbb{R}^d)+1}^{N([0,t], \mathbb{R}^d)}  z_i \right)- \cU_n^{\rho_n}(\widetilde{X}_n) \right)	 + \int_{t_n}^{t} \widetilde{\Psi}_n \textnormal{d}s,  \nonumber
\end{align}
for $t \in [t_n,t_{n+1})$, $n=0,...,M-1$ and $W_n:=W_{t_n}$, $n=0,...,M-1.$ \\
We define the following differences for $t \in [t_n,t_{n+1})$ 
\begin{align}
 \delta X_t:= X_t - \widetilde{X}_t, \quad \delta Y_t := Y_t - \widetilde{Y_t}, \quad \delta Z_t:=Z_t - \widetilde{Z}_n, \quad \delta \Psi_t := \Psi_t - \widetilde{\Psi}_n, \quad \nonumber
\end{align}
and 
\begin{align*}
\delta b_t &:=b(X_t)-b(\widetilde{X}_n), \quad 
\delta \sigma_t :=\sigma(X_t)-\sigma(\widetilde{X}_n), \quad 
\delta \gamma_t :=\gamma(X_t)-\gamma(\widetilde{X}_n), \\
\delta f_t &:=f(t,X_t,Y_t,Z_t,\Psi_t)-f(t_n,\widetilde{X}_n, \widetilde{Y}_n,\widetilde{Z}_n,\widetilde{\Psi}_n).
\end{align*}
\emph{Step 1:} It holds 
\begin{align}
\mathbb{E}\left[ \big \vert g(\widetilde{X}_M)-\widetilde{Y}_M \big\vert^2 \right] 
& = \mathbb{E} \left[ \big \vert g(\widetilde{X}_M)-g(X_T) + Y_T -\widetilde{Y}_M \big\vert^2\right]\nonumber \\
&\leq \left(1+ K_g^{-1}\right) \mathbb{E} \left[ \big \vert g(\widetilde{X}_M)-g(X_T) \big \vert^2\right]+  \left(1+ K_g\right)\mathbb{E} \left[ \big \vert  Y_T -\widetilde{Y}_M \big \vert^2\right] \nonumber \\
& \leq \left(K_g^2 + K_g\right) \mathbb{E} \left[\big \vert \delta X_{M} \big \vert^2\right]+ \left(1+ K_g\right)\mathbb{E} \left[ \big \vert \delta Y_{M} \big \vert^2\right] \nonumber \\
& \leq \left(1+ K_g\right)^2 \max \left \lbrace \mathbb{E}\left[ \big \vert \delta X_{M} \big \vert^2\right], \mathbb{E} \left[ \big \vert \delta Y_{M}\big  \vert^2\right] \right \rbrace \label{eq:EstimateStep1Max} 
\end{align}
where we use $g(X_T)=Y_T$ and that $g$ is Lipschitz continuous with constant $K_g$. Next, we derive estimates for $\mathbb{E} \left[ \big \vert \delta X_{M}\big  \vert^2\right]$ and $\mathbb{E} \left[ \big \vert \delta Y_{M}\big  \vert^2\right] $. \\
\emph{Step 2: Estimate for $\mathbb{E} \left[ \big \vert \delta X_{n+1}\big  \vert^2\right] $} \\
For $t \in [t_n,t_{n+1})$ we have
\begin{align*}
 \delta X_t= X_t - \widetilde{X}_t=X_n - \widetilde{X}_n + \int_{t_n}^t \delta b_s \textnormal{d}s + \int_{t_n}^t \delta \sigma_s \textnormal{d}W_s + \int_{t_n}^t \int_{\bR^d} \delta \gamma_s z \widetilde{N}(\textnormal{d}s,\textnormal{d}z)
\end{align*}
with $X_{n}:=X_{t_n}.$
By applying It\^{o}'s formula we get 
\begin{align*}
    \textnormal{d}\big \vert\delta X_t \big \vert^2&= 2 \delta X_{t} \textnormal{d}(\delta X_t)+ \big | \delta \sigma_t \big |^2 \textnormal{d}t + \int_{\mathbb{R}^d}\left((\delta X_{t}+ \delta \gamma_t z)^2 - \delta X_{t}^2-2\delta X_{t}\delta \gamma_t z\right) {N}(\textnormal{d}t,\textnormal{d}z)  \nonumber \\
& = \left(2 \delta X_{t} \delta b_t + \big | \delta \sigma_t \big|^2 \right)\textnormal{d}t + 2\delta X_{t} \delta \sigma_t \textnormal{d}W_t +\int_{\mathbb{R}^d}(\delta \gamma_t z)^2  {N}(\textnormal{d}t,\textnormal{d}z) + \int_{\mathbb{R}^d} 2 \delta X_t \delta \gamma_t z \widetilde{N}(\textnormal{d}t,\textnormal{d}z),
\end{align*}
and thus 
\begin{align*}
\bE \left[\big \vert \delta X_t \big\vert^2 \right]= \bE \left[\big \vert \delta X_{n} \big \vert^2 \right]+ \int_{t_n}^t \bE \left[ 2 \delta X_{s} \delta b_s + \big| \delta \sigma_s \big|^2 \right] \textnormal{d}s + \bE \left[ \int_{t_n}^t \int_{\bR^d} (\delta \gamma_s z)^2  {N}(\textnormal{d}s,\textnormal{d}z)  \right],
\end{align*}
with $\delta X_n:=\delta X_{t_n}$ for $n=0,...,M-1.$
For $\lambda_1>0$ it follows 
\begin{align}
 &\bE \left[\big \vert \delta X_t \big\vert^2 \right]\nonumber \\
 &\leq  \bE \left[\big \vert \delta X_{n}\big \vert^2 \right]+ \lambda_1  \int_{t_n}^t \bE \left[ \big|\delta X_{s}  \big|^2 \right] \textnormal{d}s + \lambda_1^{-1}\int_{t_n}^t  \bE \left[ \big|\delta b_{s}  \big|^2 \right] \textnormal{d}s + \int_{t_n}^t  \bE \left[ \big|\delta \sigma_s  \big|^2 \right] \textnormal{d}s \nonumber \\
 & \quad + \bE \left[ \int_{t_n}^t \int_{\bR^d} (\delta \gamma_s z)^2  {N}(\textnormal{d}s,\textnormal{d}z)  \right] \nonumber \\
  &  \leq \bE \left[\big \vert \delta X_{n}\big \vert^2 \right]+ \lambda_1  \int_{t_n}^t \bE \left[ \big \vert \delta X_{s}  \big\vert^2 \right] \textnormal{d}s+ \left(\lambda_1^{-1} (K_b)^2 +  (K_{\sigma})^2\right) \int_{t_n}^t \bE \left[ \big\vert X_s-\widetilde{X}_n\vert^2 \right]\textnormal{d}s  \nonumber \\
 & \quad + (K_{\gamma})^2 \left(\bE \left[ \int_{t_n}^t \int_{\bR^d} \big \vert z \big \vert^2 \big \vert X_s - \widetilde{X}_n \big \vert^2  \widetilde{N}(\textnormal{d}s,\textnormal{d}z)  \right] + \bE \left[ \int_{t_n}^t \int_{\bR^d} \big \vert z \big \vert^2 \big \vert X_s - \widetilde{X}_n \big \vert^2  \nu(\textnormal{d}z) \textnormal{d}s \right]\right)\nonumber \\
  &  \leq  \bE \left[\big \vert \delta X_{n}\big \vert^2 \right]+ \lambda_1  \int_{t_n}^t \bE \left[ \big \vert \delta X_{s}  \big \vert^2 \right] \textnormal{d}s+ \left(\lambda_1^{-1} (K_b)^2 +  (K_{\sigma})^2\right) \int_{t_n}^t \bE \left[ \big\vert X_s-\widetilde{X}_n \big \vert^2\right]\textnormal{d}s  \nonumber \\
 & \quad + (K_{\gamma})^2 \int_{\bR^d} \vert z \vert^2 \nu(\textnormal{d}z) {\int_{t_n}^t \bE \left[ \big \vert X_s - \widetilde{X}_n \big\vert^2   \right] \textnormal{d}s }\label{eq:ContinuousX1}
\end{align}
by using Assumption \ref{assump:ExistenceUniquenessFBSDE} \textbf{(A1)}, the definition of the compensated random  measure $\widetilde{N}$ and the martingale property. Moreover, for $\epsilon_1:=\lambda_2 \left(\lambda_1^{-1} (K_b)^2 + (K_{\sigma})^2+ (K_{\gamma})^2 \int_{\bR^d} \vert z \vert^2 \nu(\textnormal{d}z)\right)^{-1} $ with $\lambda_2>0,$ it holds for $s \in [t_n,t_{n+1})$
\begin{align}
\bE \left[ \big \vert X_s-\widetilde{X}_n\big\vert^2\right] 
&\leq \left(1+\epsilon_1\right) \bE \left[ \big \vert \delta X_{n} \big \vert^2\right] + \left(1+ \epsilon_1^{-1}\right) \bE \left[ \big \vert X_s-X_{n} \big \vert^2 \right]\nonumber \\
& \leq \left(1+\epsilon_1\right) \bE \left[ \big\vert \delta X_{n}\big \vert^2\right]+ \left(1+ \epsilon_1^{-1}\right) C \Delta t, \label{eq:ContinuousX2}
\end{align}
{by applying Theorem} \ref{theorem:Bouchard_Elie_Path_regularity} {in the last inequality}. Combining the inequalities \eqref{eq:ContinuousX1} and \eqref{eq:ContinuousX2} yields
\begin{align}
 &\bE \left[\big \vert \delta X_t \big \vert^2 \right]\nonumber \\
 & \leq  \bE \left[ \big \vert \delta X_{n} \big \vert^2\right] + \lambda_1  \int_{t_n}^t \bE \left[ \big \vert \delta X_{s}  \big \vert^2 \right] \textnormal{d}s \nonumber \\
& \quad + \Delta t \left(\lambda_1^{-1} (K_b)^2 +  (K_{\sigma})^2 + (K_{\gamma})^2 \int_{\bR^d} \big  \vert z \big \vert^2 \nu(\textnormal{d}z)\right)  \left( \left(1+\epsilon_1\right) \bE \left[\big  \vert \delta X_{n} \big \vert^2\right]+ \left(1+ \epsilon_1^{-1}\right)C \Delta t \right) \nonumber \\
& \leq C(\lambda_1,\lambda_2)  (\Delta t )^2 + \lambda_1  \int_{t_n}^t \bE \left[ \big \vert \delta X_{s}  \big \vert^2 \right] \textnormal{d}s + \nonumber \\
& \quad + \left(1+ \left(\lambda_1^{-1} (K_b)^2 +  (K_{\sigma})^2+ (K_{\gamma})^2 \int_{\bR^d} \vert z \vert^2 \nu( \textnormal{d}z )+\lambda_2 \right) \Delta t \right)\bE \left[ \big \vert \delta X_{n} \big \vert^2\right]. \nonumber 
\end{align}
By applying Gr\"{o}nwall's inequality it follows
\begin{align}
&\bE \left[\big \vert \delta X_{n+1}\big \vert^2 \right]\nonumber \\
& \leq e^{\lambda_1 \Delta t} \Bigg( C(\lambda_1,\lambda_2) (\Delta t )^2 + \Big(1+ \Big(\lambda_1^{-1} (K_b)^2 +  (K_{\sigma})^2 + (K_{\gamma})^2 \int_{\bR^d} \big \vert z \big \vert^2 \nu(\textnormal{d}z)+ \lambda_2\Big) \Delta t \Big)\bE \left[ \big \vert \delta X_{n} \big \vert^2\right] \Bigg ) \nonumber \\
& \leq e^{\left(\lambda_1 +\lambda_1^{-1} (K_b)^2 +  (K_{\sigma})^2+ (K_{\gamma})^2 \int_{\bR^d} \vert z \vert^2 \nu(\textnormal{d}z )+\lambda_2\right) \Delta t} \bE \left[ \big \vert \delta X_{n} \big \vert^2\right]  + e^{\lambda_1 \Delta t} C( \lambda_1,\lambda_2)  (\Delta t)^2 \nonumber \\
& \leq e^{\left(\lambda_1 +\lambda_1^{-1} (K_b)^2 +  (K_{\sigma})^2+ (K_{\gamma})^2 \int_{\bR^d}  \vert z  \vert^2 \nu(\textnormal{d}z )+\lambda_2\right) \Delta t } \bE \left[\big  \vert \delta X_{n} \big \vert^2\right]  +  {C}(\lambda_1,\lambda_2)  (\Delta t)^2, \label{eq:EstimateXStep2}
\end{align}
for sufficiently small $\Delta t$.   {Here, the inequality $1+x \leq e^x$ is applied in \eqref{eq:EstimateXStep2}.}\\
\emph{Step 3: Estimate for $\mathbb{E} \left[ \big  \vert \delta  Y_{n+1} \big \vert^2\right]$}\\ 
For $t \in [t_n,t_{n+1})$ it holds 
\begin{align*}
    \delta Y_t&=Y_t-\widetilde{Y}_t \\
    & =Y_{n}-\widetilde{Y}_n+\int_{t_n}^t \delta f_s \textnormal{d}s- \int_{t_n}^t (\delta Z_s)^{\top} \textnormal{d}W_s- \int_{t_n}^t \int_{\bR^d} U_s(z)N(\textnormal{d}s, \textnormal{d}z)+  \int_{t_n}^t \int_{\bR^d}U_s(z)\nu(\textnormal{d}z)\textnormal{d}s  \\
    & \quad + \int_{t_n}^t \int_{\bR^d} \left( \cU_n^{\rho_n}(\widetilde{X}_n + \gamma(\widetilde{X}_n) z )- \cU_n^{\rho_n}(\widetilde{X}_n) \right)N(\textnormal{d}s, \textnormal{d}z)- \int_{t_n}^t \widetilde{\Psi}_n \textnormal{d}s \nonumber \\
    &= Y_{n}-\widetilde{Y}_n+\int_{t_n}^t \delta f_s \textnormal{d}s- \int_{t_n}^t (\delta Z_s)^{\top} \textnormal{d}W_s\\
    & \quad - \int_{t_n}^t  \int_{\bR^d}\left( u(s,X_{s-}+\gamma(X_{s-})z) - u(s,X_{s-}) -\left( \cU_n^{\rho_n}(\widetilde{X}_n +\gamma(\widetilde{X}_n) z )- \cU_n^{\rho_n}(\widetilde{X}_n) \right)\right)\widetilde{N}(\textnormal{d}s, \textnormal{d}z),
\end{align*}
where $u \in C^{1,2}([0,T] \times \bR^d, \bR)$ is as in Theorem \ref{theorem:LinkPIDE}. 
{We define for $t \in [t_n,t_{n+1})$} 
\begin{align*}
    \delta u_t(z):= u(t,X_{t-}+\gamma(X_{t-})z) - u(t,X_{t-}) -\left( \cU_n^{\rho_n}(\widetilde{X}_n +\gamma(\widetilde{X}_n) z )- \cU_n^{\rho_n}(\widetilde{X}_n) \right).
\end{align*}
By applying It\^{o}'s formula we get
\begin{align*}
    \textnormal{d} \big \vert \delta Y_t\big \vert^2 = \left(- 2  \delta Y_t  \delta f_t + \big \vert \delta Z_t \big\vert^2 \right)\textnormal{d}t + 2 \delta Y_t (\delta Z_t)^{\top}\textnormal{d} W_t - \int_{\mathbb{R}^d}2 \delta Y_t \delta u_t(z) \widetilde{N}(\textnormal{d}t,\textnormal{d}z)  +\int_{\mathbb{R}^d} (\delta u_t (z) )^2  N(\textnormal{d}t,\textnormal{d}z) 
\end{align*}
and thus
\begin{align}
\bE \left[ \big \vert \delta Y_t \big \vert^2  \right]= \bE \left[ \big \vert \delta Y_{n} \big \vert^2  \right] + \int_{t_n}^t \bE \left[- 2 \delta Y_s \delta f_s + \big \vert \delta Z_s \big \vert^2\right] \textnormal{d}s + \bE \left[ \int_{t_n}^t \int_{\bR^d} (\delta u_s (z))^2 N(\textnormal{d}s,\textnormal{d}z)\right]. \label{eq:AdaptedLemma3Y1}
\end{align}
For $\lambda_1>0$ it holds
\begin{align}
 &\int_{t_n}^t \bE \left[ -2 \delta Y_s \delta f_s + \big \vert \delta Z_s \big \vert^2\right] \textnormal{d}s \nonumber \\
 & \leq \lambda_1  \int_{t_n}^t \bE \left[ \big \vert \delta Y_s \big \vert^2 \right] \textnormal{d}s + \lambda_1^{-1}  \int_{t_n}^t \bE \left[\big \vert \delta  f_s \big \vert^2 \right] \textnormal{d}s + \int_{t_n}^t  \bE \left[ \big \vert \delta Z_s \big \vert^2 \right] \textnormal{d}s \nonumber \\
    & \leq \lambda_1 \int_{t_n}^t  \bE \left[ \big \vert \delta Y_s \big \vert^2 \right] \textnormal{d}s+ \left( 1 + 5 (K_f)^2 \lambda_1^{-1}\right)\int_{t_n}^t  \bE \left[ \big \vert \delta Z_s \big \vert^2 \right] \textnormal{d}s + 5(K_{f,t})^2 \lambda_1^{-1} (\Delta t)^2  \nonumber \\
  & \quad + 5 (K_f)^2 \lambda_1^{-1}  \int_{t_n}^t \bE \left[ \big \vert X_s- \widetilde{X}_n \big \vert^2 \right] \textnormal{d}s + 5 (K_f)^2 \lambda_1^{-1}  \int_{t_n}^t \bE \left[ \big \vert Y_s- \widetilde{Y}_n \big \vert^2 \right] \textnormal{d}s \nonumber \\
  & \quad +   5 (K_f)^2 \lambda_1^{-1}  \int_{t_n}^t \bE \left[ \big \vert \delta \Psi_s \big \vert^2 \right] \textnormal{d}s \label{eq:AdaptedLemma3Y2}, 
\end{align}
where we use   {Assumptions \ref{assump:ExistenceUniquenessFBSDE} \textbf{(A1)}, \textbf{(A2)}.} 
{For the {third} term in \eqref{eq:AdaptedLemma3Y1} we have for $t \in [t, t_{n+1})$
\begin{align}
    &\bE \left[ \int_{t_n}^t \int_{\bR^d} (\delta u_s (z))^2 N(\textnormal{d}s,\textnormal{d}z)\right] \nonumber \\
    & \leq 4 \bE \left[ \int_{t_n}^t \int_{\bR^d} \left( u(s,X_{s-}+\gamma(X_{s-})z) - u(t_n, \widetilde{X}_n + \gamma(\widetilde{X}_n)z) \right)^2\nu(\textnormal{d}z)\textnormal{d}s\right] \nonumber \\
    & \quad + 4 \bE \left[ \int_{t_n}^t \int_{\bR^d} \left( u(s,X_{s-}) - u(t_n, \widetilde{X}_n) \right)^2 \nu(\textnormal{d}z)\textnormal{d}s\right] \nonumber \\
    & \quad + 4\bE \left[ \int_{t_n}^t \int_{\bR^d}\left(  u(t_n, \widetilde{X}_n + \gamma(\widetilde{X}_n)z) -\cU_n^{\rho_n}(\widetilde{X}_n +\gamma(\widetilde{X}_n)z)\right)^2\nu(\textnormal{d}z)\textnormal{d}s\right] \nonumber \\
    & \quad +4\bE \left[ \int_{t_n}^t \int_{\bR^d}\left(  u(t_n, \widetilde{X}_n) -\cU_n^{\rho_n}(\widetilde{X}_n)\right)^2\nu(\textnormal{d}z)\textnormal{d}s\right] \nonumber \\
    & \leq  4 \bE \left[ \int_{t_n}^t \int_{\bR^d}  C^2 \left( \big \vert s-t_n\big \vert^{\frac{1}{2}} + \big \vert X_{s-}-\widetilde{X}_n  \big \vert + K_{\gamma}\big \vert  X_{s-}-\widetilde{X}_n  \big \vert \big \vert z \big \vert \right)^2 \nu(\textnormal{d}z)\textnormal{d}s\right] \nonumber \\
    & \quad + 4 \bE \left[ \int_{t_n}^t \int_{\bR^d} C^2 \left( \big \vert s-t_n\big \vert^{\frac{1}{2}} + \big \vert  X_{s-}-\widetilde{X}_n  \big \vert  \right)^2 \nu(\textnormal{d}z)\textnormal{d}s\right] \nonumber \\
    & \quad + 4\Delta t \bE \left[   \int_{\bR^d}\left(  u(t_n, \widetilde{X}_n + \gamma(\widetilde{X}_n)z) -\cU_n^{\rho_n}(\widetilde{X}_n +\gamma(\widetilde{X}_n)z)\right)^2\nu(\textnormal{d}z)\right] \nonumber \\
    & \quad +4 \Delta t K_{\nu} \bE \left[\left(  u(t_n, \widetilde{X}_n) -\cU_n^{\rho_n}(\widetilde{X}_n)\right)^2\right]. \label{eq:TermNeuralNetwork1}
\end{align}
Here, we use that $u$ is assumed to be $\frac{1}{2}$-H\"older continuous with respect to $t$ and \eqref{eq:LocallyHoelder}. Moreover, it holds
\begin{align}
 &\bE \left[ \int_{t_n}^t \int_{\bR^d}  C^2 \left( \big \vert s-t_n\big \vert^{\frac{1}{2}} + \big \vert  X_{s-}-\widetilde{X}_n  \big \vert + K_{\gamma}\big \vert  X_{s-}-\widetilde{X}_n  \big \vert \big \vert z \big \vert \right)^2 \nu(\textnormal{d}z)\textnormal{d}s\right]\nonumber \\
   & \leq 2 C^2 \mathbb{E} \left[\int_{t_n}^t \int_{\bR^d}\big \vert s-t_n\big \vert \nu(\textnormal{d}z)\textnormal{d}s \right] + 2 C^2 \int_{t_n}^t \int_{\bR^d}\left(1+ K_{\gamma}^2 \big \vert z \big \vert^2\right)\bE \left [\big \vert  X_s-\widetilde{X}_n  \big \vert^2 \right]\nu(\textnormal{d}z)\textnormal{d}s \nonumber \\
    &= 2C^2 (\Delta t)^2 K_{\nu} + 2 C^2 \left( K_{\nu} + K_{\gamma}^2 \int_{\mathbb{R}^d}  \big \vert z \big \vert^2 \nu(\textnormal{d}z) \right) \int_{t_n}^t\bE \left [\big \vert  X_s-\widetilde{X}_n  \big \vert^2 \right]\textnormal{d}s.
\label{eq:TermNeuralNetwork2}
\end{align}
and 
\begin{align}
\bE \left[ \int_{t_n}^t \int_{\bR^d} C^2 \left( \big \vert s-t_n\big \vert^{\frac{1}{2}} + \big \vert  X_s-\widetilde{X}_n  \big \vert  \right)^2 \nu(\textnormal{d}z)\textnormal{d}s\right] \leq 2C^2 (\Delta t)^2 K_{\nu}+ 2 C^2 K_{\nu}\int_{t_n}^t\bE \left [\big \vert  X_s-\widetilde{X}_n  \big \vert^2 \right]\textnormal{d}s. \label{eq:TermNeuralNetwork3}
\end{align}
By combining \eqref{eq:AdaptedLemma3Y1}-\eqref{eq:TermNeuralNetwork3} we get
\begin{align}
&\bE \left[ \big \vert \delta Y_t \big \vert^2  \right] \nonumber \\
  & \leq \bE \left[ \big \vert \delta Y_{n} \big \vert^2  \right] + \lambda_1 \int_{t_n}^t  \bE \left[ \big \vert \delta Y_s \big \vert^2 \right] \textnormal{d}s+ \left( 1 + 5 (K_f)^2 \lambda_1^{-1}\right)\int_{t_n}^t  \bE \left[ \big \vert \delta Z_s \big \vert^2 \right] \textnormal{d}s \nonumber \\
& \quad +  (\Delta t)^2 \left( 5(K_{f,t})^2 \lambda_1^{-1} + 16C^2 K_{\nu}\right)+ 5 (K_f)^2 \lambda_1^{-1}  \int_{t_n}^t \bE \left[ \big \vert Y_s- \widetilde{Y}_n \big \vert^2 \right] \textnormal{d}s  \nonumber \\
& \quad + \left( 5 (K_f)^2 \lambda_1^{-1} + 8 C^2 \left( 2 K_{\nu} + K_{\gamma}^2 \int_{\mathbb{R}^d}  \big \vert z \big \vert^2 \nu(\textnormal{d}z) \right)\right) \int_{t_n}^t\bE \left [\big \vert  X_s-\widetilde{X}_n  \big \vert^2 \right]\textnormal{d}s \nonumber \\
& \quad + 4\Delta t \Bigg(\bE \left[   \int_{\bR^d}\left(  u(t_n, \widetilde{X}_n + \gamma(\widetilde{X}_n)z) -\cU_n^{\rho_n}(\widetilde{X}_n+\gamma(\widetilde{X}_n)z)\right)^2\nu(\textnormal{d}z)\right]  \nonumber \\
& \quad +  K_{\nu} \bE \left[\left(  u(t_n, \widetilde{X}_n) -\cU_n^{\rho_n}(\widetilde{X}_n)\right)^2\right] \Bigg) + 5 (K_f)^2 \lambda_1^{-1}  \int_{t_n}^t \bE \left[ \big \vert \delta \Psi_s \big \vert^2 \right] \textnormal{d}s. \label{eq:TermNeuralNetwork4}
\end{align}
}
Moreover, let $\epsilon_2:= \left(5 (K_f)^2 \lambda_1^{-1}\right)^{-1} \lambda_2 $ and 
$$
\epsilon_3:=\lambda_2 \left( 5 (K_f)^2 \lambda_1^{-1} + 8 C^2 \left( 2 K_{\nu} + K_{\gamma}^2 \int_{\mathbb{R}^d}  \big \vert z \big \vert^2 \nu(\textnormal{d}z) \right)\right)^{-1}
$$ for $\lambda_2 >0$, {we conclude that} for $s \in [t_n,t_{n+1})$
\begin{align}
\bE \left[ \big \vert Y_s-\widetilde{Y}_n \big \vert^2\right] &\leq (1+\epsilon_2)\bE \left[ \big \vert \delta Y_{n} \big \vert^2\right]+ (1+ \epsilon_2^{-1}) C \Delta t, \label{eq:AdaptedLemma3Y4} \\
\bE \left[ \big \vert X_s-\widetilde{X}_n \big \vert^2\right] &\leq (1+\epsilon_3) \bE \left[\big  \vert \delta X_{n} \big \vert^2\right]+ (1+ \epsilon_3^{-1}) C \Delta t, \label{eq:AdaptedLemma3Y3} 
\end{align}
where \eqref{eq:AdaptedLemma3Y3} and \eqref{eq:AdaptedLemma3Y4} follow by similar arguments as in \eqref{eq:ContinuousX2}.
{Combining \eqref{eq:TermNeuralNetwork4}-\eqref{eq:AdaptedLemma3Y4} yields
\begin{align}
&\bE \left[ \big \vert \delta Y_t \big \vert^2  \right] \nonumber \\
& \leq \lambda_1 \int_{t_n}^t  \bE \left[ \big \vert \delta Y_s \big \vert^2 \right] \textnormal{d}s+ \left( 1 + 5 (K_f)^2 \lambda_1^{-1}\right)\int_{t_n}^t  \bE \left[ \big \vert \delta Z_s \big \vert^2 \right] \textnormal{d}s + 5 (K_f)^2 \lambda_1^{-1}  \int_{t_n}^t \bE \left[ \big \vert \delta \Psi_s \big \vert^2 \right] \textnormal{d}s \nonumber \\
& \quad + \left(1+ \left(5(K_f)^2 \lambda_1^{-1} + \lambda_2\right) \Delta t\right) \bE \left[ \big \vert \delta Y_{n} \big \vert^2\right] \nonumber \\
& \quad + \left( 5 (K_f)^2 \lambda_1^{-1} + 8 C^2 \left( 2 K_{\nu} + K_{\gamma}^2 \int_{\mathbb{R}^d}  \big \vert z \big \vert^2 \nu(\textnormal{d}z) \right)  + \lambda_2 \right) \bE \left[ \big \vert \delta X_{n} \big \vert^2\right] \Delta t + C(\lambda_1,\lambda_2) (\Delta t)^2 \nonumber \\
& \quad + 4\Delta t \textnormal{Error}^{\rho_n}, \nonumber 
\end{align}
with 
\begin{align*}
    \textnormal{Error}^{\rho_n}&:= \bE \left[   \int_{\bR^d}\left(  u(t_n, \widetilde{X}_n + \gamma(\widetilde{X}_n)z) -\cU_n^{\rho_n}(\widetilde{X}_n+\gamma(\widetilde{X}_n)z)\right)^2\nu(\textnormal{d}z)\right] \nonumber \\
    & \quad + K_{\nu} \bE \left[\left(  u(t_n, \widetilde{X}_n) -\cU_n^{\rho_n}(\widetilde{X}_n)\right)^2\right].
\end{align*}
By applying Gr\"onwall's inequality it follows
\begin{align}
   & \bE \left[ \vert \delta Y_{n+1} \vert^2 \right]\nonumber \\
    & \leq e^{\lambda_1 \Delta t}\Bigg(  5 (K_f)^2 \lambda_1^{-1}  \int_{t_n}^t \bE \left[ \big \vert \delta \Psi_s \big \vert^2 \right] \textnormal{d}s +  \left( 1 + 5 (K_f)^2 \lambda_1^{-1}\right)\int_{t_n}^t  \bE \left[ \big \vert \delta Z_s \big \vert^2 \right] \textnormal{d}s+ C(\lambda_1,\lambda_2 )(\Delta t)^2 \nonumber \\
    & \quad  + 4\Delta t \textnormal{Error}^{\rho_n}+ \left( 5 (K_f)^2 \lambda_1^{-1} + 8 C^2 \left( 2 K_{\nu} + K_{\gamma}^2 \int_{\mathbb{R}^d}  \big \vert z \big \vert^2 \nu(\textnormal{d}z) \right)  + \lambda_2 \right) \bE \left[ \big \vert \delta X_{n} \big \vert^2\right] \Delta t \nonumber \\
    & \quad + \left(1+ \left(5(K_f)^2 \lambda_1^{-1} + \lambda_2\right) \Delta t\right) \bE \left[ \big \vert \delta Y_{n} \big \vert^2\right]\Bigg). \nonumber
\end{align}
For sufficiently small $\Delta t$\footnote{Here, we use that for fixed $\overline{d}>0$ there exists $\Delta t$ small enough such that $e^{\lambda_1 \Delta t} \leq 1 + \overline{d}$. Thus, for any $L\geq 0$ it holds $e^{\lambda_1 \Delta t} L \leq (1+ \overline{d})L$ for sufficiently small $\Delta t$ and we choose $\overline{d}$ such that $(1+ \overline{d})L=L+ \lambda_2.$}
\begin{align}
 & \bE \left[ \vert \delta Y_{n+1} \vert^2 \right]\nonumber \\
 & \leq e^{\Delta t(\lambda_1+5(K_f)^2 \lambda_1^{-1}+ \lambda_2)} \bE \left[ \big \vert \delta Y_{n} \big \vert^2\right] + \left(5 (K_f)^2 \lambda_1^{-1} + \lambda_2 \right)\int_{t_n}^t \bE \left[ \big \vert \delta \Psi_s \big \vert^2 \right] \textnormal{d}s \nonumber \\
 & \quad +  \left( 1 + 5 (K_f)^2 \lambda_1^{-1}+ \lambda_2\right)\int_{t_n}^t  \bE \left[ \big \vert \delta Z_s \big \vert^2 \right] \textnormal{d}s + C(\lambda_1, \lambda_2)(\Delta t)^2 + 4\Delta t(1 + \lambda_2) \textnormal{Error}^{\rho_n}\nonumber \\
 & \quad + \left( 5 (K_f)^2 \lambda_1^{-1} + 8 C^2 \left( 2 K_{\nu} + K_{\gamma}^2 \int_{\mathbb{R}^d}  \big \vert z \big \vert^2 \nu(\textnormal{d}z) \right)  +2 \lambda_2 \right) \bE \left[ \big \vert \delta X_{n} \big \vert^2\right] \Delta t. \label{eq:TermNeuralNetwork5}
\end{align}
}
Moreover, for any $\epsilon_4>0$   {it holds}
\begin{align}
\int_{t_n}^{t_{n+1}}  \bE \left[ \big \vert \delta Z_s \big \vert^2 \right] \textnormal{d}s & =  \int_{t_n}^{t_{n+1}}  \bE \left[ \big \vert Z_s - \widetilde{Z}_{n} \pm \widetilde{\widetilde{Z}}_{n} \big \vert^2 \right] \textnormal{d}s \nonumber \\
& \leq \left(1+ \epsilon_4\right) \Delta t \bE \left[ \big \vert  \widetilde{\widetilde{Z}}_{n}- \widetilde{Z}_{n} \big \vert^2 \right] + \left(1+\epsilon_4^{-1}\right)\int_{t_n}^{t_{n+1}}  \bE \left[ \big \vert  Z_s-\widetilde{\widetilde{Z}}_{n} \big \vert^2 \right] \textnormal{d}s, \label{eq:Groenwall1.2}
\end{align}
and by the same arguments it follows for any $\epsilon_5>0 $
\begin{align}
\int_{t_n}^{t_{n+1}}  \bE \left[ \big \vert \delta \Psi_s \big \vert^2 \right] \textnormal{d}s   \leq \left(1+ \epsilon_5\right) \Delta t \bE \left[ \big \vert  \widetilde{\widetilde{\Psi}}_{n}- \widetilde{\Psi}_{n} \big \vert^2 \right] + \left(1+\epsilon_5^{-1}\right)\int_{t_n}^{t_{n+1}}  \bE \left[ \big \vert  \Psi_s-\widetilde{\widetilde{\Psi}}_{n} \big \vert^2 \right] \textnormal{d}s. \label{eq:Groenwall1.3}
\end{align}
Thus, by combining \eqref{eq:TermNeuralNetwork5}-\eqref{eq:Groenwall1.3} we get
\begin{align}
& \bE \left[ \big \vert \delta Y_{n+1} \big \vert^2 \right] \nonumber \\
&\leq e^{\Delta t(\lambda_1+5(K_f)^2 \lambda_1^{-1}+ \lambda_2)} \bE \left[ \big \vert \delta Y_{n} \big \vert^2\right] + C(\lambda_1, \lambda_2)(\Delta t)^2 + 4\Delta t(1 + \lambda_2) \textnormal{Error}^{\rho_n}\nonumber \\
 & \quad + \left(5 (K_f)^2 \lambda_1^{-1} + \lambda_2 \right) \left( \left(1+ \epsilon_5\right) \Delta t \bE \left[ \big \vert  \widetilde{\widetilde{\Psi}}_{n}- \widetilde{\Psi}_{n} \big \vert^2 \right] + \left(1+\epsilon_5^{-1}\right)\int_{t_n}^{t_{n+1}}  \bE \left[ \big \vert  \Psi_s-\widetilde{\widetilde{\Psi}}_{n} \big \vert^2 \right] \textnormal{d}s \right) \nonumber \\
 & \quad +  \left( 1 + 5 (K_f)^2 \lambda_1^{-1}+ \lambda_2\right)\left(\left(1+ \epsilon_4\right) \Delta t \bE \left[ \big \vert  \widetilde{\widetilde{Z}}_{n}- \widetilde{Z}_{n} \big \vert^2 \right] + \left(1+\epsilon_4^{-1}\right)\int_{t_n}^{t_{n+1}}  \bE \left[ \big \vert  Z_s-\widetilde{\widetilde{Z}}_{n} \big \vert^2 \right] \textnormal{d}s\right) \nonumber \\
 & \quad + \left( 5 (K_f)^2 \lambda_1^{-1} + 8 C^2 \left( 2 K_{\nu} + K_{\gamma}^2 \int_{\mathbb{R}^d}  \big \vert z \big \vert^2 \nu(\textnormal{d}z) \right)  +2 \lambda_2 \right) \bE \left[ \big \vert \delta X_{n} \big \vert^2\right] \Delta t. \label{eq:TermNeuralNetwork6}
\end{align}
\emph{Step 4: Combine the estimates from Step 2 and 3}\\
For $n=0,...,M-1,$ we define
$$
N_n:= \max \left\lbrace \bE \left[ \big \vert \delta X_{n} \big \vert^2 \right] , \bE \left[ \big \vert \delta Y_{n} \big \vert^2 \right]\right \rbrace.
$$
Combining \eqref{eq:TermNeuralNetwork6} and \eqref{eq:EstimateXStep2} yields
\begin{align}
&N_{n+1}\nonumber \\
& \leq e^{(\max \lbrace A_1,A_2\rbrace+A_0) \Delta t} N_n   \nonumber \\
     & \quad + \left(1+ 5 (K_f)^2 \lambda_1^{-1}+\lambda_2\right)  \left( \left(1+ \epsilon_4\right) \Delta t \bE \left[ \big \vert  \widetilde{\widetilde{Z}}_{n}- \widetilde{Z}_{n} \big \vert^2 \right] +  \left(1+\epsilon_4^{-1}\right)\int_{t_n}^{t_{n+1}}  \bE \left[ \big \vert  Z_s-\widetilde{\widetilde{Z}}_{n} \big \vert^2 \right] \textnormal{d}s \right)\nonumber  \\
     & \quad  + \left(5 (K_f)^2 \lambda_1^{-1} +\lambda_2\right)\left( \left(1+ \epsilon_5\right) \Delta t \bE \left[ \big \vert  \widetilde{\widetilde{\Psi}}_{n}- \widetilde{\Psi}_{n} \big \vert^2 \right]  +\left(1+\epsilon_5^{-1}\right)\int_{t_n}^{t_{n+1}}  \bE \left[ \big \vert  \Psi_s-\widetilde{\widetilde{\Psi}}_{n} \big \vert^2 \right] \textnormal{d}s\right) \nonumber \\
     & \quad + C(\lambda_1,\lambda_2) (\Delta t)^2 +4\Delta t(1+\lambda_2) \textnormal{Error}^{\rho_n}
\end{align}
with constants 
\begin{align*}
    A_0&:=5 (K_f)^2 \lambda_1^{-1} + 8 C^2 \left( 2 K_{\nu} + K_{\gamma}^2 \int_{\mathbb{R}^d}  \big \vert z \big \vert^2 \nu(\textnormal{d}z) \right)  +2 \lambda_2, \\
    A_1&:=\lambda_1 +\lambda_1^{-1} (K_b)^2 +  (K_{\sigma})^2+ (K_{\gamma})^2 \int_{\bR^d} \vert z \vert^2 \nu(\textnormal{d}z )+\lambda_2,  \\
    A_2&:=\lambda_1 +5(K_f)^2 \lambda_1^{-1} + \lambda_2.
\end{align*}
 Moreover, we set 
 \begin{align*}
      A_3&:=1+5 (K_f)^2 \lambda_1^{-1}+ \lambda_2 , \\
    A_4&:=5 (K_f)^2 \lambda_1^{-1}+\lambda_2, \\
    \overline{A}&:= \max \lbrace A_1,A_2 \rbrace + A_0, \\
    \textnormal{Error}^{\rho}&:=\max_{n=0,...,M-1} \textnormal{Error}^{\rho_n}.
 \end{align*}
By induction and using 
\begin{align*}
    N_0= \mathbb{E} \left[ \big \vert Y_0-\widetilde{Y}_0 \big \vert^2\right]=\big \vert Y_0-y \big \vert^2
\end{align*} it follows
\begin{align}
N_M &\leq A_3  e^{\overline{A}T}  \Delta t \left( \left(1+ \epsilon_4\right) \sum_{n=0}^{M-1}\bE \left[ \big \vert  \widetilde{\widetilde{Z}}_{t_n}- \widetilde{Z}_{n} \big \vert^2 \right] + (1+\epsilon_4^{-1})\sum_{n=0}^{M-1}  \int_{t_n}^{t_{n+1}}    \bE \left[\big \vert  Z_s-\widetilde{\widetilde{Z}}_{n} \big \vert^2 \right] \textnormal{d}s\right)\nonumber \\
& \quad +  A_4 e^{\overline{A} T}  \Delta t \left( (1+ \epsilon_5) \sum_{n=0}^{M-1} \bE \left[ \big \vert  \widetilde{\widetilde{\Psi}}_{n}- \widetilde{\Psi}_{n} \big \vert^2 \right] + (1+\epsilon_5^{-1})\sum_{n=0}^{M-1}  \int_{t_n}^{t_{n+1}}  \bE \left[ \big \vert  \Psi_s-\widetilde{\widetilde{\Psi}}_{n} \big \vert^2 \right] \textnormal{d}s\right)\nonumber \\
& \quad + e^{\overline{A} T}  \big \vert Y_0-y \big \vert^2 +e^{ \overline{A} T}\left(M C(\lambda_1,\lambda_2) (\Delta t)^2  + M4\Delta t(1+\lambda_2) \textnormal{Error}^{\rho} \right)\nonumber \\
&\leq A_3 e^{\overline{A} T} \Delta t\left(  \left(1+ \epsilon_4\right)  \sum_{n=0}^{M-1}\bE \left[ \big \vert  \widetilde{\widetilde{Z}}_{n}- \widetilde{Z}_{n} \big  \vert^2 \right] + C\Delta t  \left(1+\epsilon_4^{-1}\right)  \right) \nonumber \\
& \quad +  A_4 e^{\overline{A} T} \Delta t\left( (1+ \epsilon_5) \sum_{n=0}^{M-1} \bE \left[ \big \vert  \widetilde{\widetilde{\Psi}}_{n}- \widetilde{\Psi}_{n} \big \vert^2 \right] +\left(1+\epsilon_5^{-1}\right) C\Delta t\right) \nonumber \\
& \quad + e^{\overline{A} T} \big \vert Y_0-y \big \vert^2  + e^{\overline{A} T}\left(C(\lambda_1,\lambda_2) \Delta t  + 4 T (1 + \lambda_2) \textnormal{Error}^{\rho} \right), \label{eq:UsePathregularity}
\end{align}
where we use Theorem \ref{theorem:Bouchard_Elie_Path_regularity} in \eqref{eq:UsePathregularity}. By choosing $\epsilon_4, \epsilon_5>0$ as
\begin{equation} \label{eq:PossibleChoiceEpsilon34}
 \epsilon_4:=A_3^{-1}\lambda_2 \quad \text{ and }\quad  \epsilon_5:=  A_4^{-1}\lambda_2,
\end{equation}
it follows with \eqref{eq:UsePathregularity}
\begin{align}
     N_M &\leq (A_3 + \lambda_2) e^{\overline{A} T} \Delta t \sum_{n=0}^{M-1}\bE \left[ \big \vert  \widetilde{\widetilde{Z}}_{n}- \widetilde{Z}_{n} \big \vert^2 \right]  + (A_4+\lambda_2) e^{\overline{A} T} \Delta t \sum_{n=0}^{M-1} \bE \left[ \big \vert  \widetilde{\widetilde{\Psi}}_{n}- \widetilde{\Psi}_{n}\big \vert^2 \right]\nonumber \\ 
      &\quad + e^{\overline{A} T} C \Delta t \left( A_3\frac{\lambda_2+A_3}{\lambda_2}+A_4\frac{\lambda_2+A_4}{\lambda_2} \right) + e^{\overline{A} T} \big \vert Y_0-y \big \vert^2 \nonumber\\
      & \quad + e^{\overline{A} T}\left( C(\lambda_1,\lambda_2) \Delta t  + 4 T  (1 + \lambda_2) \textnormal{Error}^{\rho} \right). \label{eq:EstimateForChoiceEpsilon}
\end{align} 
\emph{Step 5: Combine the results of Step 1 and 4}\\
Combining \eqref{eq:EstimateForChoiceEpsilon} and \eqref{eq:EstimateStep1Max} yields
\begin{align}
&\mathbb{E}\left[ \big  \vert g(\widetilde{X}_M)-\widetilde{Y}_M\big \vert^2 \right] \nonumber \\
&\leq (1+ K_g)^2 \Bigg( (A_3 + \lambda_2) e^{\overline{A} T} \Delta t \sum_{n=0}^{M-1}\bE \left[ \big \vert  \widetilde{\widetilde{Z}}_{n}- \widetilde{Z}_{n} \big \vert^2 \right]  + (A_4+\lambda_2) e^{\overline{A} T} \Delta t \sum_{n=0}^{M-1} \bE \left[\big \vert  \widetilde{\widetilde{\Psi}}_{n}- \widetilde{\Psi}_{n} \big\vert^2 \right]\nonumber \\ 
 &\quad \quad+ e^{\overline{A} T} C\Delta t \left( A_3\frac{\lambda_2+A_3}{\lambda_2}+A_4\frac{\lambda_2+A_4}{\lambda_2} \right) + e^{\overline{A} T} \big \vert Y_0-y \big \vert^2  \nonumber\\
& \quad \quad + e^{\overline{A} T}\left( C(\lambda_1,\lambda_2) \Delta t  + 4 T (1 + \lambda_2) \textnormal{Error}^{\rho} \right)\Bigg) \nonumber \\ 
&\leq (1+ K_g)^2 \Bigg( (A_3 + \lambda_2) e^{\overline{A} T} \Delta t \left(\sum_{n=0}^{M-1}\bE \left[ \big \vert  \widetilde{\widetilde{Z}}_{n}- \widetilde{Z}_{n} \big \vert^2 \right]+ \sum_{n=0}^{M-1} \bE \left[ \big \vert  \widetilde{\widetilde{\Psi}}_{n}- \widetilde{\Psi}_{n} \big \vert^2 \right]\right) \nonumber \\
& \quad \quad+  C(\lambda_1,\lambda_2) \left(\big \vert Y_0-y \big \vert^2 +  \Delta t+ \textnormal{Error}^{\rho}\right)\Bigg). \label{eq:PreliminaryResult1}
\end{align}
 Note that it holds 
\begin{align}
     \overline{A}& = \max \lbrace{A_1,A_2 \rbrace}+A_0 \nonumber \\
    & = \max  \left \{ \lambda_1 +\lambda_1^{-1} (K_b)^2 +  (K_{\sigma})^2+ (K_{\gamma})^2 \int_{\bR^d} \vert z \vert^2 \nu(\textnormal{d}z )+\lambda_2,\lambda_1 +5(K_f)^2 \lambda_1^{-1} + \lambda_2 \right \}\nonumber \\
    & \quad+ 5 (K_f)^2 \lambda_1^{-1} + 8 C^2 \left( 2 K_{\nu} + K_{\gamma}^2 \int_{\mathbb{R}^d}   \vert z  \vert^2 \nu(\textnormal{d}z) \right)  +2 \lambda_2 \nonumber \\
    & \leq \lambda_1+3 \lambda_2+5  {\bar K}^2 \lambda_1^{-1}+8 C^2   {\bar K}  \left( 2 +   {\bar K}  \int_{\mathbb{R}^d}   \vert z  \vert^2 \nu(\textnormal{d}z) \right)\nonumber \\
    & \quad +{\bar K}^2 \max \left \{ \lambda_1^{-1} + 1+ \int_{\bR^d} \vert z \vert^2 \nu(\textnormal{d}z ), 5 \lambda_1^{-1} \right \} .\nonumber
\end{align}
  Moreover, for any given $\lambda_3>0$  we can choose $\lambda_2>0$ small enough such that
\begin{align*}
     \left(A_3 + \lambda_2\right)e^{\overline{A} T} \leq \left(1+\lambda_3\right) \left(1 + 5(K_f)^2 \lambda_1^{-1}\right) e^{T A_5(K_{\textnormal{max},   {\lambda_1}})}
\end{align*}
with 
\begin{align}
A_5(  {\bar K} , \lambda_1)&:=\lambda_1+5 {\bar K}^2 \lambda_1^{-1}+{\bar K}^2 \max \left \{ \lambda_1^{-1} + 1+ \int_{\bR^d} \vert z \vert^2 \nu(\textnormal{d}z ), 5 \lambda_1^{-1}\right\}\nonumber \\
& \quad +8 C^2   {\bar K}  \left( 2 +   {\bar K}  \int_{\mathbb{R}^d}   \vert z  \vert^2 \nu(\textnormal{d}z) \right).
\end{align}
Then it follows 
\begin{align}
    &\mathbb{E}\left[  \big \vert g(\widetilde{X}_M)-\widetilde{Y}_M\big \vert^2 \right] \nonumber \\
&\leq (1+ K_g)^2 (1+\lambda_3) (1 + 5(K_f)^2 \lambda_1^{-1}) e^{T A_5(  {\bar K} , \lambda_1)}  \Delta t  \sum_{n=0}^{M-1}\left(\bE \left[ \big \vert  \widetilde{\widetilde{Z}}_{n}- \widetilde{Z}_{n} \big \vert^2 \right] + \bE \left[ \big \vert  \widetilde{\widetilde{\Psi}}_{n}- \widetilde{\Psi}_{n}\big \vert^2 \right]\right)\nonumber \\
& \quad +(1+K_g)^2 C(\lambda_1,\lambda_3) \left(\big \vert Y_0-y \big \vert^2 +  \Delta t+ \textnormal{Error}^{\rho}\right). \nonumber 
\end{align}
  {By defining} the function 
$$H(\lambda_1):=\left(1+K_g\right)^2\left(1 + 5(K_f\right)^2 \lambda_1^{-1}) e^{T A_5(  {\bar K} , \lambda_1)}$$ and $  {\bar H} :=\min_{  {\lambda_1} \in \mathbb{R}_+} H(  {\lambda_1})$,   {we conclude with $\lambda_1:=\textnormal{argmin}_{x \in \mathbb{R}_+} H(x)$ that}
\begin{align}
 \mathbb{E}\left[  \vert g(\widetilde{X}_M)-\widetilde{Y}_M\vert^2 \right] &\leq   {\bar H} (1+\lambda_3) \Delta t \sum_{n=0}^{M-1}\left(\bE \left[ \big \vert  \widetilde{\widetilde{Z}}_{n}- \widetilde{Z}_{n} \big \vert^2 \right] + \bE \left[ \big \vert  \widetilde{\widetilde{\Psi}}_{n}- \widetilde{\Psi}_{n} \big \vert^2 \right]\right) \nonumber \\
& \quad + (1+K_g)^2C(\lambda_3) \left(\big \vert Y_0-y \big \vert^2 +  \Delta t+ \textnormal{Error}^{\rho}\right).\nonumber 
\end{align}

\end{proof}
We briefly discuss the result in Proposition \ref{LemmaSecondEstimateAuxiliary}. The derived estimate for the a priori estimate depends on the following quantities
\begin{align}\label{eq:DependenceANNs}
\bE \left[ \big \vert  \widetilde{\widetilde{Z}}_{n}- \widetilde{Z}_{n} \big \vert^2 \right],  \quad \bE \left[ \big \vert  \widetilde{\widetilde{\Psi}}_{n}- \widetilde{\Psi}_{n} \big \vert^2 \right], \quad \big \vert Y_0-y \big \vert^2, \quad  \textnormal{Error}^{\rho},
\end{align}
which are all connected to the families of ANNs we use. However, note that the first two terms in \eqref{eq:DependenceANNs} depend on $\widetilde{\widetilde{Z}}_{n}, \widetilde{\widetilde{\Psi}}_{n}$ which cannot be expressed in terms of families of ANNs. Thus, we now try to simplify them and state next an auxiliary result.

\begin{lemma} \label{lemma:SecondEstimateAuxiliary3}
  Let Assumptions \ref{assump:ExistenceUniquenessFBSDE} \textbf{(A1)}, \textbf{(A2)} and \ref{assump:NeuralNetwork} hold. Let $(\widetilde{X}_n)_{n =0,...,M}$ be given by \eqref{eq:SchemeDeepSolverOneNeuralNetworkRewritten1}. For $\Delta t < (K_f)^{-1}$ there exist deterministic functions $\widetilde{U}_n: \mathbb{R}^d \to \mathbb{R}$, $\widetilde{V}_n: \mathbb{R}^d \to \mathbb{R}^d$ and $\widetilde{L}_n: \mathbb{R}^d \to \mathbb{R}^d$ for $n=0,...,M$ such that $\widetilde{Y}_n':=\widetilde{U}_n(\widetilde{X}_n)$, $\widetilde{Z}_n':=\widetilde{V}_n(\widetilde{X}_n)$ and $\widetilde{\Psi}_n':=\widetilde{L}_n(\widetilde{X}_n)$ satisfy 
    \begin{align} \label{eq:SchemeLinkFeynmanKac}
 \begin{cases}
	\widetilde{Y}_M'&=g(\widetilde{X}_M), \\
	\widetilde{Z}_n'&= \frac{1}{\Delta t} \mathbb{E} \left[\widetilde{Y}_{n+1}' \Delta W_n \vert \mathcal{F}_n \right], \\
	\widetilde{\Psi}_n'&= \frac{1}{\Delta t} \mathbb{E} \left[ \widetilde{Y}_{n+1}'\int_{t_n}^{t_{n+1}}\int_{\mathbb{R}^d} \widetilde{N}(\textnormal{d}r,\textnormal{d}z) \vert \mathcal{F}_n \right], \\
	\widetilde{Y}_n'&= \mathbb{E}\left[\widetilde{Y}_{n+1}' \vert \mathcal{F}_n \right]+f\left( t_n, \widetilde{X}_n, \widetilde{Y}_n', \widetilde{Z}_n', \widetilde{\Psi}_n'\right) \Delta t,
	\end{cases}
 \end{align}
 for $n=0,...,M-1$.
\end{lemma}
\begin{proof}
The proof uses similar arguments as the one in Lemma 5 in \cite{Han_Long_2020}. We proceed by backward induction.
For $n=M$ it holds $\widetilde{U}_M(x):=g(x)$ for $x \in \mathbb{R}^d$ and as before w.l.o.g. we assume $\widetilde{Z}'_M=\widetilde{\Psi}'_M:=0.$ \\
We assume that the statements hold for $n+1$. In particular, there exists a function $\widetilde{U}_{n+1}:\mathbb{R}^d \to \mathbb{R}$ such that $\widetilde{Y}_{n+1}'= \widetilde{U}_{n+1}(\widetilde{X}_{n+1}).$ As
 $$
\widetilde{X}_{n+1}=\widetilde{X}_n +b(\widetilde{X}_n) \Delta t+ \sigma(\widetilde{X}_n)^{\top}\Delta W_n + \sum_{i=N([0,t_n], \mathbb{R}^d)+1}^{N([0,t_{n+1}], \mathbb{R}^d)} \Gamma(\widetilde{X}_n, z_i) - \Delta t \int_{\mathbb{R}^d} \Gamma(\widetilde{X}_n,z) \nu (\textnormal{d}z),
$$
there exists a deterministic function $\bar{U}_{n}: \mathbb{R}^d \times \mathbb{R} \times \mathbb{R} \to \mathbb{R}$ such that 
$$
\widetilde{Y}_{n+1}'= \widetilde{U}_{n+1}(\widetilde{X}_{n+1})= \bar{U}_n(\widetilde{X}_n, \Delta W_n, \Delta \widetilde{N}_n)
$$
  { with $\Delta \widetilde{N}_n:=\int_{t_n}^{t_{n+1}}\int_{\mathbb{R}^d} \widetilde{N}(\textnormal{d}r,\textnormal{d}z)$. }
  {As} $\widetilde{Z}_n'$ needs to satisfy \eqref{eq:SchemeLinkFeynmanKac},   {we have}
$$
\widetilde{Z}_n'= \frac{1}{\Delta t} \mathbb{E} \left[ \widetilde{Y}_{n+1}' \Delta W_n \vert \mathcal{F}_n \right]=\frac{1}{\Delta t} \mathbb{E} \left[ \bar{U}_n(\widetilde{X}_n, \Delta W_n, \Delta \widetilde{N}_n) \Delta W_n \vert \mathcal{F}_n \right].
$$
By the independence of $\Delta W_n$, $\Delta \widetilde{N}$ to $\mathcal{F}_n$, there exists a deterministic function $\widetilde{V}_n: \mathbb{R}^d \to \mathbb{R}^d$ with $\widetilde{Z}_n'=\widetilde{V}_{n}'(  {\widetilde{X}_n}).$
Similarly, we get for $\widetilde{\Psi}_n'$
\begin{equation*}
    \widetilde{\Psi}_n'= \frac{1}{\Delta t} \mathbb{E} \left[ \widetilde{Y}_{n+1}'\int_{t_n}^{t_{n+1}}\int_{\mathbb{R}^d} \widetilde{N}(\textnormal{d}r,\textnormal{d}z) \Big \vert \mathcal{F}_n \right]= \frac{1}{\Delta t} \mathbb{E} \left[ \bar{U}_n(\widetilde{X}_n, \Delta W_n, \Delta \widetilde{N}_n) \int_{t_n}^{t_{n+1}}\int_{\mathbb{R}^d} \widetilde{N}(\textnormal{d}r,\textnormal{d}z) \Big \vert \mathcal{F}_n \right],
\end{equation*}
and we conclude as before that there exists a deterministic function $\widetilde{L}_n: \mathbb{R}^d \to \mathbb{R}$ such that  $\widetilde{\Psi}_n'=\widetilde{L}_n(\widetilde{X}_n)$. \\
We set $H_n:=L^2(\Omega, \sigma(\widetilde{X}_n), \mathbb{P})$, which is a Banach space and can be represented as
$$
H_n:=\left \{ Y=\phi(\widetilde{X}_n): \phi \text{ is measurable and } \mathbb{E}\left[\vert Y\vert^2\right] < \infty \right \}. 
$$
We define for $Y \in H_n$ the mapping
$$
\Phi_n(Y):= \mathbb{E}\left[\widetilde{Y}_{n+1}' \vert \mathcal{F}_n \right]+f\left( t_n, \widetilde{X}_n, {Y}, \widetilde{Z}_n', \widetilde{\Psi}_n'\right) \Delta t.
$$
Next, we verify that $\Phi_n(Y)$ is square-integrable. It holds
\begin{align}
&\mathbb{E}\left[ \left \vert \mathbb{E}\left[\widetilde{Y}_{n+1}' \vert \mathcal{F}_n \right]+f\left( t_n, \widetilde{X}_n, {Y}, \widetilde{Z}_n', \widetilde{\Psi}_n'\right) \Delta t\right \vert^2 \right]\nonumber \\
& \leq 2 \mathbb{E} \left[ \mathbb{E}\left[\big\vert \widetilde{Y}_{n+1}'\big \vert  \vert \mathcal{F}_n \right]^2\right]+ 2(\Delta t )^2 \mathbb{E} \left[ \left \vert f\left( t_n, \widetilde{X}_n, {Y}, \widetilde{Z}_n', \widetilde{\Psi}_n'\right)\right \vert^2 \right ] \nonumber \\
& \leq 2 \mathbb{E} \left[ \mathbb{E}\left[\big\vert \widetilde{Y}_{n+1}' \big\vert^2 \vert \mathcal{F}_n \right]\right]+ 2(\Delta t )^2 \mathbb{E} \left[ \left \vert f\left( t_n, \widetilde{X}_n, {Y}, \widetilde{Z}_n', \widetilde{\Psi}_n'\right)\right \vert^2 \right ]\nonumber \\
& =2 \mathbb{E} \left[\big \vert \widetilde{Y}_{n+1}' \big\vert^2 \right]+ 2(\Delta t )^2 \mathbb{E} \left[ \left \vert f\left( t_n, \widetilde{X}_n, {Y}, \widetilde{Z}_n', \widetilde{\Psi}_n'\right)\right \vert^2 \right ] \nonumber \\
&= 2\mathbb{E} \left[ \big \vert \widetilde{Y}_{n+1}' \big\vert^2  \right]+ 4(\Delta t )^2 \mathbb{E} \left[ \left \vert f\left( t_n, \widetilde{X}_n, {Y}, \widetilde{Z}_n', \widetilde{\Psi}_n'\right)-f(0,0,0,0,0)\right \vert^2 \right ] + 4(\Delta t )^2 \vert f(0,0,0,0,0)\vert^2 \nonumber \\
& \leq 2\mathbb{E} \left[\big\vert \widetilde{Y}_{n+1}' \big \vert^2 \right]+ 4(\Delta t )^2 \mathbb{E} \left[   \Big\vert K_{f,t} \vert t_n\vert^{\frac{1}{2}} + K_f \left(\vert \widetilde{X}_n \vert +\vert Y\vert + \vert \widetilde{Z}_n' \vert +\vert \widetilde{\Psi}_n' \vert \right) \Big \vert^2\right ] + 4(\Delta t )^2 C \nonumber \\
&  \leq 2\mathbb{E} \left[\big\vert \widetilde{Y}_{n+1}' \big\vert^2 \right]+ 8(\Delta t )^2 \mathbb{E} \left[   (K_{f,t})^2 \vert t_n\vert + (K_f)^2 \left(\vert \widetilde{X}_n \vert^2 +\vert Y\vert^2 + \vert \widetilde{Z}_n' \vert^2 +\vert \widetilde{\Psi}_n' \vert^2 \right)\right ] + 4(\Delta t )^2 C \nonumber  \\
& < \infty,
\end{align}
as by Lemma \ref{lemma:SchemeAdaptedIntegrable} $\widetilde{X}_n$ is square-integrable. Moreover, $Y \in H_n$ implies the square-integrability of $Y$ and $ \widetilde{\Psi}'_{n}, \widetilde{Z}'_n$ are also in $L^2.$\\ 
By using the representations of $\widetilde{Z}_n'$, $\widetilde{\Psi}_n'$ and $\widetilde{Y}_{n+1}'$, we can express $\Phi_n(Y)$ as a deterministic function of $\widetilde{X}_n$, i.e. $\Phi_n(Y) \in H_n$. Next, we prove that $\Phi_n$ is a contraction. To do so, let $Y_1,Y_2 \in H_n$, then it holds
\begin{align}
&\mathbb{E} \left[\big \vert \Phi_n(Y_1)-\Phi_n(Y_2) \big \vert^2\right]\nonumber \\
&= \mathbb{E} \left[\left \vert \mathbb{E}\left[\widetilde{Y}_{n+1}' \vert \mathcal{F}_n \right]+f\left( t_n, \widetilde{X}_n, {Y}_1, \widetilde{Z}_n', \widetilde{\Psi}_n'\right) \Delta t-\mathbb{E}\left[\widetilde{Y}_{n+1}' \vert \mathcal{F}_n \right]-f\left( t_n, \widetilde{X}_n, {Y}_2, \widetilde{Z}_n', \widetilde{\Psi}_n'\right) \Delta t \right\vert^2\right] \nonumber \\
&= \mathbb{E} \left[\left \vert f\left( t_n, \widetilde{X}_n, {Y}_1, \widetilde{Z}_n', \widetilde{\Psi}_n'\right) \Delta t-f\left( t_n, \widetilde{X}_n, {Y}_2, \widetilde{Z}_n', \widetilde{\Psi}_n'\right) \Delta t \right\vert^2\right] \nonumber \\
& \leq (\Delta t)^2 (K_f)^2\mathbb{E} \left[\big \vert Y_1-Y_2\big\vert^2\right] \nonumber,
\end{align}
by using the Lipschitz continuity of $f$. This allows to conclude that for $\Delta t< (K_f)^{-1}$ the mapping $\Phi_n$ is a contraction on $H_n$ and thus by the Banach fixed-point theorem there exists a unique fixed point $Y^*=\phi_n^*(\widetilde{X}_n)\in H_n$ such that $Y^*=\Phi_n(Y^*).$ By setting $\widetilde{U}_n:=\phi_n^*$ the result is proven.
\end{proof}
We now state our main result for the a priori estimate for the finite activity case. For this let us define the following objects along the lines of \cite{delong}, based on \cite{SOLE2007165}. We consider a univariate L\'evy process with characteristic triplet $(b,\sigma,\nu)$ for $\nu$ a L\'evy measure such that $\int_{\mathbb{R}}|z|^2 \nu(d z)<\infty$, $b\in\mathbb{R}$ and $\sigma \geq 0$. For $A \in \mathscr{B}([0, T]) \otimes \mathscr{B}(\mathbb{R})$ let $A(0)=\{t \in[0, T] ;(t, 0) \in A\}$ and $A^{\prime}=A \backslash A(0)$ we define the finite measure
$$
v(A)=\int_{A(0)} \sigma^2 d t+\int_{A^{\prime}} z^2 v(d z) d t.
$$
We define the martingale-valued random measure $\Upsilon$
$$
\Upsilon(A)=\int_{A(0)} \sigma {\textnormal{d}W_t}+\int_{A^{\prime}} z \widetilde{N}(\textnormal{d} t, \textnormal{d} z), \quad A \in \mathscr{B}([0, T]) \otimes \mathscr{B}(\mathbb{R}^{  {d}}).
$$
We also define
$$
I_n\left(\varphi_n\right)=\int_{([0, T] \times \mathbb{R})^n} \varphi\left(\left(t_1, z_1\right), \ldots\left(t_n, z_n\right)\right) \Upsilon\left(\textnormal{d} t_1, \textnormal{d} z_1\right) \cdot \ldots \cdot \Upsilon\left(\textnormal{d} t_n, \textnormal{d} z_n\right),
$$
and introduce the space $\mathbb{D}^{1,2}(\mathbb{R})$ of square integrable random variables  which are measurable with respect to the natural filtration generated by a L\'evy process and have the representation $$\xi=\sum_{n=0}^{\infty} I_n\left(\varphi_n\right),$$
which is the space of random variables for which the Malliavin derivative is defined as follows: we set $D \xi: \Omega \times$ $[0, T] \times \mathbb{R} \rightarrow \mathbb{R}$ to be the stochastic process of the form
$$
D_{t, z} \xi=\sum_{n=1}^{\infty} n I_{n-1}\left(\varphi_n((t, z), \cdot)\right).
$$
In the approach of \cite{SOLE2007165}, $D_{t, 0} \xi = D_{t} \xi$ is the Malliavin derivative with respect to the (univariate) Wiener process, whereas $D_{t,z}\xi$ represents the derivative with respect to the (univariate) jump component. Notice that in our setting the underlying L\'evy process takes values on $\mathbb{R}^d$: at the cost of heavier notations the construction of \cite{SOLE2007165} has been generalized to the multivariate case, which is relevant for our purposes in order to define $D\xi : \Omega\times [0,T]\times {\mathbb R}^d \to \mathbb R $, in \cite[Section 3.11]{phdNavarro} the construction is performed by combining univariate L\'evy processes as above.

\begin{theorem} \label{theorem:EstimateFinalFiniteActivity}
 Let Assumptions \ref{assump:ExistenceUniquenessFBSDE} \textbf{(A1)}-\textbf{(A3)}, \ref{assump:FiniteActivity} and \ref{assump:NeuralNetwork} hold.  Let $(\widetilde{X}_n,\widetilde{Y}_n,\widetilde{Z}_n, \widetilde{\Psi}_n)_{ n=0,..., M}$ be defined in \eqref{eq:SchemeDeepSolverOneNeuralNetworkRewritten1}. Given $ \lambda_3,\lambda_4>0$ there exists a constant $C(\lambda_3)$ such that for sufficiently small $\Delta t$ it holds
\begin{align}
    & \mathbb{E} \left[ \big\vert \widetilde{Y}_M - g(\widetilde{X}_M)\big\vert^2 \right] \nonumber \\
& \leq   3   {\bar H} \left(1+\lambda_3\right) \left(1+\lambda_{4}\right) \left( 4C\Delta t + \sum_{n=0}^{M-1}\mathbb{E}{ \left[\Big \vert \mathbb{E}\Big[ \widetilde{\widetilde{Z}}_n \vert \widetilde{X}_n\Big] - \widetilde{Z}_n \Big \vert^2 \right]} + \sum_{n=0}^{M-1} {\mathbb{E} \left[\Big \vert \mathbb{E}\Big[ \widetilde{\widetilde{\Psi}}_n \vert \widetilde{X}_n\Big] - \widetilde{\Psi}_n \Big \vert^2 \right]}\right)\nonumber \\
& \quad +  (1+K_g)^2C(\lambda_3) \left({\big \vert Y_0-y \big \vert^2} +  \Delta t+ {\textnormal{Error}^{\rho}}\right), \nonumber
\end{align}
where $\widetilde{\widetilde{Z}}_{n}=: \widetilde{\widetilde{Z}}_{t_n}$ and $ \widetilde{\widetilde{\Psi}}_{n}=: \widetilde{\widetilde{\Psi}}_{t_n}$ are defined in \eqref{eq:ProcessesBouchardElie}. Moreover, $  {\bar H} :=\min_{x \in \mathbb{R}_+} H(x)$ with $H$ given in \eqref{eq:DefinitionFunctionH} and $\textnormal{Error}^{\rho}$ is introduced in \eqref{eq:DefiError}. If 
\begin{align} \label{eq:MallivanDifferentiable}
    \xi:=g(X_T^{t,x})+\int_0^T  f\left(r,X_{r-}^{t,x}, Y_{r-}^{t,x}, Z_r^{t,x}, \int_{\mathbb{R}^d} U_r^{t,x}(z) \nu (\textnormal{d}z) \right)\textnormal{d}r \in \mathbb{D}^{1,2}(\mathbb{R}),
\end{align} then we can substitute $\widetilde{\widetilde{Z}}_n$ and $\widetilde{\widetilde{\Psi}}_n$ by $Z_{t_n}$ and $\Psi_{t_n},$ respectively.
\end{theorem}
\begin{proof}
By Proposition \ref{LemmaSecondEstimateAuxiliary} it holds for $\lambda_3>0$
  \begin{align*}
    \mathbb{E}\left[ \big \vert g(\widetilde{X}_M)-\widetilde{Y}_M \big \vert^2 \right] &\leq   {\bar H} (1+\lambda_3) \Delta t \sum_{n=0}^{M-1}\left(\bE \left[ \big \vert  \widetilde{\widetilde{Z}}_{n}- \widetilde{Z}_{n} \big \vert^2 \right] + \bE \left[ \big \vert  \widetilde{\widetilde{\Psi}}_{n}- \widetilde{\Psi}_{n} \big \vert^2 \right]\right) \nonumber \\
& \quad + (1+K_g)^2C(\lambda_3) \left(\big \vert Y_0-y \big \vert^2 +  \Delta t+ \textnormal{Error}^{\rho}\right).
  \end{align*}
  {In the following}, we aim to   {find estimates for} the terms $\bE \left[ \big \vert  \widetilde{\widetilde{Z}}_{n}- \widetilde{Z}_{n} \big \vert^2 \right]$ and $\bE \left[ \big \vert  \widetilde{\widetilde{\Psi}}_{n}- \widetilde{\Psi}_{n} \big \vert^2 \right].$ \\ 
{To do so, let $(\overline{X}_n, \overline{Y}_n,\overline{Z}_n,\overline{\Psi}_n)_{n=0,...,M}$,$(\widetilde{X}_n, \widetilde{Y}_n, \widetilde{Z}_n, \widetilde{\Psi}_n)_{n =0,...,M}$ be given by the schemes \eqref{eq:SchemeDelong}, \eqref{eq:SchemeDeepSolverOneNeuralNetworkRewritten1}, respectively,} and $\widetilde{\widetilde{Z}}_n$, $\widetilde{\widetilde{\Psi}}_n$ defined in \eqref{eq:ProcessesBouchardElie}. It holds for any $\lambda_4>0$
\begin{align}
    &\bE \left[ \big \vert  \widetilde{\widetilde{Z}}_{n}- \widetilde{Z}_{n} \big \vert^2 \right]\nonumber \\
    & =\bE \left[ \Big \vert  \widetilde{\widetilde{Z}}_{n}- \widetilde{Z}_{n} \pm \overline{Z}_n \pm \mathbb{E}\Big[ \overline{Z}_n \vert \widetilde{X}_n\Big] \pm \mathbb{E}\Big[ \widetilde{\widetilde{Z}}_n \vert \widetilde{X}_n\Big] \Big \vert^2 \right] \nonumber \\
    & \leq \left(1+\lambda_4^{-1}\right) \mathbb{E} \left[ \left \vert \overline{Z}_n -\mathbb{E}\Big[ \overline{Z}_n \vert \widetilde{X}_n\Big] \right \vert^2 \right] \nonumber \\
    & \quad + \left(1+ \lambda_4\right) \mathbb{E} \left[ \Big \vert \Big (\widetilde{\widetilde{Z}}_n - \overline{Z}_n\Big) - \mathbb{E}\Big[ \widetilde{\widetilde{Z}}_n - \overline{Z}_n \vert \widetilde{X}_n\Big]+ \Big(  \mathbb{E}\Big[ \widetilde{\widetilde{Z}}_n \vert \widetilde{X}_n\Big] - \widetilde{Z}_n\Big) \Big \vert^2 \right] \nonumber \\
    & \leq \left(1+\lambda_4^{-1}\right) \mathbb{E} \left[ \Big \vert \overline{Z}_n -\mathbb{E}\Big[ \overline{Z}_n \vert \widetilde{X}_n\Big] \Big \vert^2 \right] \nonumber \\
    & \quad + 3\left(1+ \lambda_4\right) \left(  \mathbb{E} \left[ \Big \vert \widetilde{\widetilde{Z}}_n - \overline{Z}_n \Big \vert^2 \right]+ \mathbb{E} \left[ \Big \vert \mathbb{E}\Big[ \widetilde{\widetilde{Z}}_n - \overline{Z}_n \vert \widetilde{X}_n\Big] \Big \vert^2\right]+ \mathbb{E} \left[\Big \vert \mathbb{E}\Big[ \widetilde{\widetilde{Z}}_n \vert \widetilde{X}_n\Big] - \widetilde{Z}_n \Big \vert^2 \right]\right) \nonumber \\
     & \leq \left(1+\lambda_4^{-1}\right) \mathbb{E} \left[ \left \vert \overline{Z}_n -\mathbb{E}\left[ \overline{Z}_n \vert \widetilde{X}_n\right] \right \vert^2 \right]  + 3\left(1+ \lambda_4\right) \left (  2 \mathbb{E} \left[ \Big \vert \widetilde{\widetilde{Z}}_n - \overline{Z}_n \Big \vert^2 \right] + \mathbb{E} \left[\Big \vert \mathbb{E}\Big[ \widetilde{\widetilde{Z}}_n \vert \widetilde{X}_n\Big] - \widetilde{Z}_n \Big \vert^2 \right]\right ). \label{eq:FinalProof1}
\end{align}
Furthermore, we have 
\begin{align}
\sum_{n=0}^{M-1} \Delta t \mathbb{E} \left[ \Big \vert \widetilde{\widetilde{Z}}_n - \overline{Z}_n \Big \vert^2 \right]&\leq 2 \sum_{n=0}^{M-1} \int_{t_n}^{t_{n+1}} \left(\mathbb{E} \left[ \Big \vert \widetilde{\widetilde{Z}}_n -Z_t \Big \vert^2 \right] + \mathbb{E} \left[ \big \vert Z_t - \overline{Z}_n \big \vert^2 \right]\right) \textnormal{d}t \leq {4} C \Delta {t},  \label{eq:FinalProof1_1}
\end{align}
where we use Inequalities \eqref{eq:ErrorEstimateFBSDEWithJumps2} and \eqref{eq:PathregularityEstimate} from Theorem \ref{theorem:ErrorEstimateSchemeDelong} and \ref{theorem:Bouchard_Elie_Path_regularity}, respectively.\\
Next, we focus on the term $ \mathbb{E} \left[ \Big \vert \overline{Z}_n -\mathbb{E}\left[ \overline{Z}_n \vert \widetilde{X}_n\right] \Big \vert^2 \right]$. 
{First note, that $\widetilde{X}_n=\overline{X}_n$ for all $n=0,...,  {M}$. Moreover, by the definition of $\overline{Z}_n$ and Lemma \ref{lemma:SecondEstimateAuxiliary3} there exists a deterministic function $\widetilde{V}_n$ such that $\overline{Z}_n=\widetilde{V}_n(\widetilde{X}_n)$, i.e. $\overline{Z}_n$ is measurable with respect to the $\sigma$-algebra generated by $\widetilde{X}_n$. Thus, we have
\begin{align} \label{eq:FinalProof1_1.0}
    \mathbb{E} \left[ \left \vert \overline{Z}_n -\mathbb{E}\left[ \overline{Z}_n \vert \widetilde{X}_n\right] \right \vert^2 \right]=\mathbb{E} \left[ \left \vert \widetilde{V}_n(\widetilde{X}_n) -\mathbb{E}\left[\widetilde{V}_n(\widetilde{X}_n) \vert \widetilde{X}_n\right] \right \vert^2 \right]=0.
\end{align}
Similar calculations as in \eqref{eq:FinalProof1}, \eqref{eq:FinalProof1_1} and \eqref{eq:FinalProof1_1.0} can also be done for $\bE \left[ \Big \vert  \widetilde{\widetilde{\Psi}}_{n}- \widetilde{\Psi}_{n} \Big \vert^2 \right]$. 
}
This yields
  \begin{align*}
     & \mathbb{E} \left[ \big\vert \widetilde{Y}_M - g(\widetilde{X}_M)\big\vert^2 \right] \nonumber \\
& \leq   3   {\bar H} \left(1+\lambda_3\right) \left(1+\lambda_{4}\right) \left( 4C\Delta t + \sum_{n=0}^{M-1}\mathbb{E} \left[\Big \vert \mathbb{E}\Big[ \widetilde{\widetilde{Z}}_n \vert \widetilde{X}_n\Big] - \widetilde{Z}_n \Big \vert^2 \right] + \sum_{n=0}^{M-1} \mathbb{E} \left[\Big \vert \mathbb{E}\Big[ \widetilde{\widetilde{\Psi}}_n \vert \widetilde{X}_n\Big] - \widetilde{\Psi}_n \Big \vert^2 \right]\right)\nonumber \\
& \quad +  (1+K_g)^2C(\lambda_3) \left(\big \vert Y_0-y \big \vert^2 +  \Delta t+ \textnormal{Error}^{\rho}\right).
  \end{align*}
  {Note that if the condition in \eqref{eq:MallivanDifferentiable} is satisfied then it follows by the Clark-Ocone formula, see e.g. Theorem 3.5.2 in \cite{delong}, that 
\begin{align}
    Z_s=\mathbb{E}[D_s \xi \vert \mathcal{F}_s] \quad \text{ and } \quad U_s(z)=z\mathbb{E}[D_{s,z} \xi \vert \mathcal{F}_s] \quad z \in \mathbb{R}^d, s \in [0,T],
\end{align}
where $D_s \xi $ and $D_{s,z}\xi $ denote the derivatives of $\xi$ in the Malliavin sense, see \cite{SOLE2007165,delong,phdNavarro}. In particular, we know that the processes $Z_s=(Z_s)_{s \in [0,T]}$ and $U(z)=(U_s(z))_{s \in [0,T]}$ are martingales. Thus, by the definition of $\widetilde{\widetilde{Z}}_n$ and $\widetilde{\widetilde{\Psi}}_n$ it holds
\begin{align*}
    \widetilde{\widetilde{Z}}_n&= \frac{1}{\Delta} \mathbb{E}\left[\int_{t_n}^{t_{n+1}} Z_s \textnormal{d}s \vert \mathcal{F}_n \right]=\frac{1}{\Delta} \int_{t_n}^{t_{n+1}} \mathbb{E}\left[ Z_s  \vert \mathcal{F}_n \right] \textnormal{d}s= \frac{1}{\Delta} \Delta Z_{t_n}= Z_{t_n},
\end{align*}
and
\begin{align*}
    \widetilde{\widetilde{\Psi}}_n&= \frac{1}{\Delta}\mathbb{E}\left[\int_{t_n}^{t_{n+1}} \int_{\mathbb{R^d}}U_s(z) \nu(\textnormal{d}z)\textnormal{d}s \vert \mathcal{F}_n \right]=\frac{1}{\Delta} \int_{t_n}^{t_{n+1}} \int_{\mathbb{R^d}}\mathbb{E}\left[ U_s(z)  \vert \mathcal{F}_n \right] \nu(\textnormal{d}z)\textnormal{d}s \\
    & \frac{1}{\Delta} \int_{t_n}^{t_{n+1}} \int_{\mathbb{R^d}} U_{t_n}(z)   \nu(\textnormal{d}z)\textnormal{d}s= \frac{1}{\Delta} \Delta \Psi_{t_n}= \Psi_{t_n},
\end{align*}
where we applied Fubinis theorem as $(Z,U) \in \mathbb{H}_{[0,T]}^2(\mathbb{R^d}) \times \mathbb{H}_{[0,T],N}^2(\mathbb{R})$.}

\end{proof}

Examples of validity of condition \eqref{eq:MallivanDifferentiable} can be found e.g. in Example 3.2 in \cite{delong} or, in a Brownian setting with driver \cite{bcs2015arxiv}.

\subsection{Infinite activity case} We conclude the paper by stating the generalization of the previous result to the case of infinite activity jump processes.
  {
\begin{theorem}
Let Assumptions \ref{assump:ExistenceUniquenessFBSDE} \textbf{(A1)}-\textbf{(A3)} and \ref{assump:NeuralNetwork} hold true. Let $(X^{\epsilon},Y^{\epsilon},Z^{\epsilon},U^{\epsilon} )$ be the solution of the FBSDE \eqref{eq:SDEApproximation}-\eqref{eq:BSDEApproximation} and let $\Psi^{\epsilon}$ be defined in \eqref{eq:IntegralApproximation}. Let $(\widetilde{X}^{\epsilon}_n,\widetilde{Y}^{\epsilon}_n,\widetilde{Z}^{\epsilon}_n,\widetilde{\Psi}^{\epsilon}_n)_{n=0,...,M}$ be defined by \eqref{eq:SchemeDeepSolverOneNeuralNetwork_eps}.
Given $\lambda_3,\lambda_4>0$ there exists a constant $C^{\epsilon}(\lambda_3)$ such that for sufficiently small $\Delta t$ it holds
\begin{align}
&\mathbb{E} \left[ \big\vert \widetilde{Y}_M^{\epsilon} - g(\widetilde{X}_M^{\epsilon}) \big\vert^2 \right]\nonumber \\
&\leq 3   {\bar H} \left(1+\lambda_3\right) \left(1+\lambda_{4}^{-1}\right) \left( 4C_{\epsilon}\Delta t + \sum_{n=0}^{M-1}\mathbb{E} \left[\left \vert \mathbb{E}\left[ \widetilde{\widetilde{Z}}_n^{\epsilon} \vert \widetilde{X}_n^{\epsilon}\right] - \widetilde{Z}_n^{\epsilon} \right \vert^2 \right] + \sum_{n=0}^{M-1} \mathbb{E} \left[\left \vert \mathbb{E}\left[ \widetilde{\widetilde{\Psi}}_n^{\epsilon} \vert \widetilde{X}_n^{\epsilon}\right] - \widetilde{\Psi}_n^{\epsilon} \right \vert^2 \right]\right)\nonumber \\
& \quad +  (1+K_g)^2C^{\epsilon}(\lambda_3) \left(\mathbb{E} \left[ \big\vert Y_0^{\epsilon}-\widetilde{Y}_0^{\epsilon} \big\vert^2\right] +  \Delta t+ \textnormal{Error}^{\rho,\epsilon}\right) \label{eq:APrioriEstimatesInfiniteActivity}
\end{align}
with 
\begin{align*}
\widetilde{\widetilde{Z_t}}^{\epsilon}:=\frac{1}{\Delta t} \bE \left[ \int_{t_n}^{t_{n+1}} Z_s^{\epsilon} \textnormal{d}s \Big \vert \cF_{n}\right], \quad \widetilde{\widetilde{\Psi}}_t^{\epsilon}:=\frac{1}{\Delta t} \bE \left[ \int_{t_n}^{t_{n+1}} \Psi_s^{\epsilon} \textnormal{d}s \Big \vert \cF_{n}\right]
\end{align*}
and
\begin{align}
\textnormal{Error}^{\rho,\epsilon}&:=\sup_{n=0,...,M-1}\bigg(\bE \left[   \int_{\bR^d}\left(  u^{\epsilon}(t_n, \widetilde{X}_n^{\epsilon} + \gamma(\widetilde{X}_n^{\epsilon})z) -\cU_n^{\epsilon,\rho_n}(\widetilde{X}_n^{\epsilon}+\gamma(\widetilde{X}_n^{\epsilon})z)\right)^2\nu^{\epsilon}(\textnormal{d}z)\right] \nonumber \\
& \quad \quad + K_{\nu^\epsilon} \bE \left[\left(  u^{\epsilon}(t_n, \widetilde{X}_n^{\epsilon}) -\cU_n^{\epsilon,\rho_n}(\widetilde{X}_n^{\epsilon})\right)^2\right]\bigg). 
\end{align}
Here, $C^{\epsilon}$ is a constant depending on the data of the FBSDE in \eqref{eq:SDEApproximation}-\eqref{eq:BSDEApproximation} and the constant $C^{\epsilon}(\lambda_3)$ depends additionally on $\lambda_3$ 
and $  {\bar H} ^{\epsilon}:=\min_{x \in \mathbb{R}_+} H^{\epsilon}(x)$ for 
\begin{equation} \label{eq:DefinitionFunctionHEpsilon}
H^{\epsilon}(x):=(1+K_g)^2\left(1 + 5(K_f)^2 {x}^{-1} \right) e^{T\left( x+5 (  {\bar K} ^{\epsilon})^2 x^{-1}+(  {\bar K} ^{\epsilon})^2 \max \lbrace{ x^{-1} + 1+ \int_{\bR^d} \vert z \vert^2 \nu^{\epsilon}(\textnormal{d}z ), 5 x^{-1} \rbrace} \right)}
\end{equation} 
 with {$  {\bar K} ^{\epsilon}:=\max\lbrace { K_b,K_{\sigma^{\epsilon}},K_{\gamma}, K_f, K_{\nu^{\epsilon}} \rbrace}$.} Moreover, $u^{\epsilon}:[0,T]\times \mathbb{R}^d \to \mathbb{R}$ is assumed to be a $C^{1,2}$-solution of the PIDE
 \begin{align}
 -u^{\epsilon}_t(t, x)-\mathscr{L}^{\epsilon} u^{\epsilon}(t, x) -f\left(t, x, u^{\epsilon}(t, x), D_x u^{\epsilon}(t, x) \sigma_{\epsilon}(x), \mathscr{J}^{\epsilon} u^{\epsilon}(t, x)\right)&=0, \quad \quad(t, x) \in[0, T) \times \mathbb{R}^d, \label{eq:PIDE1Epsilon}\\
 u^{\epsilon}(T, x)&=g(x), \quad x \in \mathbb{R}^d, \label{eq:PIDE2Epsilon}
\end{align}
where 
\begin{align}
\mathscr{L}^{\epsilon} u^{\epsilon}(t, x)= & \left\langle {b}(x), D_x u^{\epsilon}(t, x)\right\rangle+\frac{1}{2}\left\langle\sigma_{\epsilon}(x) D_x^2 u^{\epsilon}(t, x), \sigma_{\epsilon}(x)\right\rangle \nonumber  \\
& +\int_{\mathbb{R}^d}\left(u^{\epsilon}(t, x+\gamma(x) z)-u(t, x)-\left\langle\gamma(x) z, D_x u^{\epsilon}(t, x)\right\rangle\right) \nu^{\epsilon}(d z),  \label{eq:PIDE3Epsilon}\\
\mathscr{J}^{\epsilon} u^{\epsilon}(t, x)= & \int_{\mathbb{R}^d}(u^{\epsilon}(t, x+\gamma(x) z)-u^{\epsilon}(t, x)) \nu^{\epsilon}(d z) , \nonumber 
\end{align}
such that $u^{\epsilon}$ is $\frac{1}{2}$-H\"older-continuous in time and satisfies
$$
\vert u^{\epsilon}(t,x) \vert \leq C \left(1+ \vert x \vert \right), \quad \vert D_x u^{\epsilon}(t,x) \vert \leq C \left(1+ \vert x \vert \right) \quad \text{for } (t,x) \in [0,T] \times \mathbb{R}^d. 
$$
\end{theorem}
\begin{proof}
This follows by an application of Theorem \ref{theorem:EstimateFinalFiniteActivity}.
\end{proof}}

\bibliographystyle{agsm}
\bibliography{sample}

\end{document}